\documentclass[11pt]{article}

\usepackage{fullpage}
\usepackage[round]{natbib}  
\renewcommand{\cite}{\citep}

\usepackage[utf8]{inputenc} 
\usepackage[T1]{fontenc}    
\RequirePackage[colorlinks,citecolor=blue,urlcolor=blue]{hyperref}      
\usepackage{url}            
\usepackage{booktabs}       
\usepackage{amsfonts}       
\usepackage{nicefrac}       
\usepackage{caption}
\usepackage{subcaption}
\usepackage{mathtools}
\usepackage{amsmath}
\usepackage[ruled,linesnumbered, vlined, noend]{algorithm2e}
\usepackage{wrapfig}
\usepackage[textsize=scriptsize]{todonotes}
\usepackage{amsthm, amsfonts, amssymb, enumitem, dsfont, mathabx}
\usepackage{graphicx} 
\usepackage{color, float}
\usepackage{authblk}
\graphicspath{{./figures/}}

\mathtoolsset{showonlyrefs}
\newtheorem{definition}{Definition}
\newtheorem{thm}{Theorem}
\newtheorem{lem}{Lemma}[section]

\newtheorem{prop}{Proposition}
\newcounter{myalgctr}
\numberwithin{myalgctr}{section}

\newcommand{\ceil}[1]{{\lceil #1 \rceil}}

\newcommand{\pab}{p_{AB}}
\newcommand{\R}{\mathbb{R}}
\renewcommand{\P}{\mathbb{P}}
\newcommand{\E}{\mathbb{E}}

\newcommand{\abs}[1]{\left|#1\right|}

\newcommand{\Bern}[1]{\mathrm{Bernoulli}(#1)}
\newcommand{\iid}{\text{i.i.d}}

\newcommand{\Ybar}{\widebar{Y}}
\newcommand{\up}{\mathrm{up}}
\newcommand{\low}{\mathrm{low}}
\renewcommand{\epsilon}{\varepsilon}
\newcommand{\muhat}{\widehat{\mu}}

\newcommand{\ah}{\texttt{A-Hoeffding}\xspace}
\newcommand{\eb}{\texttt{E-Bernstein}\xspace}
\newcommand{\ho}{\texttt{HRMS-Bernstein}\xspace}
\newcommand{\hog}{\texttt{HRMS-Bernstein-GE}\xspace}
\newcommand{\bk}{\texttt{A-Bentkus}\xspace}

\newcommand{\CI}{\mathrm{CI}}
\newcommand{\confseq}{\mathrm{ConfSeq}}
\DeclareMathOperator*{\argmax}{arg\,max}
\newcommand{\Arm}{\mathbb{A}}
\newcommand{\ahat}{\widehat{\alpha}}

\newcommand{\one}{\mathbf{1}}

\newcommand*\samethanks[1][\value{footnote}]{\footnotemark[#1]}
\title{\bf Near-Optimal Confidence Sequences for Bounded Random Variables}
\author[1]{
    Arun Kumar Kuchibhotla\thanks{Equal Contribution.}\thanks{Email: {\tt arunku@stat.cmu.edu}}}
\author[2]{
    Qinqing Zheng\samethanks[1]\thanks{Email: {\tt zhengqinqing@gmail.com}.}}
\affil[1]{Department of Statistics \& Data Science, CMU}
\affil[2]{Facebook AI Research}
\date{}                     

\begin{document}
\maketitle

\begin{abstract}
  Many inference problems, such as sequential decision problems like A/B  testing, adaptive sampling schemes like bandit selection, are often online in nature. The fundamental problem for online inference is to provide a sequence of confidence intervals that are valid uniformly over the growing-into-infinity sample sizes. To address this question, we provide a near-optimal confidence sequence for bounded random variables by utilizing Bentkus' concentration results. We show that it improves on the existing approaches that use the Cram{\'e}r-Chernoff technique such as the Hoeffding, Bernstein, and Bennett inequalities. The resulting confidence sequence is confirmed to be favorable in synthetic coverage problems, adaptive stopping algorithms, and multi-armed bandit problems.
\end{abstract}

\section[Intro]{Introduction}\label{sec:intro}
The abundance of data over the decades has increased the demand for sequential algorithms and inference procedures in statistics and machine learning.   For instance, when the data is too large to fit in a single machine, it is natural to split data into small batches and process one at a time.
Besides, many industry or laboratory data, like user behaviors on a website, patient records, temperature histories, are naturally generated and available in a sequential order. In both scenarios, the collection or processing of new data can be costly, and practitioners often would like to stop data sampling when a required criterion is satisfied. This gives the pressing call for algorithms that minimize the number of sequential samples subject to the prescribed accuracy of the estimator is satisfied.

Many important problems fit into this framework, 
including sequential hypothesis testing problems such as testing positiveness of the
mean~\cite{zhao2016adaptive}, testing equality of distributions and testing independence~\cite{balsubramani2016sequential,yang2017framework}, A/B
testing~\cite{johari2015always,johari2017peeking},  sequential probability ratio
test~\cite{wald2004sequential}, best arm identification for
multi-arm bandits (MAB)~\cite{zhao2016adaptive,yang2017framework}, etc.
All these applications require \emph{confidence sequences} to determine the
number of samples required for a certain guarantee.

A simple example to start from is estimating the mean of a random variable from sequentially 
available data. This is a classic
problem in statistics and widely applied to various applications.
An estimator $\widehat{\mu}$ is said to be $(\varepsilon,\delta)$-accurate for the mean
$\mu$ if $\mathbb{P}(|\widehat{\mu}/\mu - 1| \le \varepsilon) \ge 1 - \delta$
\cite{dagum2000optimal,mnih2008empirical,huber2019optimal}.
This means that the estimator has a relative error of at most $\varepsilon$ with probability
at least $1 - \delta$. In the sequential setting, one important question we would like to answer is
\emph{how many samples are needed to obtain an estimator of the mean that is
    $(\varepsilon, \delta)$ accurate? } \citet{mnih2008empirical}
shows the answer can be derived from a confidence sequence.

\begin{definition}
Let $Y_1, Y_2, \ldots$ be independent real-valued random variables, available sequentially,
with mean $\mu\in\mathbb{R}$. Given $\delta\in[0, 1]$, a $1-\delta$ \emph{confidence sequence} is a
sequence of confidence intervals $\confseq(\delta) = \{\CI_1(\delta), \CI_2(\delta),
\ldots\}$, where $\CI_n(\delta)$ is constructed on-the-fly after observing data sample $(Y_1,\ldots,Y_n)$, such that
\begin{equation}\label{eq:validity-guarantee}
\mathbb{P}\left(\mu~\in~ \mathrm{CI}_n(\delta) \quad\mbox{for all}\quad n\ge 1\right) ~\ge~ 1 - \delta.
\end{equation}
\end{definition}
\vskip-5pt
For the $(\varepsilon,\delta)$-mean estimation problem above, 
suppose one can construct a $1-\delta$ confidence sequence of $\mu$:
$$\confseq(\delta)  = \{\CI_n(\delta) = [\widebar{Y}_n - Q_n, \widebar{Y}_n + Q_n], \quad n \geq 1 \},$$
where $\Ybar_n$ is the empirical mean of the first $n$ samples.  \citet{mnih2008empirical}
shows that with number of samples $N = \min\{n:\,(1 - \varepsilon)\mathrm{UB}_n \le (1 +
\varepsilon)\mathrm{LB}_n\}$, where $\mathrm{UB}_n$ and $\mathrm{LB}_n$ are two simple
functions of the radius of the confidence intervals, the estimator
$\widehat{\mu} = (1/2)\mathrm{sign}(\widebar{Y}_N)[(1 - \varepsilon)\mathrm{UB}_N + (1 +
\varepsilon)\mathrm{LB}_N]$ is $(\epsilon, \delta)$-accurate. See Algorithm~\ref{alg:adaptive-stopping} in Section~\ref{sec:adaptive-stopping} for details.

The need for sequential algorithms has triggered a surge of
interest in developing sharp confidence sequences. 
Unlike the traditional confidence interval in statistics, the
guarantee~\eqref{eq:validity-guarantee} is non-asymptotic and is uniform over the sample
sizes. Ideally, we want $\mathrm{CI}_n(\delta)$ to reduce in width as either $n$ or
$\delta$ increase. Unfortunately, 
guarantee~\eqref{eq:validity-guarantee} is impossible to achieve
non-trivially\footnote{Of course, if we take $\mathrm{CI}_n(\delta) = (-\infty,
    \infty)$, then~\eqref{eq:validity-guarantee} is trivially satisfied.} without
further assumptions~\cite{bahadur1956nonexistence,singh1963existence}. In this paper, we
assume that the random variables are bounded: there exist known constants $L,
U\in\mathbb{R}$ such that $\mathbb{P}(L \le Y_i \le U) = 1$ for all $i\geq 1$, which yields $\mu\in[L, U]$.
Although boundedness can be replaced by tail
assumptions such as sub-Gaussianity or polynomial tails, we will restrict our discussion
to the bounded case in this paper. 

In recent years, several techniques have been proposed to construct confidence sequences~\cite{zhao2016adaptive, mnih2008empirical, howard2018uniform}. These confidence sequences can be thought as a generalization of classical fixed sample size concentration inequalities including Hoeffding, Bernstein, and Bennett. Arguably the simplest construction of a confidence sequence is based on {\it stitching} the fixed sample size concentration inequalities. Other techniques include self-normalization, method of mixtures or pseudo-maximization~\cite{victor2007pseudo,howard2018uniform}. The stitching method (unlike the others) makes use of a union bound (or Bonferroni inequality) which might result in a sub-optimal confidence sequence compared to those obtained from method of mixtures.

To the best of our knowledge, all the existing confidence sequences are built upon
concentration results that bound the moment generating function and follow the
Cram{\'e}r--Chernoff technique. The Cram{\'e}r--Chernoff technique leads to conservative bounds and can be
significantly improved~\cite{philips1995moment}.  In this paper, we leverage the refined
concentration results introduced by \citet{bentkus2002remark}.  We first develop a
``maximal'' version of Bentkus' concentration inequality. Based on it, we construct the
confidence sequence via stitching.  In honor of Bentkus, who pioneered this line of
refined concentration inequalities, we call our confidence sequence as \emph{Bentkus' Confidence Sequence}. The fixed sample size Bentkus concentration inequality is theoretically an improvement of the best possible Cram{\'e}r--Chernoff bound; see Theorem~\ref{thm:initial-maximal-ineq} and the discussion that follows. This improvement implies that stitching the Bentkus concentration inequality improves upon the stitching of the best possible Cram{\'e}r--Chernoff bound. Hence, our confidence sequence is an improvement on the stitched Hoeffding, Bernstein, and Bennett confidence sequences. Although this is an obvious fact, we find in applications that our confidence sequence leads to about 50\% reduction in sample complexity when compared to the classical ones. Surprisingly, we find in simulations that our confidence sequence also improves on the method of mixture confidence sequences that do not use a union bound like stitching. 

To summarize, our major contributions are as follows.
\vspace{-\medskipamount}
\begin{itemize}[leftmargin=*]\itemsep0em
    \item We provide a self-contained introduction to \emph{near-optimal} concentration inequality based on the results of~\citet{bentkus2002remark,bentkus2004hoeffding} and~\citet{pinelis2006binomial}.
    Unlike the Cram{\'e}r--Chernoff bounds, which can be infinitely suboptimal,
    our bound is optimal up to $e^2/2$. In other words, 
    our tail bound is at most $e^2/2$ times the best tail bound that can be obtained under our assumptions. We believe ours is the first application of Bentkus' concentration inequality for confidence sequences and machine learning (ML) applications including the best-arm identification problem. All ML algorithms that use classical concentration inequalities like Hoeffding or Bernstein can be improved substantially, by simply replacing them with the concentration inequalities discussed in this paper.   
    
    \item We use these results in conjunction with a ``stitching'' method~\cite{zhao2016adaptive,mnih2008empirical} to construct non-asymptotic confidence sequences. At sample size $n$, for $\widebar{Y}_n = n^{-1}\sum_{i=1}^n Y_i$, the confidence interval is
    $\mathrm{CI}_n(\delta) := [\widebar{Y}_n - q^{\low}_n(\delta), \widebar{Y}_n + q^{\up}_n(\delta)],$
    for different values $q^{\low}_n(\delta),\,q_n^{\up}(\delta) \ge 0$ and they scale like $\sqrt{\mbox{Var}(Y_1)\log\log(n)/n}$ as $n\to\infty$. 
    
    \item Similar to the Bernstein inequality, Bentkus' method utilizes the variance of $Y_i$'s. Therefore, 
    variance estimation is needed to make the stitched Bentkus confidence sequence actionable in practice.
    We propose a closed form upper bound of the unknown variance based on 
    one-sided concentration for the non-negative variables $(Y_i - \mu)^2$ from \citet{pinelis2016optimal}.
    This one-sided concentration bound is an improvement on the classical Cram{\'e}r--Chernoff bound~\citep[Theorem 2.19]{pena2008self} for non-negative random variables. 
    Once again, this leads to a better upper bound on the unknown variance compared to the ones from~\citet{audibert2009exploration} and~\citet{maurer2009empirical}.
    

    \item  We derive a computable form of the Bentkus' method based on \citet{bentkus2006domination}, and further provide a
    \emph{constant} time algorithm to compute it efficiently (see Appendix~\ref{appsec:computation-of-q-function}). In comparison,
    a brute-force method leads to a linear time complexity. 
    
    \item  We conduct numerical experiments to verify our theoretical claims. Moreover,
    we apply the Bentkus confidence sequence to the $(\epsilon, \delta)$ mean estimation
    problem and the best-arm identification problem. For both problems, our method
    significantly reduces the sample complexity by about $\frac12$
    compared with the other methods.
\end{itemize}

\vspace{-\medskipamount}
The rest of this article is organized as follows.
Section~\ref{sec:related} reviews the related work. Section~\ref{sec:unif-conf-seq}
contains our theoretical results. Section~\ref{sec:experiments} presents the experiments that confirm the
superiority of our method. Section~\ref{sec:summary-future-directions}
summarizes the contributions and discusses some future directions.

\section{Related Work}
\label{sec:related}
Several confidence sequences built on classical concentration inequalities have been proposed and can be applied to bounded
random variables. \citet{zhao2016adaptive} propose confidence sequences through
Hoeffding's inequality, assuming that $Y_i$'s are $\frac12$-sub-Gaussian. For random variables supported on $[L, U]$, this assumption is satisfied after scaling by $\frac{1}{U-L}$. However, this confidence sequence does not scale with the true variance and hence can be conservative. \citet{mnih2008empirical} building on~\citet{audibert2009exploration} construct a confidence sequence through Bernstein's inequality. Due to the nature of Bernstein's inequality, those intervals scale correctly with the true variance. 
The methods in these papers is stitching of fixed sample size concentration inequalities. As mentioned before, they make use of union bound and can have more coverage than required in practice. 
In probability,~\citet{darling1967confidence} and~\citet{victor2007pseudo} (among others) have considered confidence sequences based on martingale techniques and method of mixtures. These methods do not require union bound and can be sharper than the stitched confidence sequences.
More recently,~\citet{howard2018uniform} have unified the techniques of obtaining confidence sequences under a variety of assumptions on random variables. 
This work builds on much of the existing statistics literature and we refer the reader to this paper for a detailed historical account.  

All the confidence sequences in the works mentioned above depend on concentration results that bound the moment generating function and follow the Cram{\'e}r-Chernoff technique. Such concentration results, and consequently the obtained confidence sequences, are conservative and can be significantly improved. To understand the deficiency of such concentration inequalities, consider for example the Bernstein's inequality: for $\widebar{Y}_n = \sum_{i=1}^n Y_i / n$,
\[\textstyle
\mathbb{P}\left(\sqrt{n}(\widebar{Y}_n - \mu) \ge t\right) \le 
e^{-t^2/[2A^2 + (U-L)t/(3\sqrt{n})]},
\]
which scales like $\exp(-t^2/(2A^2))$, for ``small'' $t$.
However, the central limit theorem implies $$\mathbb{P}(\sqrt{n}(\widebar{Y}_n - \mu) \ge t) \approx 1 - \Phi(t/A) \le \frac{e^{-t^2/(2A^2)}}{\sqrt{2\pi((t/A)^2 + 1)}}.$$ See, e.g.,~\citet[Formula 7.1.13]{abramowitz1948handbook}. Therefore, Bernstein's inequality and the true tail differ by the scaling $\sqrt{2\pi(t^2/A^2+1)}$, which can be significant for large $t$. This scaling difference is referred to as the missing factor in~\citet{talagrand1995missing} and~\citet{fan2012missing}. This missing factor does not exist just with Bernstein's inequality but also with the optimal bound that could be derived from the Cram{\'e}r--Chernoff technique; see the discussion surrounding Eq. (1.4) of~\citet{talagrand1995missing}. This explains why a further improvement is possible and~\citet{bentkus2002remark} presents such sharper concentration inequalities. 
Our work essentially builds on the works~\citet{bentkus2002remark,bentkus2004hoeffding,pinelis2006binomial,pinelis2016optimal}, to derive a near-optimal concentration inequality, followed by an improved confidence sequence through the technique of stitching.

Given that Bentkus' concentration inequality is an improvement on the Cram{\'e}r--Chernoff inequalities and that the tightness of the stitched confidence sequence is mainly controlled by the sharpness of the fixed sample size concentration inequality used, our results are not entirely unexpected. Because the improvement we obtain over the existing confidence sequences is significant (Figs. 4-6), we believe this paper will be an important addition to the literature for practical ML applications. 
\section{Bentkus' Confidence Sequences}\label{sec:unif-conf-seq}
For any random variable $Y_i$ with mean $\mu$, $X_i = Y_i - \mu$ is mean zero and hence we will mostly restrict to the case of mean zero random variables. The result for general $\mu$ will readily follow; see Theorem~\ref{thm:empirical-Bentkus-arbitrary-mean}. We first discuss Bentkus' concentration inequality for bounded mean zero random variables. Afterwards, we present a refined confidence sequence through stitching. This confidence sequence is not readily actionable because it depends on the true variance of random variables. To address this, we present a practical version where we replace the true variance by an estimated upper bound. This provides an analog of the empirical Bernstein confidence sequence, and we call our method \emph{Empirical Bentkus Confidence Sequence}.

\paragraph{Assumptions.} Suppose $X_1, X_2, \ldots$ are independent random variables satisfying
\begin{equation}\label{eq:boundedness-conditions}
\mathbb{E}[X_i] = 0,\;\mbox{Var}(X_i) \le A^2_i,\;\mbox{and}\; \mathbb{P}(X_i > B) = 0,
\end{equation}
for all $i\ge1$.
We will first derive concentration inequalities under the one-sided bound assumption as in~\eqref{eq:boundedness-conditions} which only requires $X_i \le B$ almost surely.
To derive actionable versions of the concentration inequalities (with estimated variance), we will impose a two-sided bound assumption.

\subsection{Bentkus' Concentration Inequality for a fixed Sample Size}
We now present a near-optimal concentration inequality for $S_t = \sum_{i=1}^t X_i$ that holds uniformly over all sample sizes $t\le n$. 
The main idea behind the optimality is to replace the exponential function used in the Cram{\'e}r-Chernoff technique with a slowly growing function.
Fix $\alpha \in [0, \infty]$, and set $(a)_+ = \max\{a, 0\}$. It is easy to verify that for all $\nu \in\mathbb{R}$,
\begin{equation}\label{eq:inequality_pinelis}\textstyle
\one\{\nu \geq 0\} ~\leq~ (1 + {\nu}/{\alpha})^{\alpha}_+ ~\le~ e^{\nu}.
\end{equation}
Taking $\nu = \lambda(S_n - u)$ for some $\lambda > 0$ in inequality~\eqref{eq:inequality_pinelis} and applying expectation, we obtain for all $u \in\mathbb{R}$,
\begin{equation}\label{eq:pinelis-implication}\textstyle
    \P(S_n \ge u) 
    \leq \inf_{\lambda \geq 0} \E \left[ ( 1 + {\lambda(S_n - u)}/{\alpha})^{\alpha}_{+} \right].
\end{equation}
The second inequality in~\eqref{eq:inequality_pinelis} readily shows that~\eqref{eq:pinelis-implication} is better than the Cram{\'e}r-Chernoff bound.
Reparameterizing $\lambda = {\alpha}/({u - x})$ with $x \leq u$ in~\eqref{eq:pinelis-implication}, we obtain
\begin{equation}\label{eq:bentkus_prob_bound}
    \P(S_n \ge u) 
    \leq \inf_{x \leq u} \frac{\E[(S_n - x)^\alpha_{+}]}{(u - x)^\alpha_+}, \quad \forall u \in \R.
\end{equation}
Next, we bound $\E[(S_n - x)^\alpha_{+}]$ for all random variables $X_i$'s satisfying~\eqref{eq:boundedness-conditions}. This should be done optimally in order to obtain a near-optimal concentration inequality. Surprisingly, for all $\alpha \ge 2$, $\E[(S_n - x)^\alpha_{+}]$ can be bounded in terms of a ``worst case'' two-point distribution satisfying~\eqref{eq:boundedness-conditions}.

Define independent random variables $G_i, i\ge 1$ as
\begin{equation}\label{eq:G_distribution}
\begin{split}
    \mathbb{P}\left(G_i = -A_i^2/B\right) ~&=~ {B^2}/{(A^2_i + B^2)},\\
    \mathbb{P}\left(G_i = B\right) ~&=~ {A^2_i}/{(A^2_i + B^2)}.
\end{split}
\end{equation}
These random variables satisfy~\eqref{eq:boundedness-conditions} and $G_i$'s are the worst case random variables
satisfying~\eqref{eq:boundedness-conditions}, in the sense that for all $n\ge1, \alpha \ge 2,$ and $x\in\mathbb{R}$, 
\begin{equation}\label{eq:intermediate-optimality}
    \textstyle \mathbb{E}[\left(\sum_{i=1}^n G_i - x\right)_+^\alpha] = \displaystyle \sup_{X_i\sim\eqref{eq:boundedness-conditions}}
    \textstyle \,\mathbb{E}[\left(\sum_{i=1}^n X_i - x\right)_+^\alpha],
\end{equation}
where the supremum is over all distributions of $X_i$'s satisfying~\eqref{eq:boundedness-conditions}. 
We refer the readers to~\citet[Eq. (11)]{bentkus2002remark} and~\citet[Theorem 2.1]{pinelis2006binomial} for the proof of~\eqref{eq:intermediate-optimality}. The definition of the ``worst-case'' distribution of $G_i$'s follows from finding the best quadratic function that upper bounds $t\mapsto(t - x)_+^{\alpha}$; see, e.g.,~\citet[Lemma 8]{burgess2019engineered}. \citet{pinelis2006binomial} proves that~\eqref{eq:intermediate-optimality} holds true when $t\mapsto(t - x)_+^{\alpha}$ is replaced with any function $t\mapsto f(t)$ that has a convex first derivative.

Inequality~\eqref{eq:bentkus_prob_bound} with $\alpha = 2$ and \eqref{eq:intermediate-optimality} show that\footnote{
The function $\alpha\mapsto(1 + \nu / \alpha)^\alpha_+$ increases as $\alpha$ increases,
so using the smallest possible $\alpha$ leads to the best bound. Because
\eqref{eq:intermediate-optimality} only holds for $\alpha \geq 2$,
$\alpha = 2$ is optimal in this context.
} 
\begin{equation}\label{eq:bentkus_goal}
    \P(S_n \ge u) \leq \inf_{x \leq u} \frac{\E[(\sum_{i=1}^n G_n - x)^2_{+}]}{(u - x)^2_+},
\end{equation}
and we find $u$ such that the right hand side of \eqref{eq:bentkus_goal} is upper bounded by $\delta$.
Set $\mathcal{A} = \{A_1, A_2, \ldots\}$ as the collection of standard deviations and for $n\ge1$, define
\begin{equation}\label{eq:definition-P2-tilde}
\tilde{P}_{2,n}(u) ~:=~ \inf_{x \le u}\,\frac{\mathbb{E}[(\sum_{i=1}^n G_i - x)_+^2]}{(u - x)_+^2}.
\end{equation}
For $\delta\in[0,1]$, define $q(\delta; n, \mathcal{A}, B)$ as the solution to the equation $\tilde{P}_{2,n}(u) = \delta.$ In other words,
\begin{equation}\label{eq:quantile-A-B-2nd-version}
q(\delta; n, \mathcal{A}, B) = \tilde{P}_{2,n}^{-1}(\delta)
\end{equation}
The inverse exists uniquely for $\delta \ge \mathbb{P}(\sum_{i=1}^n G_i = nB)$ and is defined to be $nB+1$ if $\delta < \mathbb{P}(\sum_{i=1}^n G_i = nB)$.
The following result provides a refined concentration inequality for $S_n = \sum_{i=1}^n X_i$.
It is a ``maximal'' version of Theorem 2.1 of~\citet{bentkus2006domination}, see Appendix~\ref{appsec:proof-of-theorem-initial-maximal-ineq} for the proof.
\begin{thm}\label{thm:initial-maximal-ineq}
Fix $n\ge1$. If $X_1, X_2, \ldots, X_n$ are independent random variables satisfying~\eqref{eq:boundedness-conditions}, then
\begin{equation}\label{eq:maximal-inequality}
\mathbb{P}\left(\max_{1 \le t \le n}\,S_t \ge q(\delta; n, \mathcal{A}, B)\right) ~\le~ \delta,\;\forall \delta\in[0, 1].
\end{equation}
Further, if $A_1 = \cdots = A_n = A$ and if $\tilde{q}(\cdot; A, B)$ is a function such that $\mathbb{P}(\max_{1\le t\le n}\,S_t \ge n\tilde{q}(\delta^{1/n}; A, B)) \le \delta$ for all $\delta\in[0, 1]$ 
then $q(\delta; n, \mathcal{A}, B) \le n\tilde{q}(\delta^{1/n}; A, B)$.
\end{thm}
\emph{Remark.} The first part of Theorem~\ref{thm:initial-maximal-ineq} provides a finite sample valid estimate of the quantile. The second part implies that it is sharper than classical concentration inequalities such as Hoeffding, Bernstein, Bennett or Prokhorov inequalities. To see this fact, note that $\mathbb{P}(\max_{1\le t\le n}\,S_t \ge n\tilde{q}(\delta^{1/n}; A, B)) \le \delta$ for all $\delta \in[0, 1]$ is equivalent to the existence of a function $H(u; A, B)$ such that $\mathbb{P}(\max_{1\le k\le n}\, S_t \ge nu) \le H^n(u; A, B)$ for all $u$. The classical concentration inequalities mentioned above are all of this product from $H^n(u; A, B)$ for some $H$, hence weaker than our bound.

\subsection{Comparison to Classical bounds}
\begin{figure}[t]
    \centering
    \includegraphics[width=.55\textwidth]{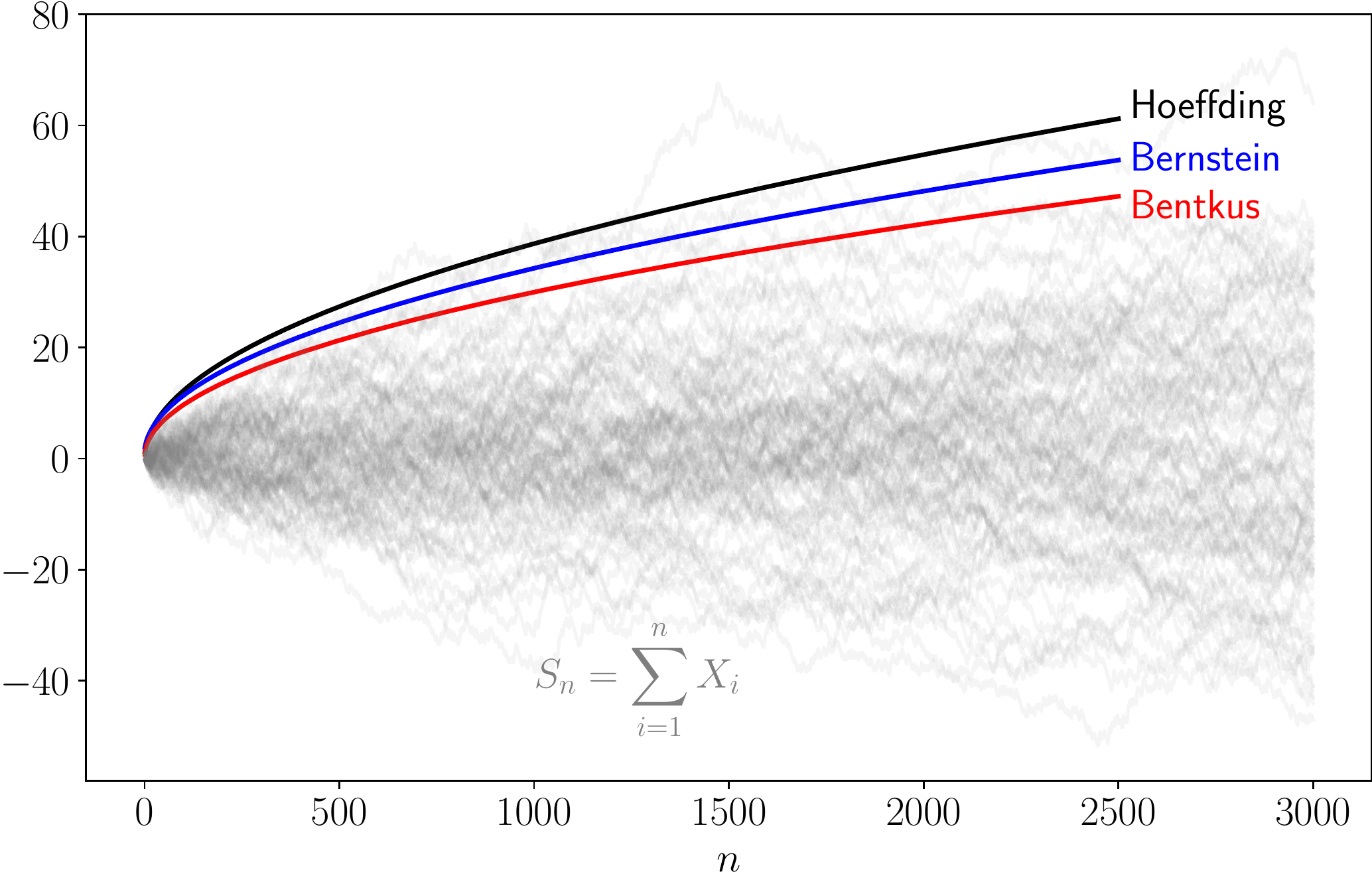}
    \caption{
        Comparison of the concentration bounds when $\delta=0.05$.
        $X_i$ are centered i.i.d. $\Bern{\frac14}$. We give the true standard deviation
        $A_i = \frac{\sqrt{3}}{4}$ and upper bound $B = \frac34$ to all the methods.
        The average failure frequencies across 300
        trials and $1 \leq n\le 3000$ are: Hoeffding $0.00205 \pm 0.00261$, Bernstein
        $0.00593 \pm 0.0044$, Bentkus $0.01411 \pm 0.00769$. The Bentkus’ bound is the least conservative. 
    }
     \label{fig:Bernoulli-1/4}
\end{figure}
Most of the classical concentration inequalities including Hoeffding, Bernstein, Bennett, or Prokhorov inequalities~\cite{bentkus2002remark,wellner2017bennett} are derived based on the Cram{\'e}r--Chernoff technique. The Cram{\'e}r--Chernoff technique makes use of exponential moments unlike the positive part second moment used in Bentkus' concentration inequality. We have mentioned in~\eqref{eq:intermediate-optimality} that random variables $G_i$'s defined in~\eqref{eq:G_distribution} is worst case for the positive part second moment. Interestingly, the same random variables are also worst case for exponential moments too, i.e., for all $\lambda \ge 0$,
\[
   \textstyle \mathbb{E}\left[\exp\left(\lambda\sum_{i=1}^n G_i\right)\right] 
=  \displaystyle \sup_{X_i\sim\eqref{eq:boundedness-conditions}}
   \textstyle\mathbb{E}\left[\exp\left(\lambda\sum_{i=1}^n X_i\right)\right].
\]
See~\citet[Page 42]{bennett1962probability} for a proof. Hence, the optimal Cram{\'e}r--Chernoff concentration inequality is given by
\begin{equation}\label{eq:optimal-cramer}
\textstyle\mathbb{P}(\sum_{i=1}^n X_i \ge u) \le \displaystyle\inf_{\lambda \ge 0}
\frac{\mathbb{E}\left[\exp(\lambda\sum_{i=1}^n G_i)\right]}{\exp(\lambda u)}.
\end{equation}
Furthermore, it can be proved that for all $u\in\mathbb{R}$, 
\begin{equation}\label{eq:comparison-Bentkus-Cramer}
\tilde{P}_{2,n}(u) \le \inf_{\lambda \ge 0}\frac{\mathbb{E}\left[\exp(\lambda\sum_{i=1}^n G_i)\right]}{\exp(\lambda u)},
\end{equation}
see Eqns~\eqref{eq:inequality_pinelis}--\eqref{eq:pinelis-implication},~\eqref{eq:definition-P2-tilde}. This implies that Bentkus' concentration inequality is sharper than the optimal Cram{\'e}r--Chernoff inequality, and hence sharper than Hoeffding, Bernstein, Bennett, and Prokhorov inequalities. 
Inequality~\eqref{eq:comparison-Bentkus-Cramer} only proves that Bentkus' inequality is an improvement but does not show how significant the improvement is. In order to describe the improvement, let us denote the right hand side of~\eqref{eq:comparison-Bentkus-Cramer} as $\tilde{P}_{\infty,n}(u)$. It can be proved that
\begin{equation}\label{eq:ratio-unbounded}
1~\le~ \lim_{n\to\infty}\sup_{u\in\mathbb{R}}\,\frac{\tilde{P}_{\infty,n}(u)}{\mathbb{P}(\sum_{i=1}^n G_i \ge u)} ~=~ \infty.
\end{equation}
See~\citet[Eq. (1.4)]{talagrand1995missing}. Moreover,\footnote{Because $G_i$'s satisfy assumption~\eqref{eq:boundedness-conditions}, inequality $(i)$ is trivial using~\eqref{eq:bentkus_goal}.
Inequality $(ii)$ holds for all $u$ in the support of $\sum_{i=1}^n G_i$; it holds for all $u\in\mathbb{R}$ if $\mathbb{P}(\sum_{i=1}^n G_i \ge u)$ is replaced by its log-linear interpolation; see~\citet{bentkus2002remark} for details. }
\begin{equation}\label{eq:main-optimality}
\hspace{-0.5pt}
\mathbb{P}\big(\sum_{i=1}^n G_i \ge u\big) \stackrel{(i)}{\le}\tilde{P}_{2,n}(u) \stackrel{(ii)}{\le}
\frac{e^2}{2} \mathbb{P}\big(\sum_{i=1}^n G_i \ge u\big).
\end{equation}
Inequalities in~\eqref{eq:main-optimality} show that our concentration inequalities based on the two-point random variables $G_i$ are sharp up to a constant factor $e^2/2$. Further, inequalities \eqref{eq:ratio-unbounded} and \eqref{eq:main-optimality} show that there exists a distribution for which Bentkus' inequality can be infinitely better than the optimal Cram{\'e}r--Chernoff bound.
See Figure~\ref{fig:Bernoulli-1/4} for an illustration and~\citet{bentkus2002remark,bentkus2004hoeffding, pinelis2006binomial} for further discussion.

\subsection{Computation of Bentkus' bound}
\begin{figure}[t]
    \centering
    \includegraphics[width=0.7\textwidth]{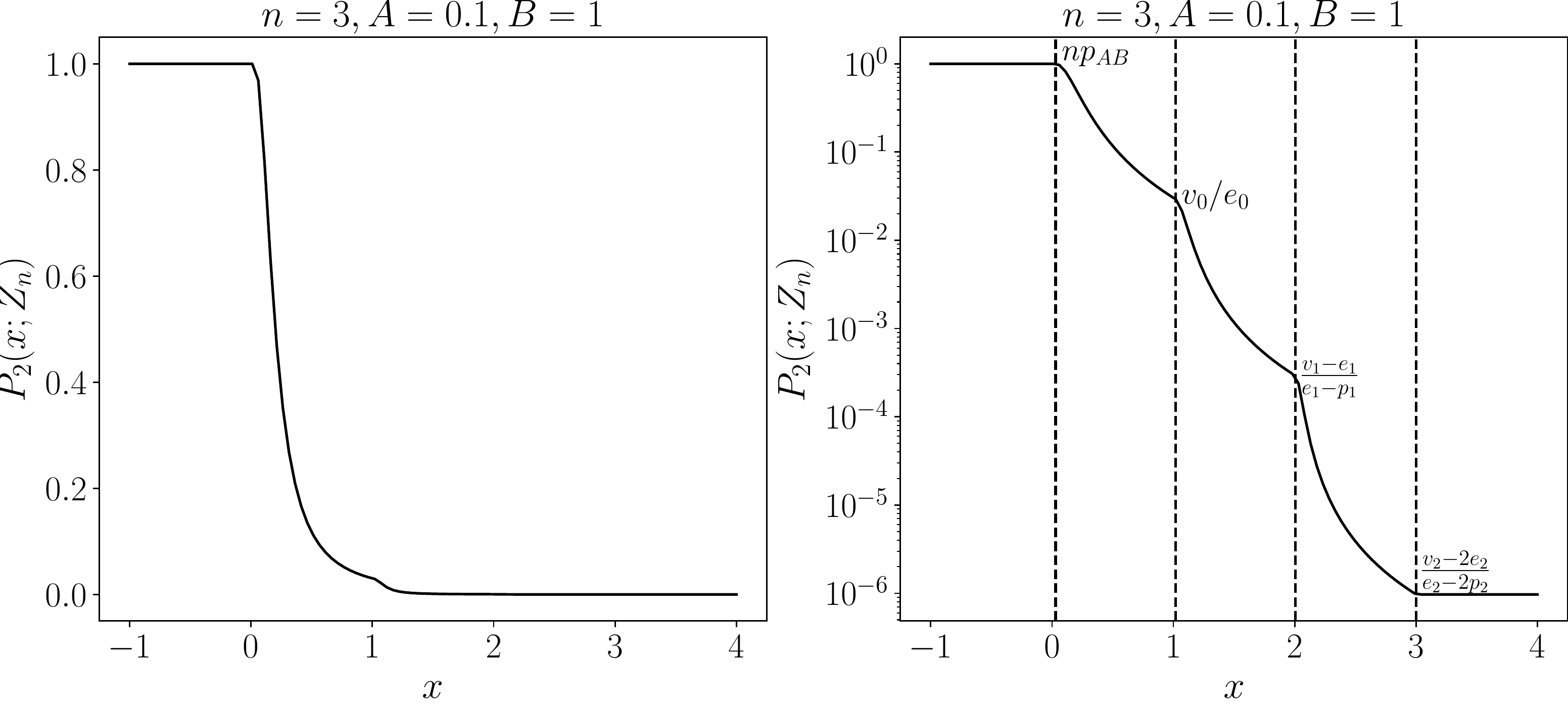}
    \caption{Examples function $P_2(x; Z_n)$ when $n=3$, $A=0.1$ and
        $B=1.0$. We plot $P_2(x;Z_n)$ in both linear (left) and log (right) scales on the y-axis.}
    \label{fig:function_P2}
\end{figure}
Computation of $q(\cdot; n, \mathcal{A}, B)$ is discussed in~\citet[Section 9]{bentkus2006domination} and we provide a detailed discussion in Appendix~\ref{appsec:computation-of-q-function}. In this respect, the following result describes the function in~\eqref{eq:definition-P2-tilde} as a piecewise smooth function in homoscedastic case, i.e., $A_1 = \ldots = A_n = A$.
\begin{prop}\label{prop:computation-P_2}
Set $\pab = A^2/(A^2 + B^2)$ and $Z_n = \sum_{i=1}^nR_i$ where $R_i\sim\Bern{\pab}$. Then for $u\in\mathbb{R}$,
\begin{equation}
    \tilde{P}_{2,n}(u) ~=~ P_2\left(n\pab + u(1 - \pab)/B
    ;\, Z_n \right),
\end{equation}
where $P_2(x; Z_n) = 1$ for $x \le n\pab$ and
\begin{align*}
&P_2\left(x; Z_n\right) :=\\ &
\begin{dcases}
    \frac{n\pab(1 - \pab)}{(x - n\pab)^2 + n\pab(1 - \pab)}, &\mbox{if }n \pab < x \le \Psi_0,\\
    \frac{v_{k}p_{k} - e_{k}^2}{x^2p_{k} - 2xe_{k} + v_{k}}, &\mbox{if }\Psi_{k-1} < x \le \Psi_k,\\
    \mathbb{P}\left(Z_n = n\right) = \pab^n, &\mbox{if }x~\ge~ \Psi_{n-1} = n.
\end{dcases}
\end{align*}
Here $p_k = \mathbb{P}(Z_n \ge k), e_k = \mathbb{E}[Z_n\mathds{1}\{Z_n \ge k\}]$, $v_k = \mathbb{E}[Z_n^2\mathds{1}\{Z_n \ge k\}]$, and $\Psi_k = {(v_k - ke_k)}/{(e_{k} - kp_k)}$.
\end{prop}
The function $P_2(\cdot; Z_n)$ is illustrated in Figure~\ref{fig:function_P2} for $n = 3$ in both linear and logarithmic scale.
Using
Proposition~\ref{prop:computation-P_2} and~\eqref{eq:quantile-A-B-2nd-version}, computation of $q(\cdot; n, A, B)$ follows. In Appendix~\ref{subsec:computation-of-quantile}, we also provide a similar piecewise description of $q(\cdot; n, A, B)$. It is worth pointing out that a similar expression for $P_2(\cdot; Z_n)$ can be derived when $A_i$'s are unequal. 
Proposition~\ref{prop:computation-P_2} is stated for equal variances for simplicity and also because of the widely used i.i.d. assumption.

\subsection{Adaptive Bentkus' Concentration Inequality with Known Variance}
\begin{figure}[t]
    \centering
    \includegraphics[width=.55\textwidth]{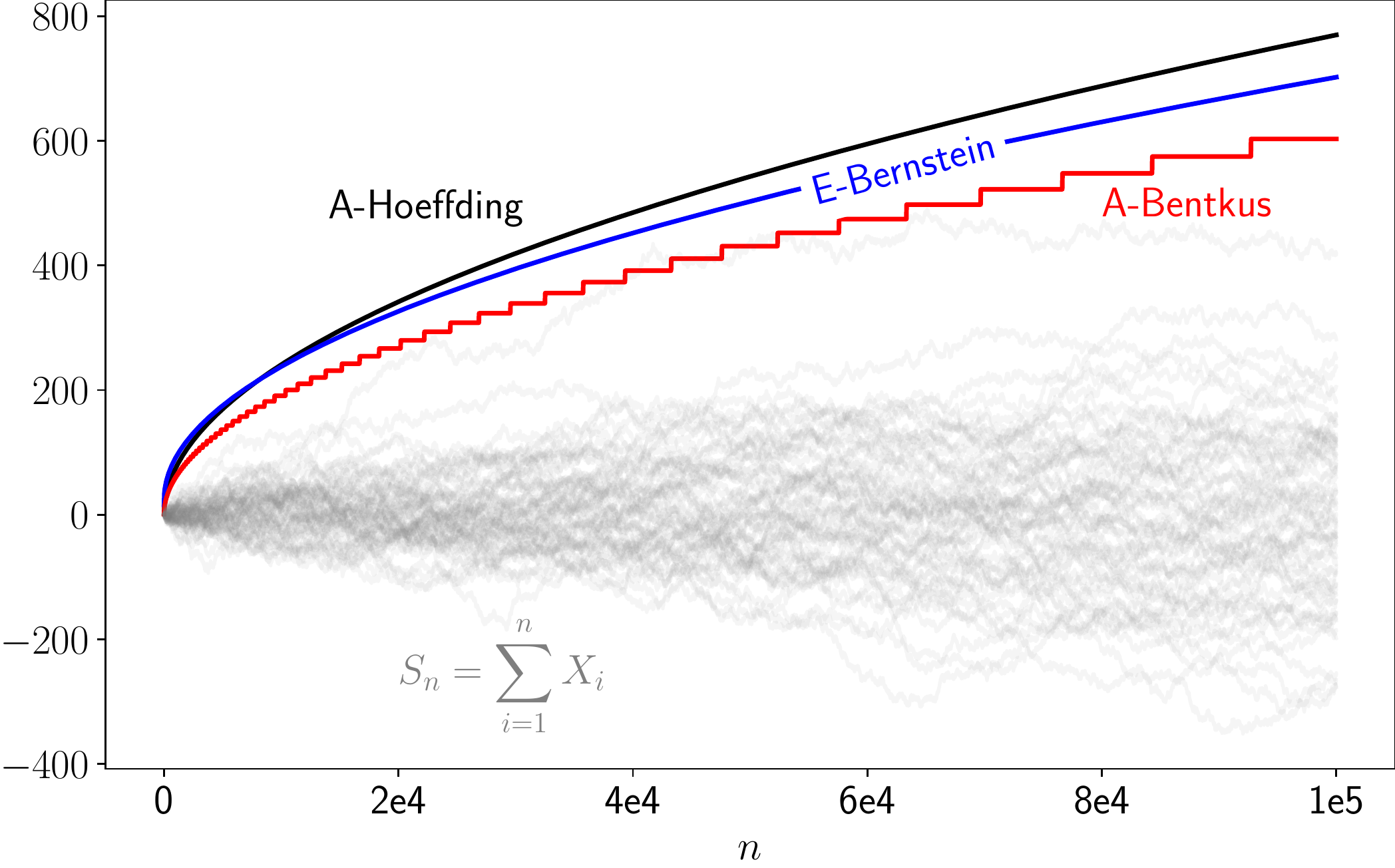}
    \caption{
        Comparison of the uniform concentration bounds when $\delta=0.05$.
        $X_i$ are centered i.i.d. $\Bern{\frac14}$. True standard deviation
        $A_i = {\sqrt{3}}/{4}$ and upper bound $B = 3/4$ are provided to all the methods.
        A-Bentkus is computed using $\eta = 1.1$, $h(k) =(k+1)^{1.1}
        \zeta(1.1)$. For 3000 trials, there is zero failure for Adaptive Hoeffding
        and Empirical Bernstein, but $3$ for A-Bentkus~\eqref{eq:adaptive-Bentkus}.
        All the bounds have failure frequency bounded above by $\delta$ but
        the Bentkus' bound is the least conservative. The differences between the bounds continue to grow as $n$ increases.
    }
    \label{fig:Bernoulli-1/4-unif-n}
\end{figure}

Although Theorem~\ref{thm:initial-maximal-ineq} leads to a uniform in sample size confidence sequence until size $n$, it is very wide for sample sizes much smaller than $n$. We now use the method of stitching to obtain a confidence sequence that is valid for all sample sizes and scales reasonably well with respect to the sample size. See~\citet[Section 3.2]{mnih2008empirical} and~\citet[Section 3.1]{howard2018uniform} for other applications.
The stitching method requires two user-chosen parameters:
\begin{enumerate}[topsep=0pt]
  \item[-] a scalar $\eta > 1$ that determines the geometric spacing.
  \item[-] a function $h:\mathbb{R}_+\to\mathbb{R}_+$ such that $\sum_{k=0}^{\infty} {1}/{h(k)} \le 1$. Ideally, $1/h(k), k\ge0$ adds up to $1$.
\end{enumerate}
The following result
(proved in Appendix~\ref{appsec:proof-of-theorem-uniform-in-n})
gives a uniform over $n$ tail inequality by splitting $\{n\ge1\}$ into $\bigcup_{k\ge1}\{\lceil\eta^k\rceil \le n \le \lfloor\eta^{k+1}\rfloor\}$ and then applying~\eqref{eq:maximal-inequality} within $\{\lceil\eta^k\rceil \le n\le \lfloor\eta^{k+1}\rfloor\}$.
For each $n\ge1$, set $k_n := \min\{k\ge0:\,\lceil\eta^k\rceil \le n \le \lfloor\eta^{k+1}\rfloor\}$ and $c_n := \lfloor\eta^{k_n + 1}\rfloor$.
\begin{thm}\label{thm:uniform-in-n}
If $X_1, X_2, \ldots$ are independent random variables satisfying~\eqref{eq:boundedness-conditions}, then for any $\delta\in[0,1]$,
\begin{equation}\label{eq:adaptive-Bentkus}
\mathbb{P}\left(\exists\, n \ge 1:\,S_n\ge q\left(\frac{\delta}{h(k_n)}; c_n, \mathcal{A}, B\right)\right) \le \delta.
\end{equation}
\end{thm}

The choice of the spacing parameter $\eta$ and stitching function $h(\cdot)$ determine
the shape of the confidence sequence and there is no universally optimal setting. The
growth rate of $h(\cdot)$ determines how the budget of $\delta$ is spent over sample
sizes; a quickly growing $h(\cdot)$ such as $2^k$ yield confidence intervals of
essentially $100\%$ confidence for larger sample sizes. The choice of $\eta$ determines
how conservative the bound is for small $n$ in $\{\lceil\eta^k\rceil \le n\le
\lfloor\eta^{k+1}\rfloor\}$; for $\eta$ too large the bound will be conservative for $n$
close to $\eta^k$. Eq.~\eqref{eq:main-optimality} shows that bound is tightest at $n =
\lfloor\eta^{k+1}\rfloor$ in each epoch. See Appendix~\ref{appsec:hyperparam_stitch} for the graphical
illustration. Throughout this paper, we use
$\eta = 1.1$ and $h(k) = \zeta(1.1)(k+1)^{1.1}$ where $\zeta(\cdot)$ is the Riemann zeta function.

The same stitching method used in Theorem~\ref{thm:uniform-in-n} can also be used with Hoeffding and Bernstein inequalities as done in~\citet{zhao2016adaptive} and~\citet{audibert2009exploration}, respectively. However, given that inequality~\eqref{eq:maximal-inequality} is sharper than Hoeffding and Bernstein inequalities, our bound~\eqref{eq:adaptive-Bentkus} is sharper for the same spacing parameter $\eta$ and stitching function $h(\cdot)$; see Figure~\ref{fig:Bernoulli-1/4-unif-n}. Stitched bounds as in Theorem~\ref{thm:uniform-in-n} are always piecewise constant but the Hoeffding and Bernstein versions from~\citet{zhao2016adaptive} and~\citet{mnih2008empirical} are smooth because they are upper bounds of the piecewise constant boundaries (obtained using $n \le c_n \le \eta n$ and $k_n \le \log_{\eta} n + 1$). For practical use, smoothness is immaterial and the piecewise constant versions are sharper.

\subsection{Adaptive Bentkus Confidence Sequence with Estimated Variance}\label{sec:empirical-Bentkus}
Theorem~\ref{thm:uniform-in-n} is impractical in its form because it involves the unknown sequence of $A_1, A_2, \ldots$. In the case where $A_1 = A_2 = \cdots = A$, one needs to generate an upper bound of $A$ (for a known $B$) and obtain an actionable version of Theorem~\ref{thm:uniform-in-n}. Finite-sample over-estimation of $A$ requires a two-sided bound on the $X_i$'s; one-sided bounds on the random variables do not suffice. This actionable version is a refined version of empirical Bernstein inequality that is uniform over the sample sizes.

We will assume that $\mathbb{P}(\underline{B} \le X_i \le B) = 1,\,\forall\,i$. It follows that $\mbox{Var}(X_i) = A^2 \le -B\underline{B}$~\cite{bhatia2000better}. Because $X_i$'s have mean zero, $\underline{B} \le 0$ and $B \ge 0$; this implies that $-\underline{B}B \ge 0$. If one wants to avoid variance estimation, then one can use this upper bound in Theorem~\ref{thm:uniform-in-n} to obtain an actionable confidence sequence. This sequence, however, will not have width scaling with the true variance. 

Define $\widebar{A}_1(\delta) = (B - \underline{B})/2$ and for $n\ge2, \delta\in[0, 1]$
\begin{equation}\label{eq:variance-over-estimation}
\begin{split}
\widehat{A}_n^2 &:= \textstyle\lfloor n/2\rfloor^{-1}\sum_{i=1}^{\lfloor n/2\rfloor} {(X_{2i} - X_{2i - 1})^2}/{2},\\ 
\widebar{A}_n(\delta) &:= \sqrt{\widehat{A}_n^2 + g_{2,n}^2(\delta)} + {g_{2,n}(\delta)},
\end{split}
\end{equation}
where $g_{2,n}(\delta) := (2\sqrt{2}n)^{-1}\sqrt{\lfloor c_n/2\rfloor}(B - \underline{B})\times\Phi^{-1}\left(1 - {2\delta}/{(e^2h(k_n))}\right)$, for the distribution function $\Phi(\cdot)$ of a standard normal random variable.
We will write $\widebar{A}_n(\delta; B, \underline{B})$, when needed, to stress the dependence of $\widebar{A}_n(\delta)$ on $B, \underline{B}$.
Lemma~\ref{lem:uniform-in-n-variance} shows that $\widebar{A}_n(\delta)$ is a valid over-estimate of $A$ uniformly over $n$ and yields the following actionable bound. We defer the proof to Appendix~\ref{appsec:proof-of-lemma-variance-estimation}.

\begin{thm}\label{thm:uniform-in-n-empirical-Bentkus}
If $X_1, X_2, \ldots$ are mean-zero independent random variables satisfying $\mathrm{Var}(X_i) = A^2$ and $\mathbb{P}(\underline{B} \le X_i \le B) = 1$ for all $i\ge1$, then for any $\delta_1, \delta_2\in[0, 1]$, with probability at least $1 - \delta_1 - \delta_2$, simultaneously for all $n\ge1$,
\[
S_n \le q\left(\frac{\delta_1}{h(k_n)}; c_n, \widebar{A}_n^*(\delta_2), B\right)
\mbox{ and } A \le \widebar{A}_n^*(\delta_2, B, \underline{B}).
\]
Similarly, with probability at least $1-\delta_1 - \delta_2$, simultaneously for all $n\ge1$,
\[
S_n \ge -q\left(\frac{\delta_1}{h(k_n)}; c_n, \widebar{A}_n^*(\delta_2), -\underline{B}\right)
\mbox{ and } A \le \widebar{A}_n^*(\delta_2, B, \underline{B}).
\]
Here $\widebar{A}_n^*(\delta_2) := \min_{1\le s\le n}\,\widebar{A}_n(\delta_2, B, \underline{B})$, and $k_n, c_n$ are those defined before Theorem~\ref{thm:uniform-in-n}.
\end{thm}
Theorem~\ref{thm:uniform-in-n-empirical-Bentkus} is an analogue of the empirical Bernstein inequality~\citet[Eq. (5)]{mnih2008empirical}. The over-estimate of $A$ in~\eqref{eq:variance-over-estimation} can be improved by using non-analytic expressions, but we present the version above for simplicity; see Appendix~\ref{appsec:proof-of-lemma-variance-estimation} for details on how to improve $\widebar{A}_n(\delta)$ in~\eqref{eq:variance-over-estimation}.

Theorem~\ref{thm:uniform-in-n-empirical-Bentkus} can be used to construct a confidence
sequence as follows. Suppose $Y_1, Y_2, \ldots$ are independent random variables with
mean $\mu$, variance $A^2$, and satisfying $\mathbb{P}(L \le Y_i \le U) = 1$.
Then $X_i = Y_i - \mu$ is a zero mean random variable where $\P(L - \mu \leq X_i \leq U
- \mu) = 1$, and Theorem~\ref{thm:uniform-in-n-empirical-Bentkus} is directly applicable
with $B = -\underline{B} = U - L$.  An interesting observation is that we can refine the values of $\underline{B}$ and $B$
while we are updating the confidence interval for $\mu$. Suppose 
with $n$ data points, we have:
$-q_n^\low \leq n\Ybar_n - n \mu  \leq q_n^\up$,
then
\[\textstyle
    \mu^\low_n := \Ybar_n  - n^{-1}{q_n^\up}  ~\leq~  \mu  ~\leq~ \Ybar_n  +
    n^{-1}{q_n^\low} =: \mu_n^\up,
\]
where $\Ybar_n$ is the empirical mean of $Y$. We thus have a valid estimate $[L - \mu_n^{\up}, U - \mu_n^{\low}]$ of the support of $X$,
and when we observe $Y_{n+1}$, we can use
$U - \mu_n^\low$ as $B$ and $L - \mu_n^\up$ as $\underline{B}$.
Importantly, as Theorem~\ref{thm:uniform-in-n-empirical-Bentkus} provides
a uniform concentration bound,
these recursively defined upper and lower bounds hold simultaneously too. This
leads to the following result,
proved in Appendix~\ref{appsec:proof-empirical-bentkus-arbitrary-mean}.

\begin{thm}\label{thm:empirical-Bentkus-arbitrary-mean}
If random variables $Y_1, Y_2, \ldots$ are independent with mean $\mu$, variance $A^2$
and satisfy $\mathbb{P}(L \le Y_i \le U) = 1$.
Define $\mu_0^\up := U$, $\mu_0^\low := L$, and for $n \geq 1$
\begin{equation}
    \begin{aligned}
        \mu_n^\up &= \Ybar_n + \frac{1}{n} q \left(\frac{\delta_1}{2h(k_n)}; c_n, \widebar{A}^*_n(\delta_2, U , L ),\, \mu_{n-1}^\up-L \right) \\
        \mu_n^\low  &= \Ybar_n -\frac{1}{n} q\left(\frac{\delta_1}{2h(k_n)}; c_n, \widebar{A}^*_n(\delta_2, U, L),\,  U - \mu_{n-1}^\low  \right)
    \end{aligned}
\end{equation}
Let $\mu_n^{\up*} = \min_{0\leq i \leq n} \mu_i^\up$ and $\mu_n^{\low*} = \max_{0\leq i \leq n} \mu_i^\low$.
Then for any $\delta_1, \delta_2\in[0, 1]$, with probability at least $1 - \delta_1 - \delta_2$, simultaneously for all $n\ge1$,
\begin{equation}\label{eq:empirical-bentkus-conf-seq}
    \mu \in[\mu_n^{\low*}, \mu_n^{\up*} ]\quad \mbox{and} \quad A \le
        \widebar{A}_n^*(\delta_2, U, L).
\end{equation}
\end{thm}
Because $\mu_0^{\up} = U, \mu_0^{\low} = L$, the confidence intervals $[\mu_n^{\low*}, \mu_n^{\up*}]$ is always a subset of $[L, U]$.

\section{Experiments}\label{sec:experiments}
We compare our adaptive Bentkus confidence
sequence~\eqref{eq:empirical-bentkus-conf-seq} with the adaptive Hoeffding
~\cite{zhao2016adaptive}, empirical Bernstein~\cite{mnih2008empirical}, and two other
versions of empirical Bernstein inequality from~\cite{howard2018uniform}: Eq. (24) and
Theorem 4 with the gamma-exponential boundary. Eq. (24) of~\citet{howard2018uniform} is a stitched confidence sequence, while Theorem 4 is a method of mixture confidence sequence.
We denote these methods by \bk, \ah, \eb, \ho, and \hog,
respectively.   We use $\delta = 0.05$ for all the experiments. For \bk, we fix the
spacing parameter $\eta = 1.1$, the stitching function $h(k) = (k+1)^{1.1}\zeta(1.1)$,
and $\delta_1 = 2\delta/3, \delta_2 = \delta/3$.

Section~\ref{sec:comparison-with-conf-seq} examines the coverage probability and the
width of the confidence intervals constructed on a synthetic data from $\Bern{0.1}$; for other cases, see Appendix~\ref{appsec:more-simulations}.
Section~\ref{sec:adaptive-stopping} and ~\ref{sec:bandit} apply the confidence sequences to an adaptive
stopping algorithm for $(\varepsilon, \delta)$-mean estimation and the best arm identification problem.
\subsection{Confidence Sequences for Bernoulli Variables}
\label{sec:comparison-with-conf-seq}
\begin{figure}[t]
    \centering
    \begin{subfigure}{.45\textwidth}
        \includegraphics[width=0.95\textwidth]{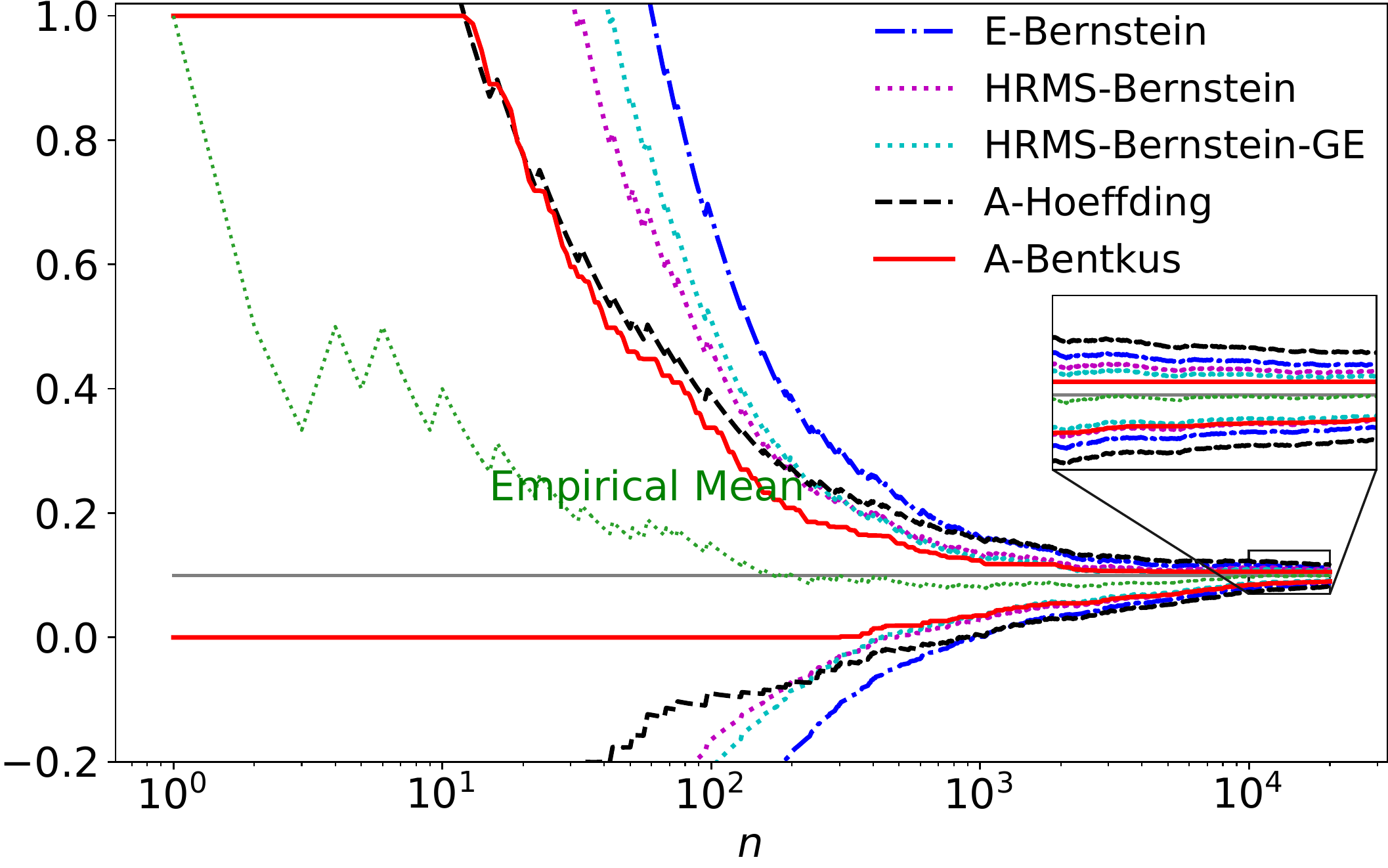}
    \caption{One Replication}
    \label{fig:one-replication}
    \end{subfigure}
    \begin{subfigure}{.45\textwidth}
    \includegraphics[width=0.95\textwidth]{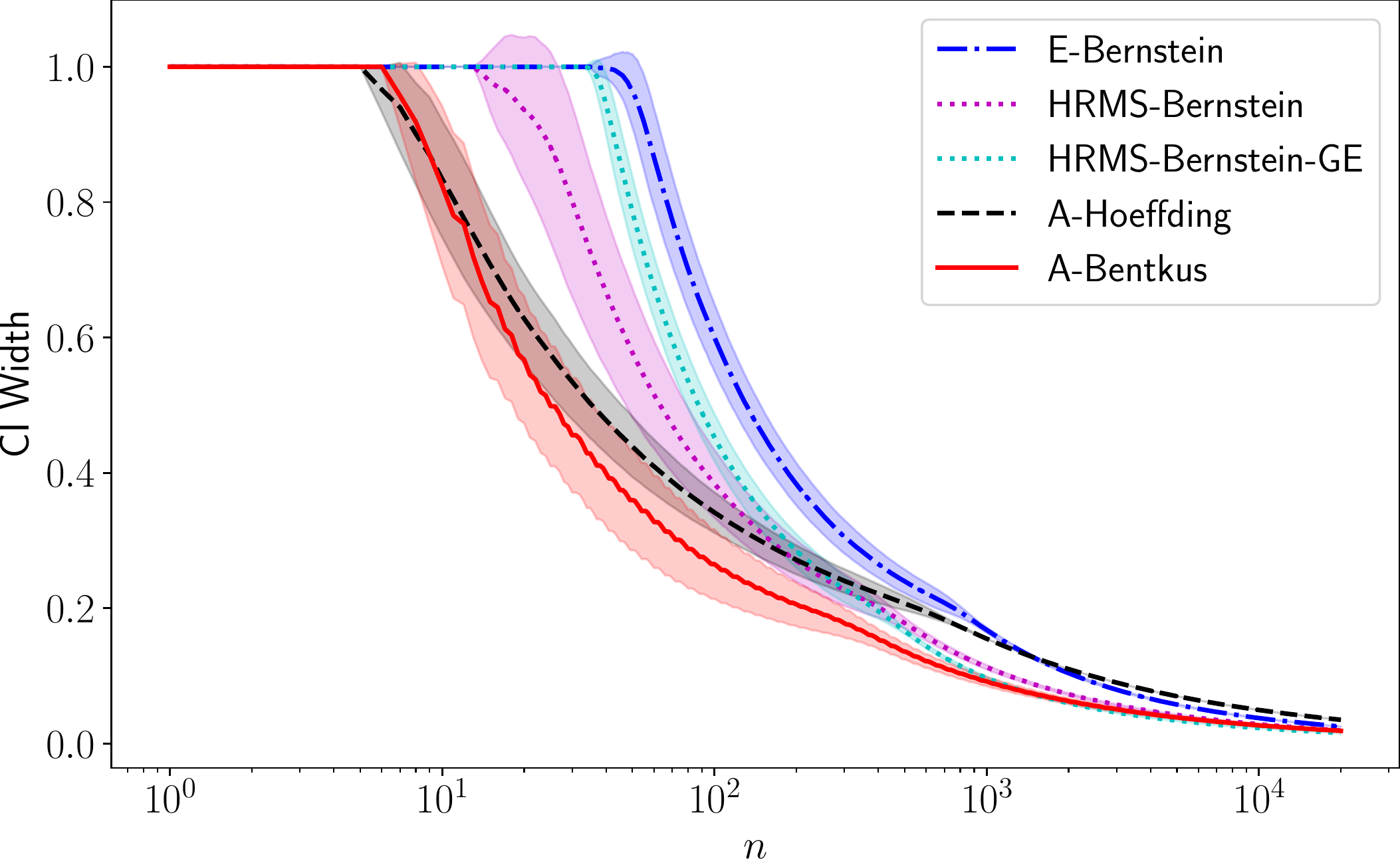}
    \caption{Average Width over 1000 Replications}
    \label{fig:multi-replications}
    \end{subfigure}
    \caption{
        The 95\% confidence sequences for the mean when $Y_i \sim \Bern{0.1}$. All the methods estimate the variance except \ah. \hog involves a tuning parameter $\rho$ which is chosen to optimize the sequence at $n = 500$ as suggested in \citet{howard2018uniform}.
        \textbf{(a)} shows the confidence sequences from a single replication.
        \textbf{(b)} shows the average widths of the confidence sequences over 1000 replications.
        The upper and lower bounds for all the other methods are cut at $1$ and $0$.
    }
\end{figure}

We generate samples $Y_1, Y_2, \ldots, Y_{20000}
\stackrel{\iid}{\sim} \Bern{0.1}$ and compute the confidence sequences for
$\mu = 0.1$.  Figure~\ref{fig:one-replication} illustrates the confidence
sequences obtained and shows the sharpness of \bk.
For most of the cases ($n \geq 20$), \bk dominates the other methods.  For smaller
sample sizes, \bk also closely traces \ah and outperforms the others.  This is because the variance estimation is likely conservative and in which case our $\widebar{A}^*_n$ ends up using the trivial upper bound $(U - L)/2$, which is essentially what \ah is exploiting.  In fact, we have provided the same upper bound for all the other Bernstein-type methods too, and \bk still outperforms. This phenomenon shows the intrinsic sharpness of our bound.  


We repeat the above experiment 1000 times and report the average miscoverage
rate:
\[
    \frac{1}{1000}\sum_{r=1}^{1000} \mathds{1}\{\mu \notin \mathrm{CI}_n^{(r)}
        \mbox{ for some } 1\le n\le 20000 \}.
\]
where $\mathrm{CI}_n^{(r)} $ is the confidence interval constructed after observing
$Y_1, \ldots, Y_n$ in the $r$-th replication.
The results are $0.001$ for $\bk$, $0.003$ for \hog, and $0$
for the others. All the methods control the miscoverage rate by $\delta = 0.05$ but are
all conservative. Recall from~\eqref{eq:main-optimality} that our failure
probability bound can be conservative up to a constant of $e^2/2$.  Furthermore, from the proofs of
Theorems~\ref{thm:uniform-in-n} and~\ref{thm:empirical-Bentkus-arbitrary-mean}, we get
that for $\eta = 1.1, h(k) = (k+1)^{1.1}\zeta(1.1)$,
\begin{align*}\textstyle
    &\mathbb{P}\big(\mu\notin\mathrm{CI}_n(\delta)\mbox{ for some }1\le n\le 20000\big)
    \\ 
    &\quad
    \le \textstyle\sum_{k=0}^{\log_{\eta}(20000)} {\delta}/{h(k)} \le 0.41\delta.
\end{align*}
For $\delta = 0.05$, $0.41\delta = 0.0205$. This
explains why the average miscoverage rate in the experiment is small.

We also report the average width of the confidence intervals in
Figure~\ref{fig:multi-replications}. All the values are between $0$ and $1$ as we cut the
bounds from above and below for the other methods.
As mentioned above, when $n$ is very small \bk closely traces \ah and both
    have smaller width. Yet the advantage of \ah disappears for $n\ge20$ and \bk enjoys
    smaller confidence interval width afterwards. \hog
    improves slightly on \bk after observing very large number of samples.



%
%

\subsection{Adaptive Stopping for Mean Estimation}\label{sec:adaptive-stopping}
In this section, we apply our confidence sequence to adaptively estimate the mean
of a bounded random variable $Y$.  The goal is to obtain an estimator $\muhat$ such that
the relative error $|\widehat{\mu}/\mu - 1|$ is bounded by $\varepsilon$, and
terminate the data sampling once such criterion is satisfied.



Given $\Ybar$ the empirical mean and any confidence sequence centered at
$\widebar{Y}$ satisfying~\eqref{eq:validity-guarantee},
Algorithm~\ref{alg:adaptive-stopping} yields a valid stopping time and an $(\varepsilon,
\delta)$-accurate estimator; see~\citet[Section 3.1]{mnih2008empirical} for a proof.
Clearly, a tighter confidence sequence will require less data sampling
and yields a smaller stopping time. We follow the setup in \citet{mnih2008empirical}. The data samples are i.i.d
generated as $Y_i = m^{-1}\sum_{j=1}^m U_{ij}$, where $U_{ij}$ are
$\iid$ uniformly distributed in $[0, 1]$. This implies that $\mu =
\frac{1}{2}$ and $A^2 = 
\frac{1}{12m}$.
\begin{algorithm}[h]
 \DontPrintSemicolon
  \textbf{Initialization:} $n \gets 0$,\, $\mathrm{LB} \gets 0$,\, $\mathrm{UB} \gets \infty$\;
  \While{$(1 + \varepsilon)\mathrm{LB} < (1 - \varepsilon)\mathrm{UB}$}
  {
    $n \gets n$ + 1\;
        {\parbox{\dimexpr\textwidth-2\algomargin\relax}{
        Sample $Y_n$ and compute the $n$-th CI in the sequence:
        }}\\
            \hskip20pt $[\widebar{Y}_n - Q_n, \widebar{Y}_n + Q_n] \gets \mbox{ConfSeq}(n, \delta)$\;
      $\mathrm{LB} \gets \max\{\mathrm{LB},\,|\widebar{Y}_n| - Q_n\}$\;
      $\mathrm{UB} \gets \min\{\mathrm{UB},\,|\widebar{Y}_n| + Q_n\}$\;
  }
  \Return{stopping time $N = n$ and estimator $\widehat{\mu} = (1/2)\mathrm{sign}(\widebar{Y}_N)[(1 + \varepsilon)\mathrm{LB} + (1 - \varepsilon)\mathrm{UB}]$}
    \caption{Adaptive Stopping Algorithm}
    \label{alg:adaptive-stopping}
\end{algorithm}
\begin{figure}[tb]
    \centering
    \includegraphics[width=0.55\textwidth]{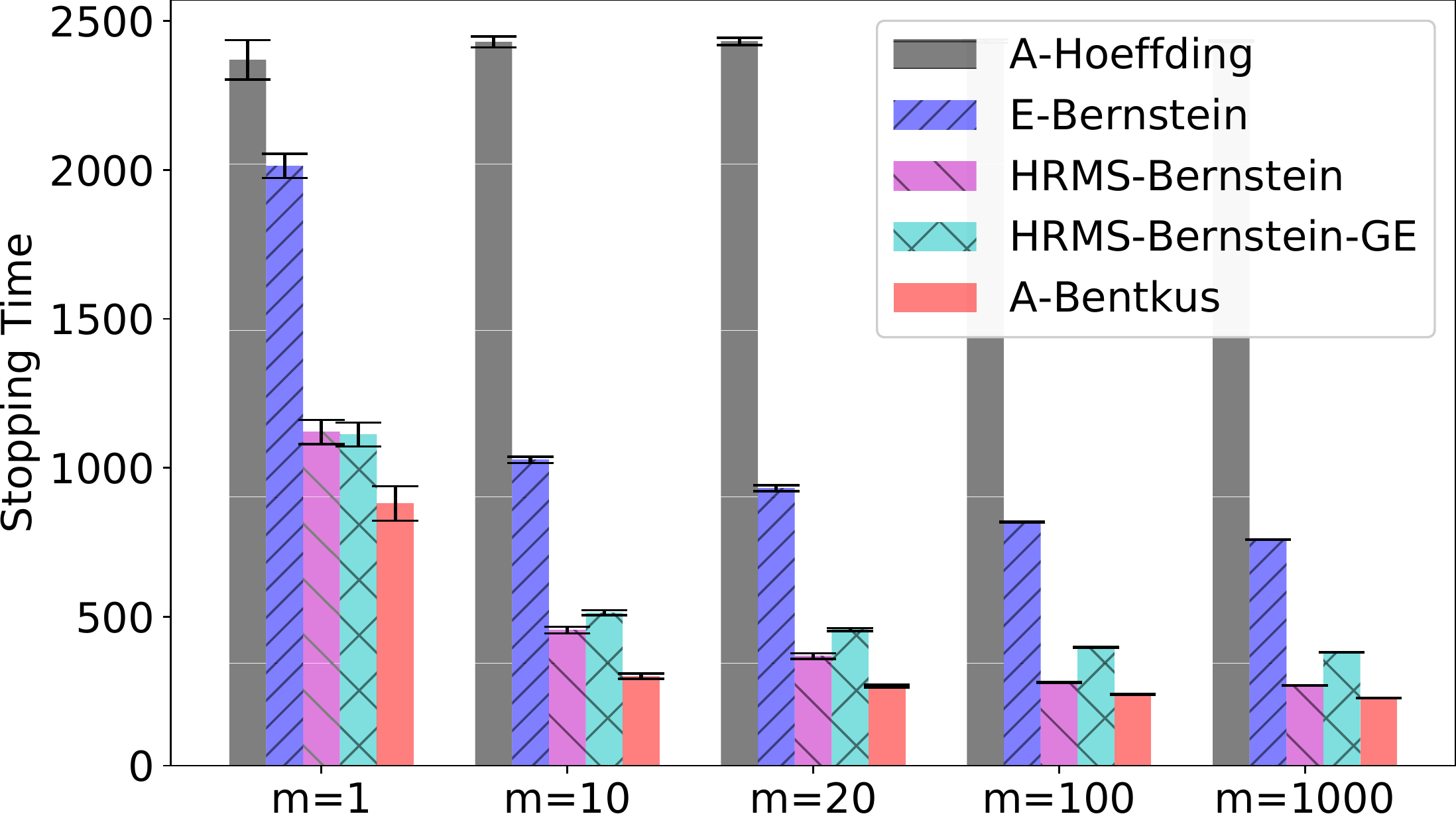}
    \caption{Comparison of confidence sequences for an  $(\varepsilon,
        \delta)$-estimator. Here $\varepsilon=0.1$ and $\delta = 0.05$.}
    \label{fig:adaptive_stop}
\end{figure}

Because Algorithm~\ref{alg:adaptive-stopping} requires
symmetric intervals, we shall symmetrize the intervals returned by \bk by taking
the largest deviation.
We consider 5 cases: $m = 1, 10, 20, 100, 1000$ and report the average stopping time (i.e. the number of samples required to achieve $(\epsilon, \delta) = (0.1, 0.05)$ accuracy)
based on $200$ trials in Figure~\ref{fig:adaptive_stop}. \hog involves a tuning parameter $\rho$, chosen here to optimize the confidence sequence at $n = 10$ (best out of $10, 50, 100, 200$). As $m$ increases, the variance of $Y_i$
decreases.
As expected, \ah does not exploit the variance of random variables so the stopping times
remains roughly the same. For others, the stopping time is decreasing.  It is clear that
on average, \bk is the best for all values of $m$
and the ratios of our stopping time to the second best are $0.79$, $0.66$, $0.72$,
$0.86$, $0.84$. 
\subsection{Best Arm Identification}
\label{sec:bandit}
In this section, we study the fixed confidence best arm identification, a classic multi-arm bandit problem. An agent is presented with a set of $K$ arms $\Arm$, with unknown expected rewards $\mu_1, \ldots, \mu_K$. Sequentially, the agent pulls an arm $\alpha \in \Arm$ of his choice and observes a reward value, until he finally claims one arm to have the largest expected reward. The goal is to correctly identify the best arm with fewer pulls $N$, i.e. smaller sample complexity. This problem has been extensively studied; see, e.g., \citet{
even2002pac, 
karnin2013almost, 
jamieson2013finding, 
jamieson2014lil, jamieson2014best, chen2015optimal}. \citet{zhao2016adaptive} provided an algorithm based on \ah that outperforms previous algorithms
including LIL-UCB, LIL-LUCB. 
Here we present it as Algorithm~\ref{alg:bestarm}
in a general form that utilizes any valid confidence sequences,
and use \bk as well as the competing ones in it to
compare their performance. Following the proof of \citet[Theorem 5]{zhao2016adaptive}, 
one can show that Algorithm~\ref{alg:bestarm} outputs the best arm with probability at least $1 - \delta$.

\begin{algorithm}[h]
\DontPrintSemicolon
\textbf{Input:} failure probability $\delta$, a set of arms $\Arm$\;
\textbf{Initialization:} $N \gets 0$; $n_\alpha \gets 0, \; \forall \alpha \in \Arm$\;
\While{$\Arm$ has more than one arms}
{
Compute empirical mean reward $\muhat_\alpha$, $\forall \alpha \in \Arm$\;
$\ahat \gets \argmax_{\alpha \in \Arm} \muhat_\alpha$\;
\For{every arm $\alpha \in \Arm$}{
    {\parbox{\dimexpr\textwidth-2\algomargin\relax}{
        $ \delta_\alpha \gets
        \begin{cases}
            \frac{\delta/2}{|\Arm|-1} & \mbox{if } \alpha = \ahat \\
            \frac{\delta}{2} & \mbox{otherwise}
        \end{cases}
        $\;
        $[L_{\alpha}, U_{\alpha}] \leftarrow$   the $n_\alpha$-th CI of $\mbox{ConfSeq}(\delta_\alpha)$\;
        $R_\alpha \leftarrow $ radius of the $n_\alpha$-th CI of $\mbox{ConfSeq}(\delta_\alpha)$\;
    }}
}
Sample from the arm $\alpha$ with largest radius $R_\alpha$\;
$n_{\alpha} \gets n_{\alpha} + 1$, $N \gets N+1$\;
Remove arm $\alpha$ from $\Arm$ if $U_\alpha < L_{\ahat}$
}
\Return{the remaining arm in $\Arm$, number of pulls $N$}
\caption{Best Arm Identification}
\label{alg:bestarm}
\end{algorithm}
\begin{figure}[t]
    \centering
    \includegraphics[width=0.55\textwidth]{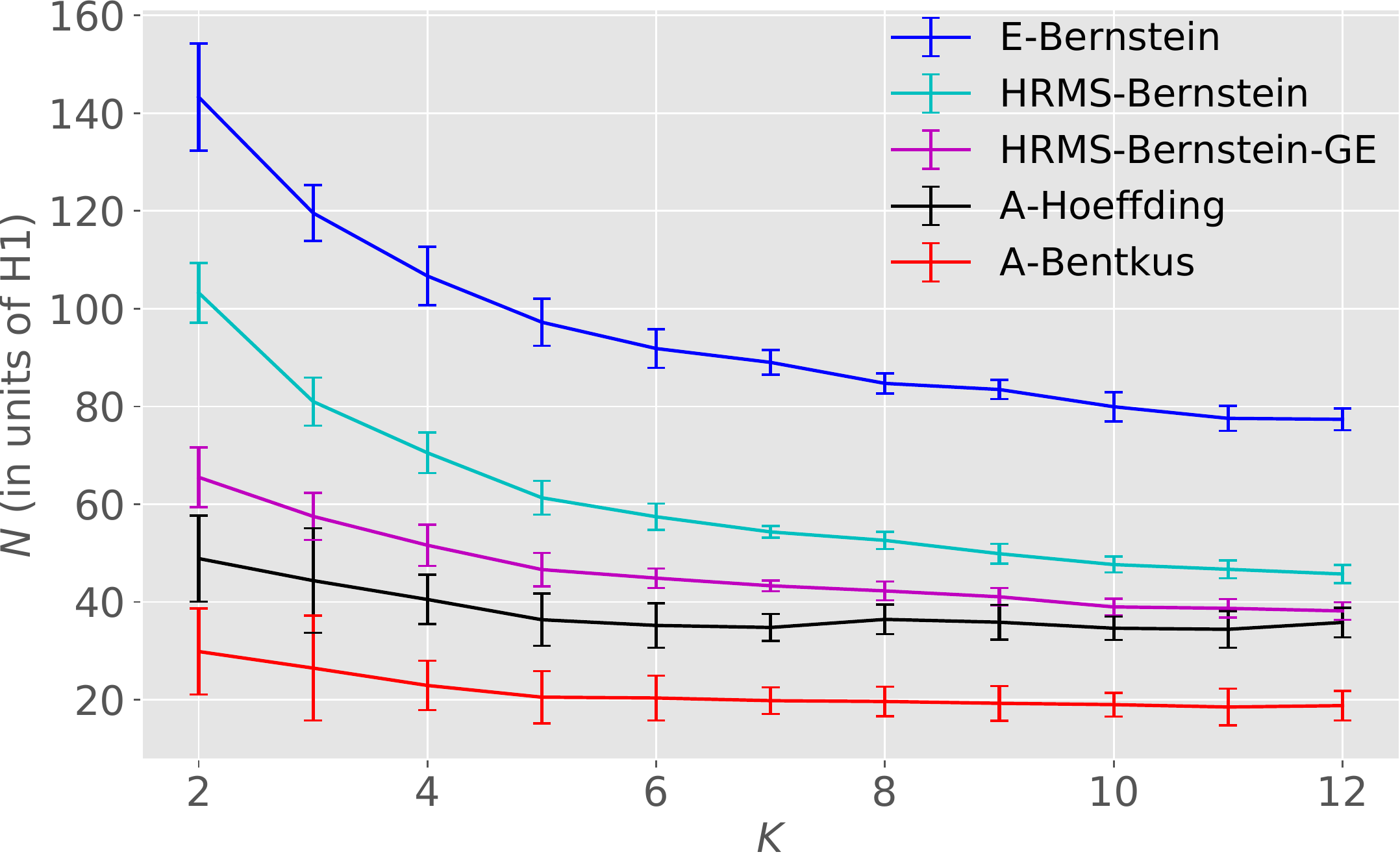}
    \caption{The number of pulls $N$ versus the the number of arms $K$. $\delta = 0.05$. The results are averaged over 10 trials. }
    \label{fig:best_arm}
\end{figure}

The experiment setup  follows \citet{
jamieson2014best, 
zhao2016adaptive}. Each arm is generating random Bernoulli rewards with 
$\mu_\alpha = 1 - (\frac{\alpha}{K})^{0.6},$ $\alpha = 0, \ldots, K-1$; the first arm has highest expected reward $\mu_0 =1$. The problem hardness is measured by 
$\text{H1} = \sum_{\alpha \neq 0} (\mu_\alpha - \mu_0)^{-2}$~\cite{jamieson2014best}, which is roughly $0.4K^{1.2}$ in our setup.

{
In Algorithm~\ref{alg:bestarm}, the sampling of an arm depends on $R_{\alpha}$, the radius of the confidence interval.
In our experiments, we find 
that a confidence sequence for which $R_{\alpha}$ stays constant for a stretch of samples yields a larger sample complexity. Our intuition is that Algorithm~\ref{alg:bestarm} keeps selecting the same arm when the radius is not updated, therefore it forgoes a number of samples; see Appendix~\ref{appsec:best_arm} for more details. 
This phenomenon happens for all confidence sequences when truncated to $[0,1]$, where the intervals stay constant at $[0, 1]$ for the first few samples, see Figure~\ref{fig:one-replication}. 
For \bk, the cumulative maximum/minimum ($\mu_n^{\low*}$ and $\mu_n^{\up*}$ in Theorem~\ref{thm:empirical-Bentkus-arbitrary-mean}) also leads to the constant radius problem. Hence, for smaller sample complexity, we set $L_{\alpha} = \mu_{n_\alpha}^{\low*}, U_{\alpha} = \mu_{n_\alpha}^{\up*}$ and $R_{\alpha}= (\mu_{n_\alpha}^{\up} - \mu_{n_\alpha}^{\low}) / 2 $ instead of $(\mu_{n_{\alpha}}^{\low*} - \mu_{n_{\alpha}}^{\up*})/2$.

Our experiments are reported in Figure~\ref{fig:best_arm}. \bk significantly outperforms the competing approaches, including \ah which beats LIL-UCB, LIL-LUCB \cite{zhao2016adaptive}. 
Further, \bk only requires 52\% to 61\% of the samples required by \ah. 
Finally, we note that the Bernstein type of methods underperform because they have larger confidence intervals for small to moderate number of samples as can be seen in Figure~\ref{fig:one-replication}.
}

\section{Conclusion}\label{sec:summary-future-directions}
We proposed a confidence sequence for bounded random variables and
examined its efficacy in synthetic examples and applications.
Our method is favorable to methods that utilize classical concentration results. It can be applied to various problems for improved performance, including testing equality of distributions, testing independence~\cite{balsubramani2016sequential}, etc.
Our work can be extended in a few future directions. We assumed that $X_i$'s are independent and bounded. The generalizations for the dependent case and/or the sub-Gaussian case are of interest. The generalized results can be obtained based on the results of \citet[Theorem 2.1]{pinelis2006binomial} and \citet{bentkus2010bounds}.

Regarding dependence, Theorem~2.1 of~\citet{pinelis2006binomial} shows that Theorem~\ref{thm:initial-maximal-ineq} holds even if $X_i$'s form a supermartingale difference sequence
, i.e., assumption~\eqref{eq:boundedness-conditions} is replaced by $\mathbb{E}[X_i|X_1, \ldots, X_{i-1}] \le 0$, $\mathbb{P}(X_i > B) = 0$, and
$\mathbb{P}(\mathbb{E}[X_i^2|X_1, \ldots, X_{i-1}] \le A_i^2) = 1.$
Theorem~\ref{thm:uniform-in-n} follows readily, but Theorem~\ref{thm:uniform-in-n-empirical-Bentkus} requires further restrictions that allow estimation of $A^2_i$. 

The boundedness assumption~\eqref{eq:boundedness-conditions}, which maybe restrictive for applications in statistics, finance and economics,
can be replaced by
\begin{equation}\label{eq:generalize-bnd-conditions}
\textstyle
\mathbb{E}[X_i] = 0,\;\mbox{Var}(X_i) \le A_i^2,\quad\mbox{and}\quad \mathbb{P}(X_i > x) \le \widebar{F}(x),
\end{equation}
for all $x\in\mathbb{R}$, where $\widebar{F}(\cdot)$ is a survival function on $[0, \infty)$, i.e., $\widebar{F}(\cdot)$ is non-increasing and $\widebar{F}(0) = 1$ and $\widebar{F}(\infty) = 0$.
For example,
$\widebar{F}(x) = 1/\{1 + (x/K)_+^{\alpha}\}$, {or} $\widebar{F}(x) = \exp(-(x/K)_+^{\alpha})$, {for some} $K > 0$;
$\alpha = 2$ in the second example is sub-Gaussianity.
Akin to~\eqref{eq:main-optimality}, there exist random variables $\eta_i = \eta_i(A_i^2, \widebar{F})$ satisfying~\eqref{eq:generalize-bnd-conditions} such that
$$\textstyle \sup_{X_i\sim~\eqref{eq:generalize-bnd-conditions}}\mathbb{E}[\left(\sum_{i=1}^n X_i - t\right)_+^2] = \mathbb{E}[\left(\sum_{i=1}^n \eta_i - t\right)_+^2], $$
for all $n\ge1$ and $t\in\mathbb{R}$. Here the superemum is taken over all distributions satisfying~\eqref{eq:generalize-bnd-conditions}. Hence, Theorem~\ref{thm:initial-maximal-ineq} can be generalized, which in turn leads to generalizations of Theorems~\ref{thm:uniform-in-n} and~\ref{thm:uniform-in-n-empirical-Bentkus}. The details on the construction of~$\eta_i$ and the corresponding confidence sequence will be discussed elsewhere.

\bibliography{AssumpLean}

\begin{thebibliography}{}

\bibitem[Abramowitz and Stegun, 1948]{abramowitz1948handbook}
Abramowitz, M. and Stegun, I.~A. (1948).
\newblock {\em Handbook of mathematical functions with formulas, graphs, and
  mathematical tables}, volume~55.
\newblock US Government printing office.

\bibitem[Audibert et~al., 2009]{audibert2009exploration}
Audibert, J.-Y., Munos, R., and Szepesv{\'a}ri, C. (2009).
\newblock Exploration--exploitation tradeoff using variance estimates in
  multi-armed bandits.
\newblock {\em Theoretical Computer Science}, 410(19):1876--1902.

\bibitem[Bahadur and Savage, 1956]{bahadur1956nonexistence}
Bahadur, R.~R. and Savage, L.~J. (1956).
\newblock The nonexistence of certain statistical procedures in nonparametric
  problems.
\newblock {\em The Annals of Mathematical Statistics}, 27(4):1115--1122.

\bibitem[Balsubramani and Ramdas, 2016]{balsubramani2016sequential}
Balsubramani, A. and Ramdas, A. (2016).
\newblock Sequential nonparametric testing with the law of the iterated
  logarithm.
\newblock In {\em Proceedings of the Thirty-Second Conference on Uncertainty in
  Artificial Intelligence}, pages 42--51.

\bibitem[Bennett, 1962]{bennett1962probability}
Bennett, G. (1962).
\newblock Probability inequalities for the sum of independent random variables.
\newblock {\em Journal of the American Statistical Association},
  57(297):33--45.

\bibitem[Bentkus, 2002]{bentkus2002remark}
Bentkus, V. (2002).
\newblock A remark on the inequalities of {B}ernstein, {P}rokhorov, {B}ennett,
  {H}oeffding, and {T}alagrand.
\newblock {\em Liet. Mat. Rink}, 42(3):332--342.

\bibitem[Bentkus, 2004]{bentkus2004hoeffding}
Bentkus, V. (2004).
\newblock On {H}oeffding’s inequalities.
\newblock {\em The Annals of Probability}, 32(2):1650--1673.

\bibitem[Bentkus, 2008]{bentkus2008extension}
Bentkus, V. (2008).
\newblock An extension of the hoeffding inequality to unbounded random
  variables.
\newblock {\em Lithuanian Mathematical Journal}, 48(2):137--157.

\bibitem[Bentkus, 2010]{bentkus2010bounds}
Bentkus, V. (2010).
\newblock Bounds for the stop loss premium for unbounded risks under the
  variance constraints.
\newblock {\em Preprint. Available at https://www.math.uni-bielefeld.
  de/sfb701/preprints/view/423}.

\bibitem[Bentkus et~al., 2006]{bentkus2006domination}
Bentkus, V., Kalosha, N., and Van~Zuijlen, M. (2006).
\newblock On domination of tail probabilities of (super) martingales: explicit
  bounds.
\newblock {\em Lithuanian Mathematical Journal}, 46(1):1--43.

\bibitem[Chen and Li, 2015]{chen2015optimal}
Chen, L. and Li, J. (2015).
\newblock On the optimal sample complexity for best arm identification.
\newblock {\em arXiv preprint arXiv:1511.03774}.

\bibitem[Dagum et~al., 2000]{dagum2000optimal}
Dagum, P., Karp, R., Luby, M., and Ross, S. (2000).
\newblock An optimal algorithm for {M}onte {C}arlo estimation.
\newblock {\em SIAM Journal on computing}, 29(5):1484--1496.

\bibitem[Even-Dar et~al., 2002]{even2002pac}
Even-Dar, E., Mannor, S., and Mansour, Y. (2002).
\newblock Pac bounds for multi-armed bandit and markov decision processes.
\newblock In {\em International Conference on Computational Learning Theory},
  pages 255--270. Springer.

\bibitem[Fan et~al., 2012]{fan2012missing}
Fan, X., Grama, I., and Liu, Q. (2012).
\newblock The missing factor in bennett’s inequality.
\newblock {\em arXiv preprint arXiv:1206.2592}.

\bibitem[Gin{\'e} and Nickl, 2016]{gine2016mathematical}
Gin{\'e}, E. and Nickl, R. (2016).
\newblock {\em Mathematical foundations of infinite-dimensional statistical
  models}, volume~40.
\newblock Cambridge University Press.

\bibitem[Hoeffding, 1963]{hoeffding1963probability}
Hoeffding, W. (1963).
\newblock Probability inequalities for sums of bounded random variables.
\newblock {\em Journal of the American Statistical Association},
  58(301):13--30.

\bibitem[Howard et~al., 2018]{howard2018uniform}
Howard, S.~R., Ramdas, A., McAuliffe, J., and Sekhon, J. (2018).
\newblock Uniform, nonparametric, non-asymptotic confidence sequences.
\newblock {\em arXiv preprint arXiv:1810.08240}.

\bibitem[Huber, 2019]{huber2019optimal}
Huber, M. (2019).
\newblock An optimal $(\epsilon, \delta)$-randomized approximation scheme for
  the mean of random variables with bounded relative variance.
\newblock {\em Random Structures \& Algorithms}, 55(2):356--370.

\bibitem[Jamieson et~al., 2013]{jamieson2013finding}
Jamieson, K., Malloy, M., Nowak, R., and Bubeck, S. (2013).
\newblock On finding the largest mean among many.
\newblock {\em arXiv preprint arXiv:1306.3917}.

\bibitem[Jamieson et~al., 2014]{jamieson2014lil}
Jamieson, K., Malloy, M., Nowak, R., and Bubeck, S. (2014).
\newblock {LIL-UCB}: An optimal exploration algorithm for multi-armed bandits.
\newblock In {\em Conference on Learning Theory}, pages 423--439. PMLR.

\bibitem[Jamieson and Nowak, 2014]{jamieson2014best}
Jamieson, K. and Nowak, R. (2014).
\newblock Best-arm identification algorithms for multi-armed bandits in the
  fixed confidence setting.
\newblock In {\em 2014 48th Annual Conference on Information Sciences and
  Systems (CISS)}, pages 1--6. IEEE.

\bibitem[Johari et~al., 2017]{johari2017peeking}
Johari, R., Koomen, P., Pekelis, L., and Walsh, D. (2017).
\newblock Peeking at {A}/{B} tests: Why it matters, and what to do about it.
\newblock In {\em Proceedings of the 23rd ACM SIGKDD International Conference
  on Knowledge Discovery and Data Mining}, pages 1517--1525.

\bibitem[Johari et~al., 2015]{johari2015always}
Johari, R., Pekelis, L., and Walsh, D.~J. (2015).
\newblock Always valid inference: Bringing sequential analysis to {A}/{B}
  testing.
\newblock {\em arXiv preprint arXiv:1512.04922}.

\bibitem[Karnin et~al., 2013]{karnin2013almost}
Karnin, Z., Koren, T., and Somekh, O. (2013).
\newblock Almost optimal exploration in multi-armed bandits.
\newblock In {\em International Conference on Machine Learning}, pages
  1238--1246. PMLR.

\bibitem[Mnih et~al., 2008]{mnih2008empirical}
Mnih, V., Szepesv{\'a}ri, C., and Audibert, J.-Y. (2008).
\newblock Empirical {B}ernstein stopping.
\newblock In {\em Proceedings of the 25th international conference on Machine
  learning}, pages 672--679.

\bibitem[Philips and Nelson, 1995]{philips1995moment}
Philips, T.~K. and Nelson, R. (1995).
\newblock The moment bound is tighter than {C}hernoff's bound for positive tail
  probabilities.
\newblock {\em The American Statistician}, 49(2):175--178.

\bibitem[Pinelis, 2006]{pinelis2006binomial}
Pinelis, I. (2006).
\newblock Binomial upper bounds on generalized moments and tail probabilities
  of (super) martingales with differences bounded from above.
\newblock In {\em High dimensional probability}, pages 33--52. Institute of
  Mathematical Statistics.

\bibitem[Pinelis, 2009]{pinelis2009bennett}
Pinelis, I. (2009).
\newblock On the {B}ennett-{H}oeffding inequality.
\newblock {\em arXiv preprint arXiv:0902.4058}.

\bibitem[Pinelis, 2016]{pinelis2016optimal}
Pinelis, I. (2016).
\newblock Optimal binomial, poisson, and normal left-tail domination for sums
  of nonnegative random variables.
\newblock {\em Electronic Journal of Probability}, 21.

\bibitem[Singh, 1963]{singh1963existence}
Singh, R. (1963).
\newblock Existence of bounded length confidence intervals.
\newblock {\em The Annals of Mathematical Statistics}, 34(4):1474--1485.

\bibitem[Talagrand, 1995]{talagrand1995missing}
Talagrand, M. (1995).
\newblock The missing factor in {H}oeffding's inequalities.
\newblock In {\em Annales de l'IHP Probabilit{\'e}s et statistiques},
  volume~31, pages 689--702.

\bibitem[Wald, 2004]{wald2004sequential}
Wald, A. (2004).
\newblock {\em Sequential analysis}.
\newblock Courier Corporation.

\bibitem[Wellner, 2017]{wellner2017bennett}
Wellner, J.~A. (2017).
\newblock The bennett-orlicz norm.
\newblock {\em Sankhya A}, 79(2):355--383.

\bibitem[Yang et~al., 2017]{yang2017framework}
Yang, F., Ramdas, A., Jamieson, K.~G., and Wainwright, M.~J. (2017).
\newblock A framework for multi-{A} (rmed)/{B} (andit) testing with online fdr
  control.
\newblock In {\em Advances in Neural Information Processing Systems}, pages
  5957--5966.

\bibitem[Zhao et~al., 2016]{zhao2016adaptive}
Zhao, S., Zhou, E., Sabharwal, A., and Ermon, S. (2016).
\newblock Adaptive concentration inequalities for sequential decision problems.
\newblock In {\em Advances in Neural Information Processing Systems}, pages
  1343--1351.

\end{thebibliography}
\bibliographystyle{apalike}
\clearpage
\appendix
\section{Competing Concentration Bounds}
\begin{thm}[Hoeffding; Theorem 3.1.2 of~\citet{gine2016mathematical}]
If $X_1, \ldots, X_n$ are independent mean-zero random variables satisfying $\mathbb{P}(\underline{B} \le X_i \le B) = 1$, then
    \[
        \P \left( S_n \geq \sqrt{\frac{1}{2} n (B - \underline{B})^2 \log \left(\frac{1}{\delta}\right)} \, \right) \leq \delta, \quad  \forall \delta \in [0, 1].
    \]
\end{thm}
(There is a generalization of Hoeffding's inequality that relaxes the boundedness assumption by a sub-Gaussian assumption; see~\citet{zhao2016adaptive} for details.)

\begin{thm}[Adaptive Hoeffding; Corollary 1 of~\citet{zhao2016adaptive}]
If $X_1, \ldots, X_n$ are independent mean-zero random variables satisfying $\mathbb{P}(\underline{B} \le X_i \le B) = 1$, then
    \[
        \P \left(\exists n\ge 1:\, S_n \geq (B-\underline{B}) \sqrt{ 0.6 n \log(\log_{1.1}n + 1) + \frac{\log(12/\delta)}{1.8}n } \, \right)
        \leq \delta, \quad \forall \delta \in [0, 1].
    \]
\end{thm}

\begin{thm}[Bernstein; Theorem 3.1.7 of~\citet{gine2016mathematical}] If $X_1, \ldots, X_n, \ldots$ are independent random variables satisfying~\eqref{eq:boundedness-conditions}, then
    \[
        \P \left( S_n \geq \sqrt{
         2 \sum_{i=1}^n A_i^2 \log\left(\frac{1}{\delta}\right)
         + \frac{1}{9} B^2\log^2\left(\frac{1}{\delta}\right)
        } + \frac{1}{3}B\log\left(\frac{1}{\delta}\right)\, \right)
        \leq \delta, \quad \forall \delta \in [0, 1].
    \]
\end{thm}

\begin{thm}[Empirical Bernstein; Eq. (5) of~\citet{mnih2008empirical}]
If $X_1, X_2, \ldots$ are independent mean zero random variables satisfying~\eqref{eq:boundedness-conditions} with $A_1 = A_2 = \ldots = A$, then
\[
\mathbb{P}\left(\exists n \ge 1:\, S_n \ge \sqrt{2 n \eta \widehat{A}_n^2\log(3h(k_n)/(2\delta))} + 3 B \eta\log(3h(k_n)/(2\delta))\right) \le \delta,
\]
where $\widehat{A}_n^2$ is the sample variance and $k_n$ is the constant defined in Theorem~\ref{thm:uniform-in-n}.
\end{thm}

\section{More Simulations}\label{appsec:more-simulations}
\subsection{Hyperparameters of Stitching}\label{appsec:hyperparam_stitch}
In Section~\ref{sec:unif-conf-seq}, we mentioned that there are two hyperparameters of our stitching methods: (1) the spacing parameter $\eta >
1$ and (2) the power parameter $c > 1$ for the stitching function $h_c(k) = \zeta(c)
(k+1)^c$ where $\zeta(\cdot)$ is the Riemann zeta function.
\clearpage
\begin{figure}[h]
    \centering
    \includegraphics[width=0.5\textwidth]{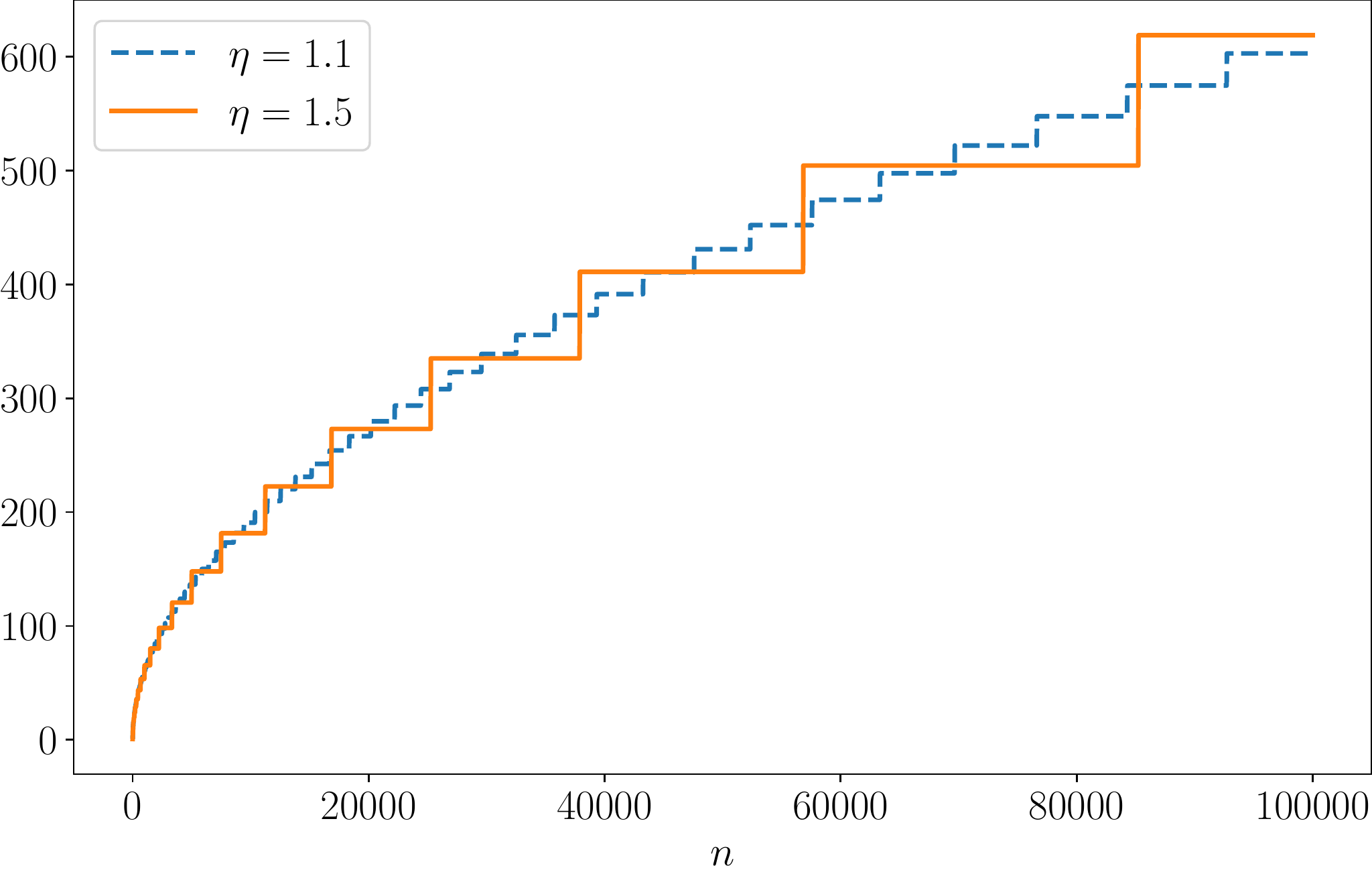}
    \caption{The upper bound of $S_n$ obtained by adaptive Bentkus bound in Theorem~\ref{thm:uniform-in-n} for different values of $\eta$.
        Both the variance $A = \sqrt{3}/4$ and the upper bound $B=3/4$ is known.}
    \label{fig:hyperparam_eta}
\end{figure}
Figure~\ref{fig:hyperparam_eta} illustrates that
the choice of $\eta$ determines how the budget $\delta$ is distributed across different sample sizes.

\begin{figure}[h]
    \centering
    \includegraphics[width=0.45\textwidth]{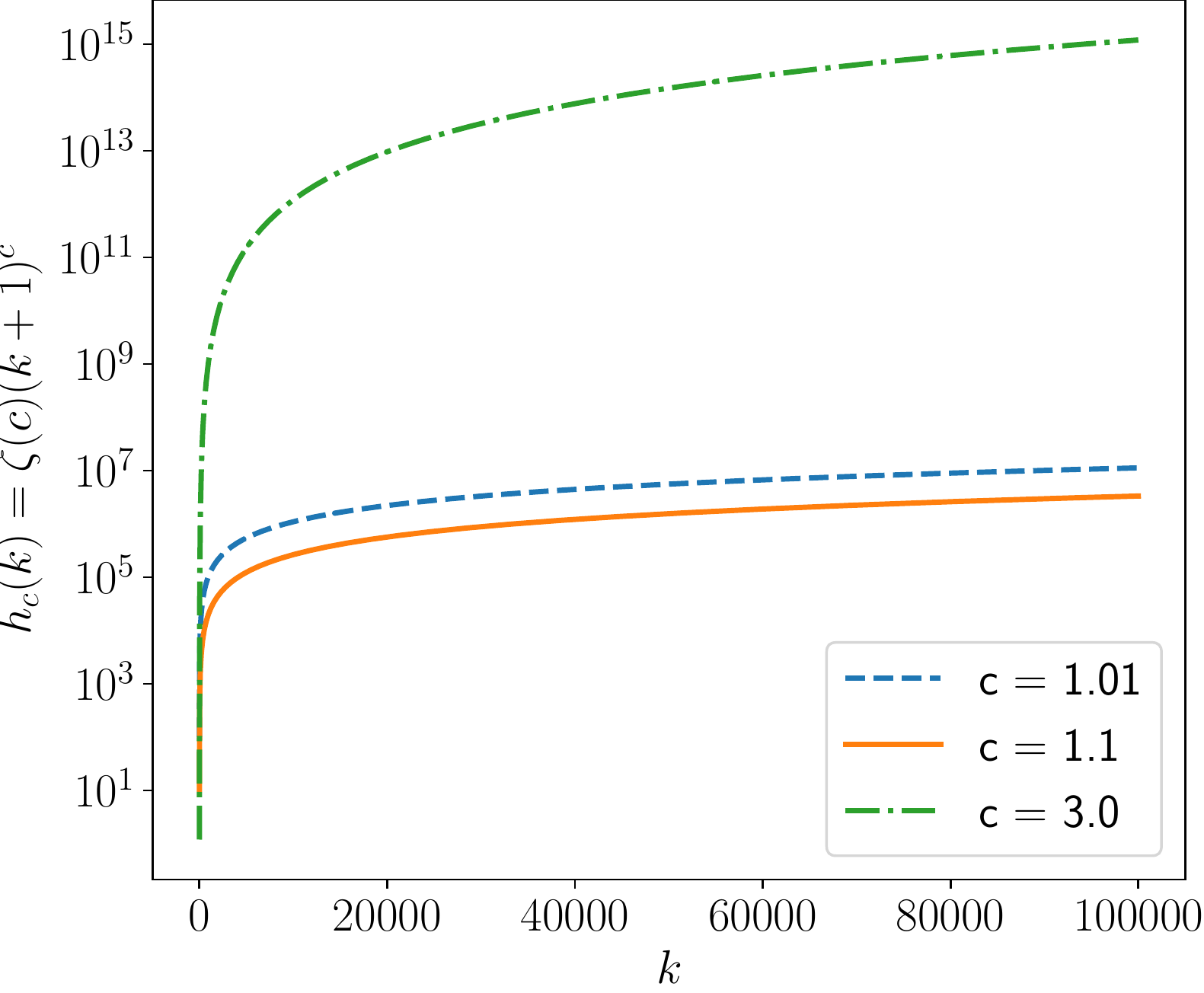}
    \includegraphics[width=0.43\textwidth]{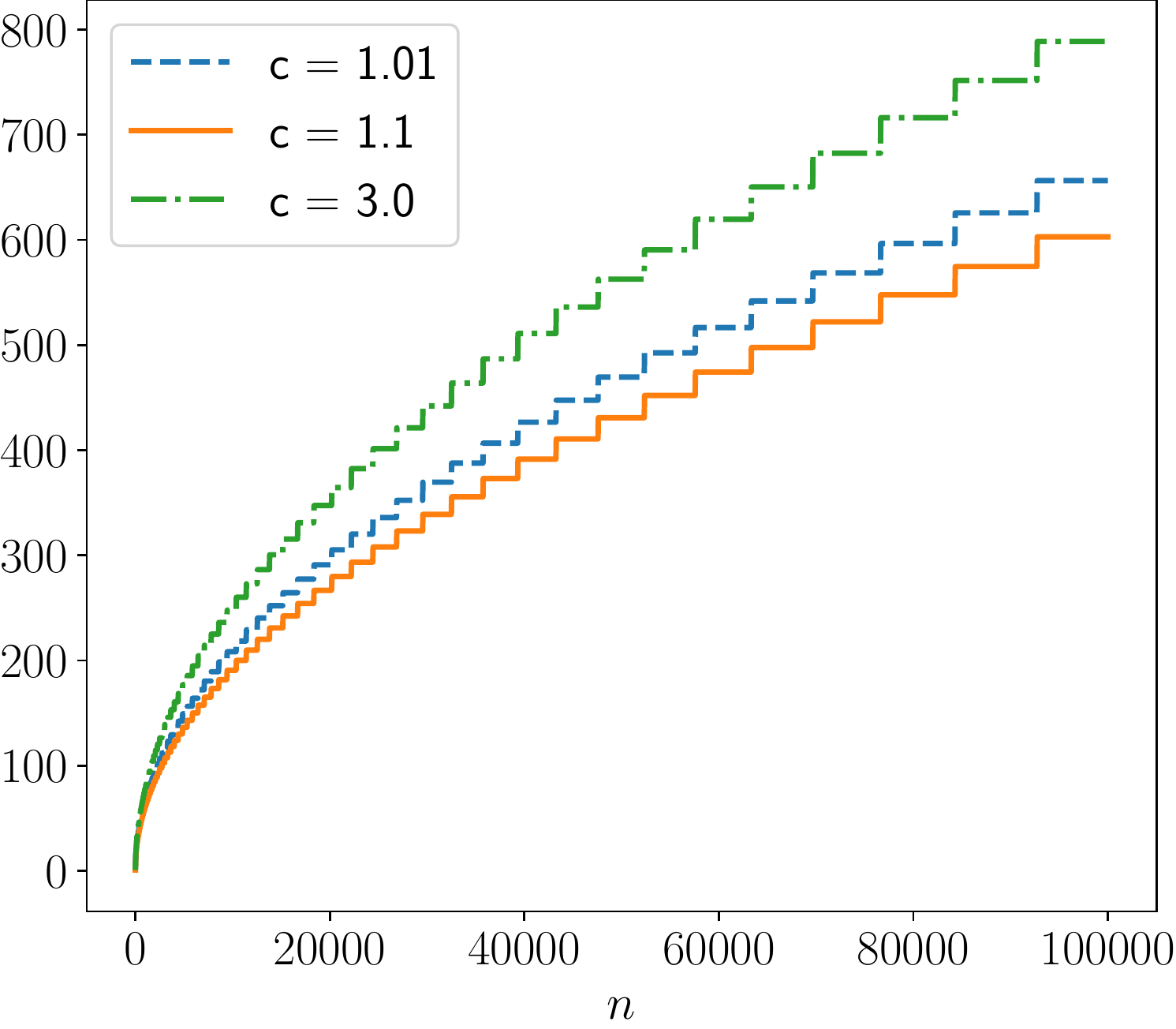}
    \caption{
        \textbf{Left:} The stitching function $h_c(\cdot)$ for different values of $c$.
        \textbf{Right:} The upper bound of $S_n$ obtained by \bk with different values
        of $c$. Both the variance $A^2 = 3/16$ and the upper bound $B=3/4$ is known.
    }
    \label{fig:hyperparam_power}
\end{figure}
Figure \ref{fig:hyperparam_power} shows both the stitching function $h_c(\cdot)$
and corresponding upper bound \bk obtains. For a fixed sample size $n$,
the bigger $h_c(k_n)$ is, the smaller budget
$\delta / h_c(k_n)$ it obtains and hence it needs a larger upper bound.
Hence, the faster $h_c(\cdot)$ grows, the more conservative upper bound (and
corresponding, wider confidence interval) one will get.

\subsection[Confidence Sequence Unbiased Coin Toss]{Confidence Sequence for $\Bern{0.5}$}
In this section, we present a comparison of our confidence sequence with \ah, \eb,
\ho, and \hog on synthetic data from $\Bern{0.5}$. In this case, $Y_1, Y_2, \ldots\sim\Bern{0.5}$ and the variance is $1/4$. Hence in this case Hoeffding's inequality is sharp and nothing can be gained by variance exploitation. We observe this very fact in our experiment, where our method behaves as well as \ah for moderate to large sample sizes. Figures~\ref{fig:one-replication_phalf} and~\ref{fig:multi-replications_phalf} show the comparison of confidence sequences in one replication and comparison of average width over 1000 replications. As in the case of $\Bern{0.1}$ (Section~\ref{sec:comparison-with-conf-seq}), for small sample sizes, \ah and \bk behave very closely and are better than all other methods but for $n$ moderately large, the sharpness of \bk clearly pays off by outperforming \ah and all other methods.
\begin{figure}[h]
        \centering
        \begin{subfigure}{.49\textwidth}
            \includegraphics[width=0.95\textwidth]{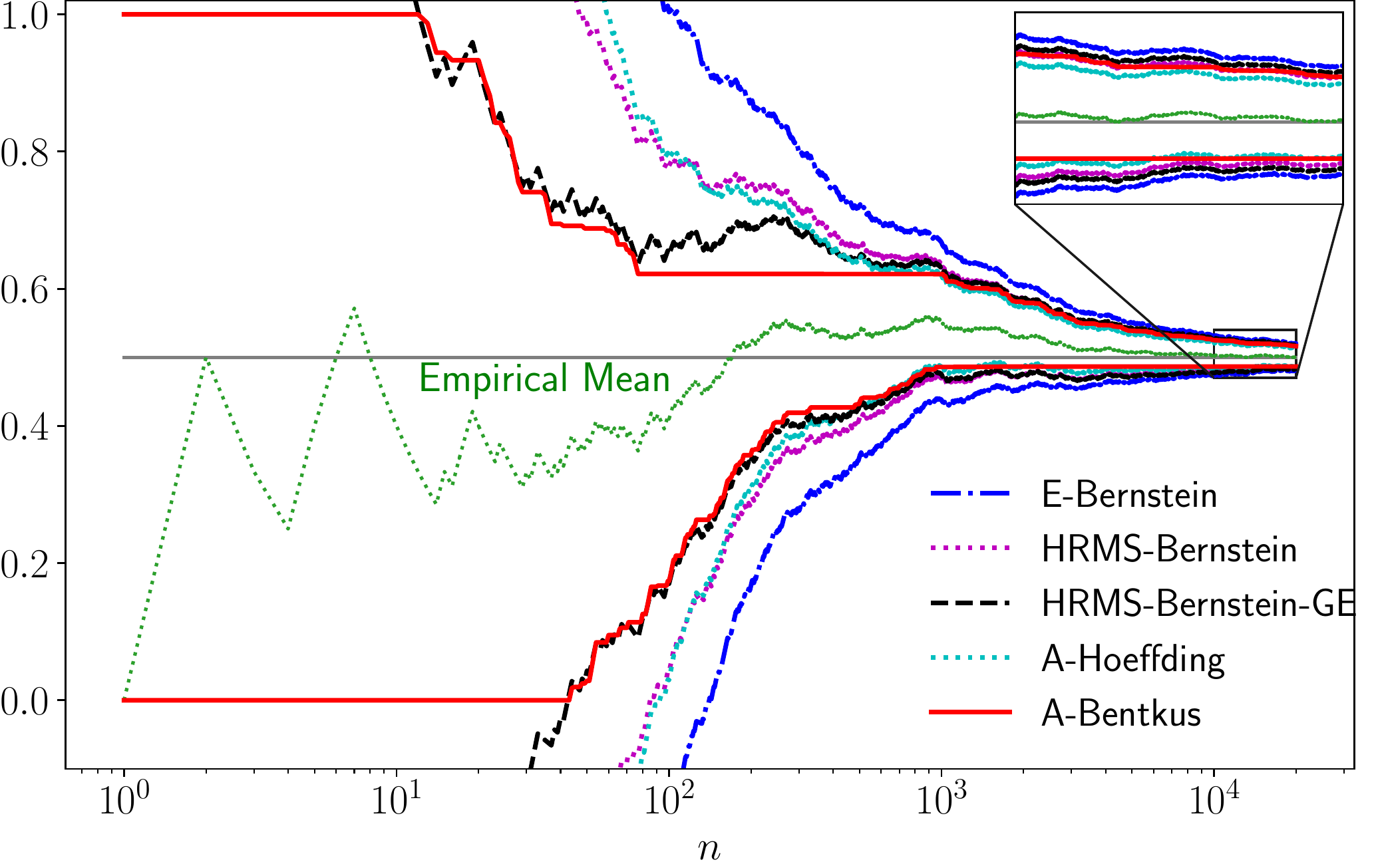}
        \caption{One Replication}
        \label{fig:one-replication_phalf}
        \end{subfigure}
        \begin{subfigure}{.49\textwidth}
        \includegraphics[width=0.95\textwidth]{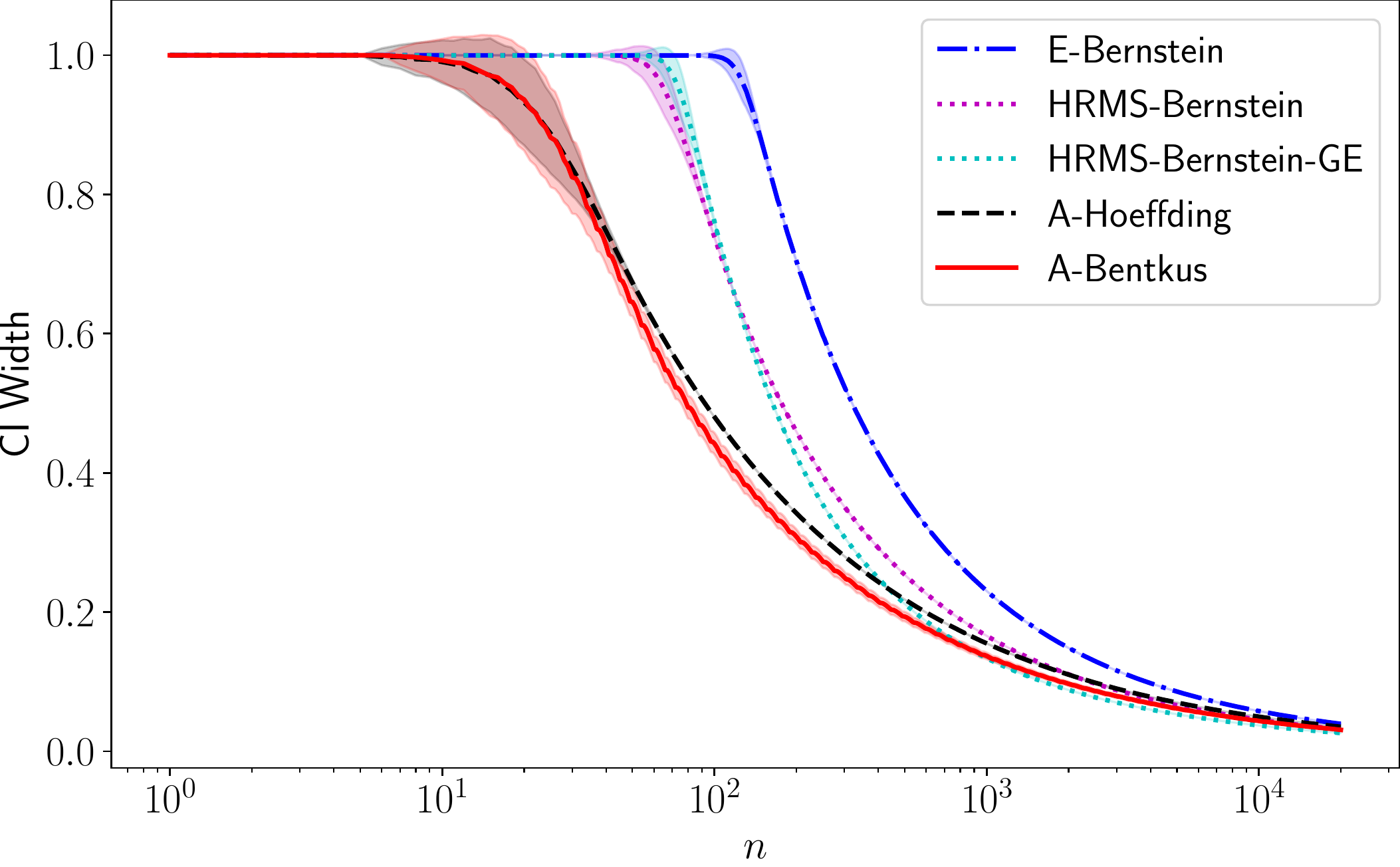}
        \caption{Average Width over 1000 Replications}
        \label{fig:multi-replications_phalf}
        \end{subfigure}
        \caption{
            Comparison of the 95\% confidence sequences for the mean when $Y_i \sim \Bern{0.5}$. Except \ah, all other methods estimate the variance. \bk is the confidence sequence in~\eqref{eq:empirical-bentkus-conf-seq}. \hog involves a tuning parameter $\rho$ which is chosen to optimize the boundary at $n = 500$.
            \textbf{(a)} shows the confidence sequences from a single replication.
            \textbf{(b)} shows the average widths of the confidence sequences over 1000 replications.
            The upper and lower bounds for all the other methods are cut at $1$ and $0$
            for a fair comparison. The failure frequency is $0.001$ for \hog and $0$
            for the others.
        }
\end{figure}

\subsection{Discussion for the Best Arm Identification Problem}\label{appsec:best_arm}
In Section~\ref{sec:bandit}, we mentioned that a confidence sequence for which the radius $R_{\alpha}$ stays constant for a stretch of samples yields a larger sample complexity. We present here more experimental details regarding this behavior. 

In the following, we experiment with a single instance of best arm identification problem
where the number of arms is 2 (i.e., $K=2$). The expected rewards are generated as the same as in Section~\ref{sec:bandit}, so that Arm 0 has mean $\mu_0 = 1$ is the best arm, and Arm 1 has mean $\mu_1 \approx 0.34$. For all the methods, we use the same data.

\begin{figure}[H]
\centering
\begin{subfigure}{.49\textwidth}
\centering
    \includegraphics[width=0.95\textwidth]{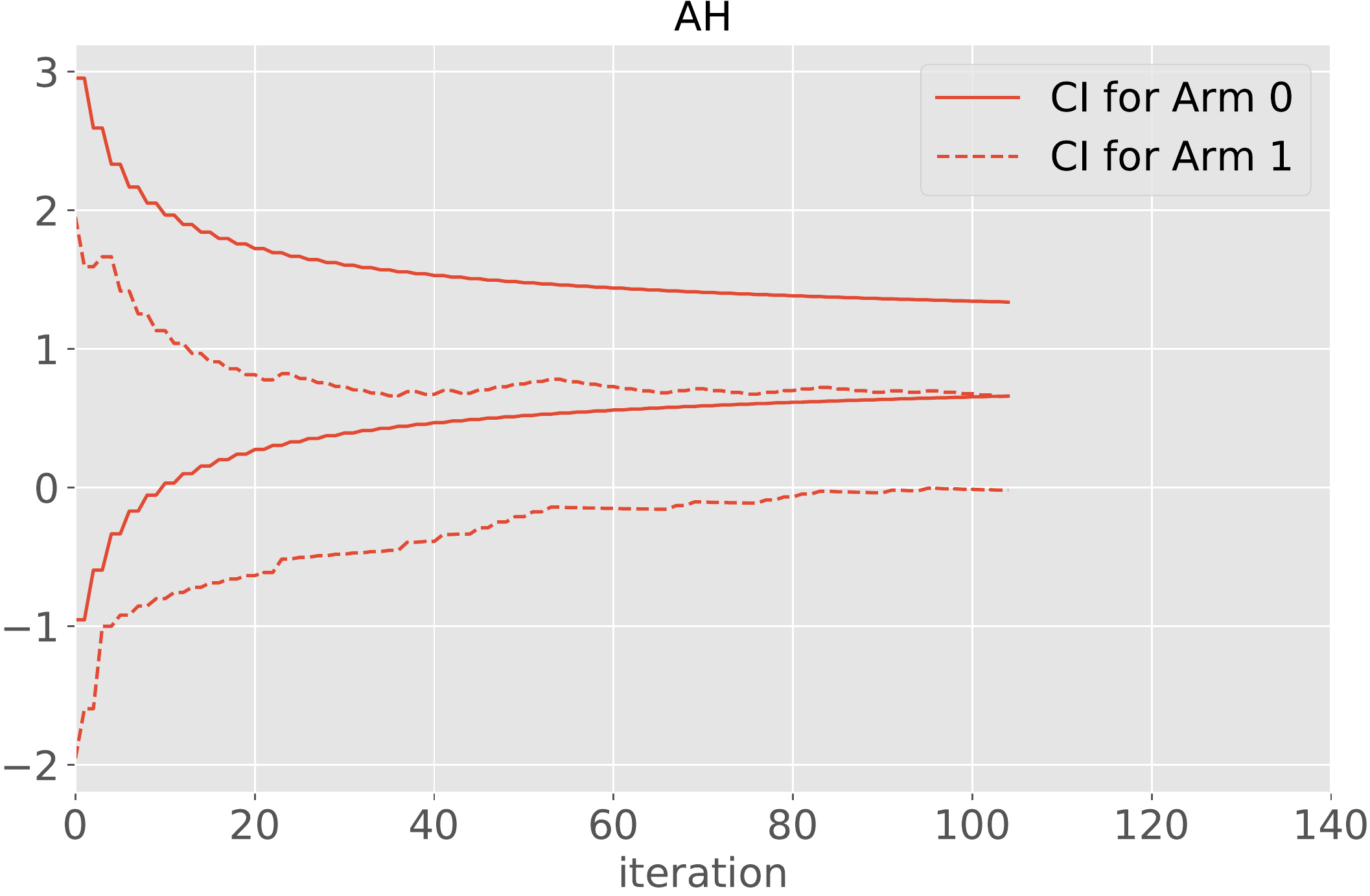}
\caption{Confidence intervals of \ah for two arms.}
\label{fig:best_arm_confseq_ah}
\end{subfigure}
\begin{subfigure}{.49\textwidth}
    \centering
    \includegraphics[width=0.95\textwidth]{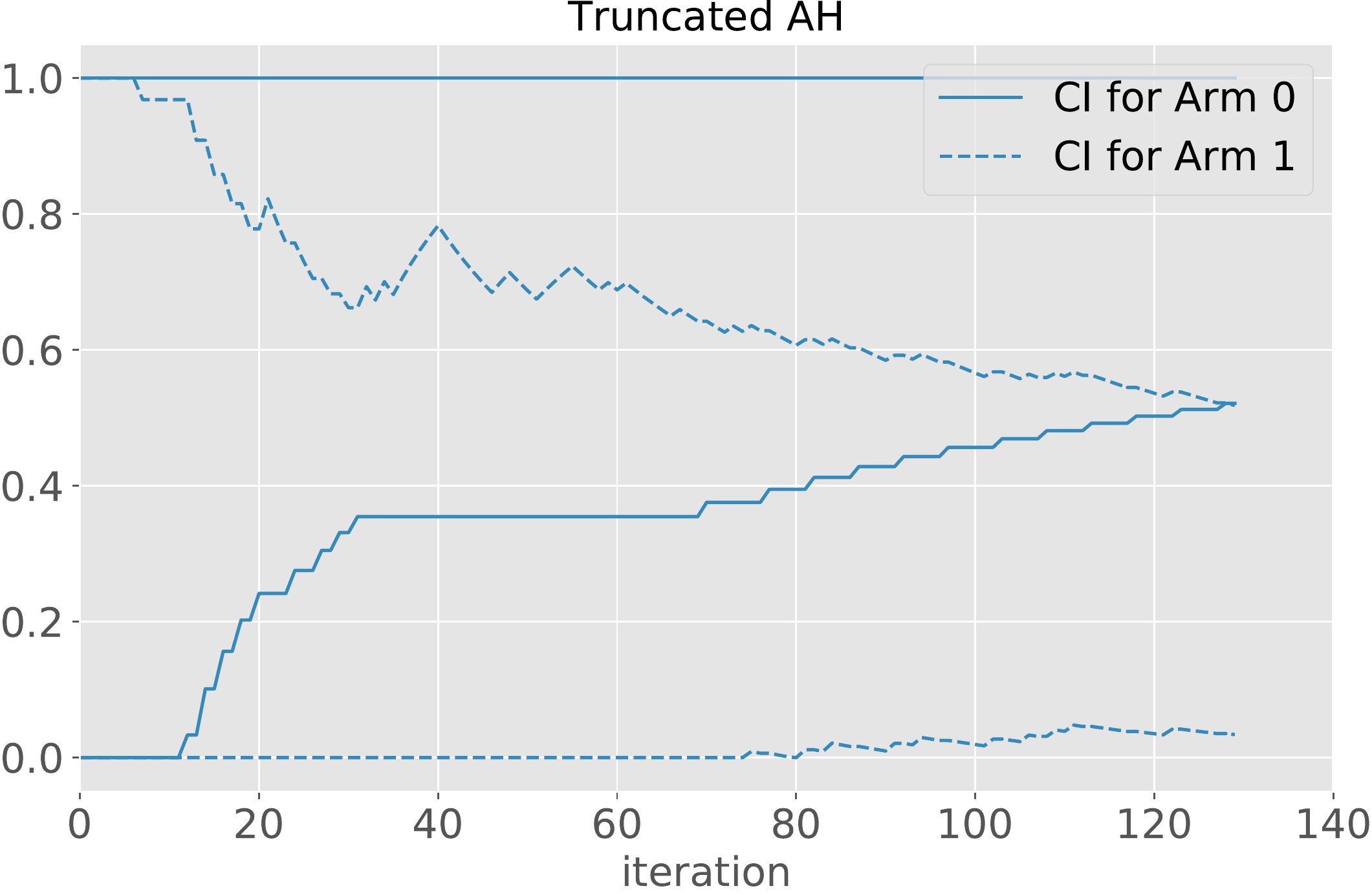}
    \caption{Confidence intervals of truncated \ah.} 
\label{fig:best_arm_confseq_ah_truncated}
\end{subfigure}
\begin{subfigure}{.49\textwidth}
    \centering
    \includegraphics[width=0.95\textwidth]{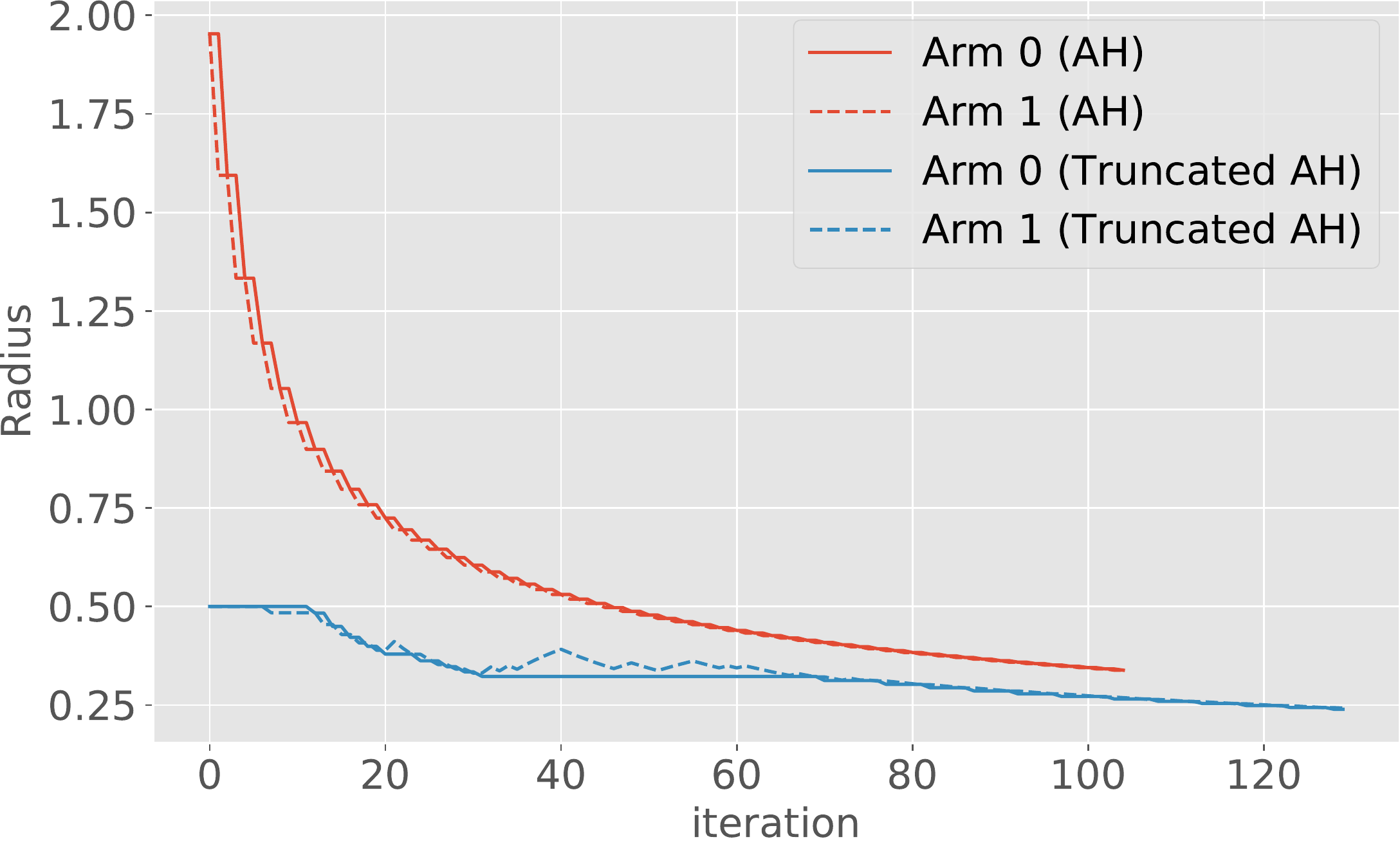}
    \caption{Radius of \ah (original and truncated) for two arms.}
\label{fig:best_arm_radius_ah.pdf}
\end{subfigure}
\begin{subfigure}{.49\textwidth}
    \centering
    \includegraphics[width=0.95\textwidth]{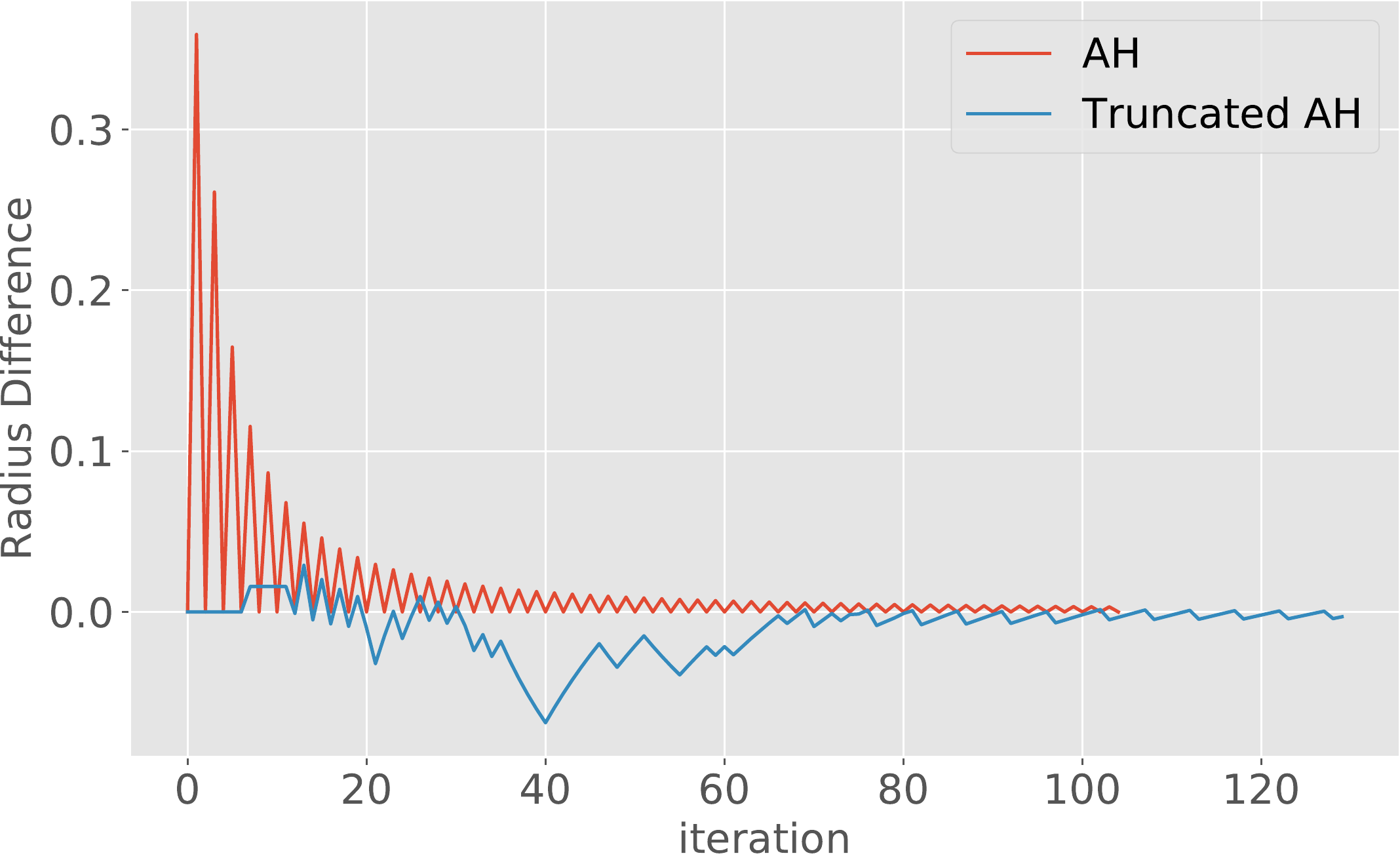}
    \caption{The difference of the radius for two arms: $R_0 - R_1$. For positive difference value, Arm 0 will be pulled. For negative difference
    value, Arm 1 will be pulled.}
\label{fig:best_arm_radius_diff_ah}
\end{subfigure}
\begin{subfigure}{.49\textwidth}
    \centering
    \includegraphics[width=0.95\textwidth]{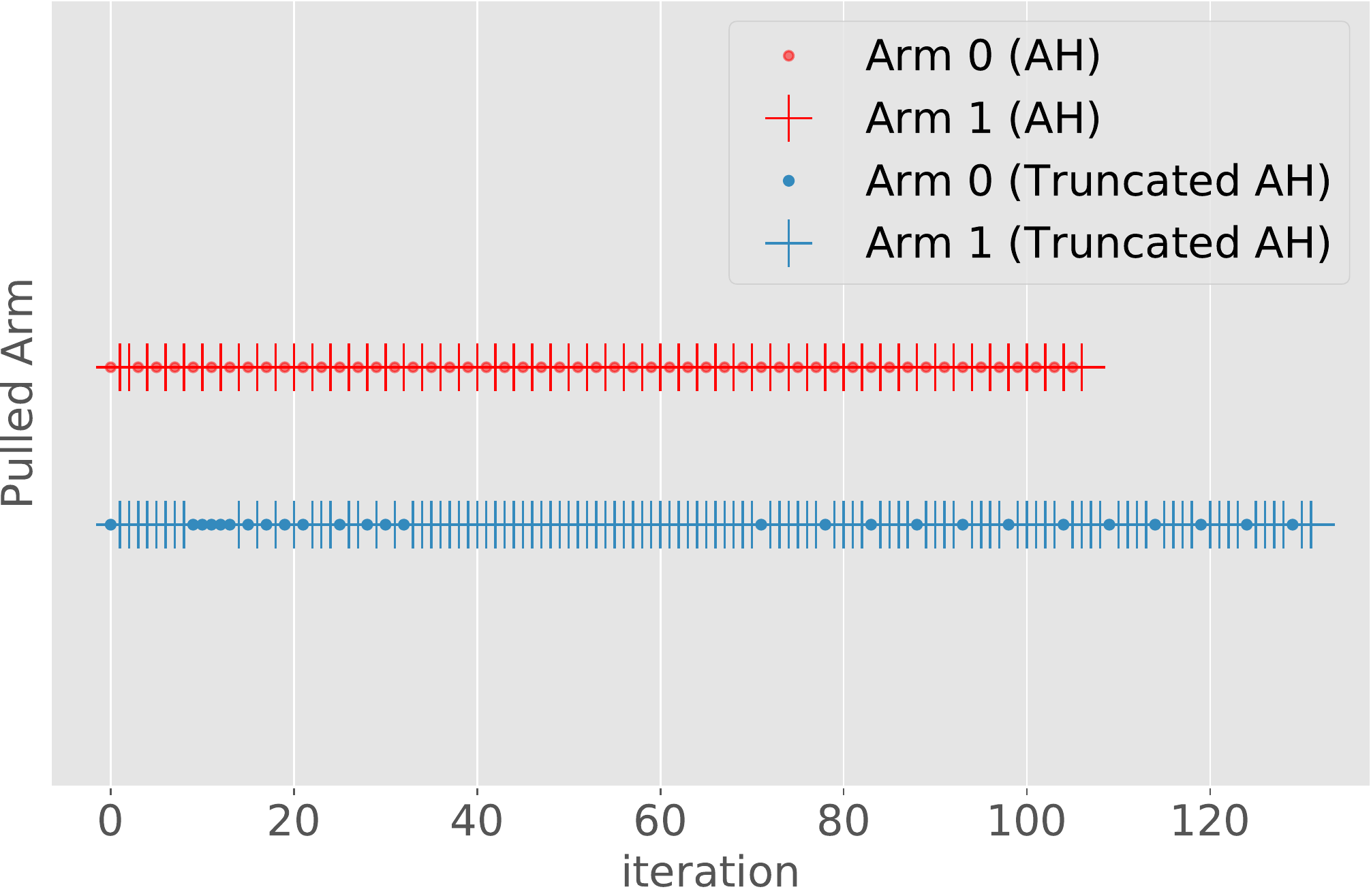}
    \caption{The arms pulled at each iteration. `$\cdot$': Arm 0 is pulled. `+': Arm 1 is pulled. \ah and truncated \ah are marked in 
    red and blue, respectively. }
\label{fig:best_arm_pulled_arm_ah}
\end{subfigure}
\caption{Identify the best arm out of two using \ah and its truncated variant.}
\label{fig:best_arm_ah}

\end{figure}

We first explain this phenomenon using \ah and its truncated variant. 
\ah can result in confidence intervals that are larger than $[0, 1]$. In the truncated version of \ah, the upper confidence term of a confidence interval 
will be capped at $1$, and the lower confidence term will be cut at $0$, so that all the confidence intervals stay in $[0, 1]$ throughout the experiment.  We shall see that the truncated variant
would result in stationary radius and yield larger sample complexity compared with \ah.

Figures~\ref{fig:best_arm_confseq_ah} and~\ref{fig:best_arm_confseq_ah_truncated} show the confidence intervals of each arm at each iteration, when \ah and truncated \ah are plugged into Algorithm~\ref{alg:bestarm}. 
The algorithm will stop when the confidence intervals of the two arms completely separate (i.e., the lower bound of Arm 0 goes above the upper bound of Arm 1). Figure~\ref{fig:best_arm_confseq_ah} and~\ref{fig:best_arm_confseq_ah_truncated} show that
\ah used 107 iterations, while the truncated \ah used 132 iterations. 
One can observe that in the initial stage of the algorithm, the confidence interval, without truncation, will likely get updated once a sample adds in, which does not hold for the truncated version; compare the first 15 iterations in Figures~\ref{fig:best_arm_confseq_ah} and~\ref{fig:best_arm_confseq_ah_truncated}.
Therefore, the radius will not get updated for truncated \ah, as shown in Figure~\ref{fig:best_arm_radius_ah.pdf}. 
Recall that  Algorithm~\ref{alg:bestarm} samples an arm with largest radius; when both radii are same, we sample the arm with smaller empirical mean. 
Due to the stationary radius,  in those iterations, truncated \ah keeps sampling the same arm 
till an update happens. 

In Figure~\ref{fig:best_arm_radius_diff_ah}, we plot the difference between the radius for Arm 0 and Arm 1: $R_0 - R_1$.  Arm 0 will be sampled if this value is positive and vice versa. Again, if $R_0$ is equal to $R_1$, we shall sample the arm with lower empirical mean.
We can see the difference fluctuates evenly for \ah, so that \ah almost alternatively samples each arm, and the confidence intervals of both arms gets updated alternatively as shown in Figure~\ref{fig:best_arm_confseq_ah}. In contrast,
for truncated \ah, the difference consistently stays above or below zero for some time, which means the same arm gets sampled. See Figure~\ref{fig:best_arm_pulled_arm_ah} for the arms pulled at each iteration; the `+' and `$\cdot$' appear almost side-by-side with \ah and they appear disproportionately with truncated \ah. 

As mentioned, Algorithm~\ref{alg:bestarm} stops when the two confidence intervals separate, and it is not crucial for those intervals to be shorter.
Hence, it will stop fast if (i) the confidence interval gets updated by every sample and (ii) the updates are significant for 
small number of samples (the early stage). Truncated \ah underperforms in both aspects. This is also the reason why the Berstein type of confidence sequences underperforms \ah in this problem (c.f. Section~\ref{sec:bandit}). Even though they are shorter for larger samples; \ah is better with smaller samples.

Next, we investigate the performance for Bentkus type of methods. 
We write \bk to be the variant from Section~\ref{sec:bandit}, that is,
we output confidence interval $\{[\mu_n^{\low*}, \mu_n^{\up*}], n\ge 1\}$ as in Theorem~\ref{thm:empirical-Bentkus-arbitrary-mean},
but output radius $R_n = \mu_n^{\up} - \mu_n^{\low}$.

We write original \bk to be the one directly from Theorem~\ref{thm:empirical-Bentkus-arbitrary-mean}, i.e.,
we output confidence interval $\{[\mu_n^{\low*}, \mu_n^{\up*}], n\ge 1\}$ and radius $R_n = \mu_n^{\up*} - \mu_n^{\low*}$.
Note that $\mu_n^{\up*} = \min_{1 \leq i \leq n} \mu_i^{\up}$ is the cumulative minimum, which essentially serves as the truncation
of the upper confidence term, and similarly does the $\mu_n^{\low*}$.
We refer the readers to Theorem~\ref{thm:empirical-Bentkus-arbitrary-mean} for the details. 
Similar to the previous experiment, we shall see
that the original \bk results in a larger sample complexity than \bk. Figure~\ref{fig:best_arm_confseq_bk} presents the results.

 \begin{figure}[H]
 \centering
 \begin{subfigure}{.49\textwidth}
     \centering
     \includegraphics[width=0.95\textwidth]{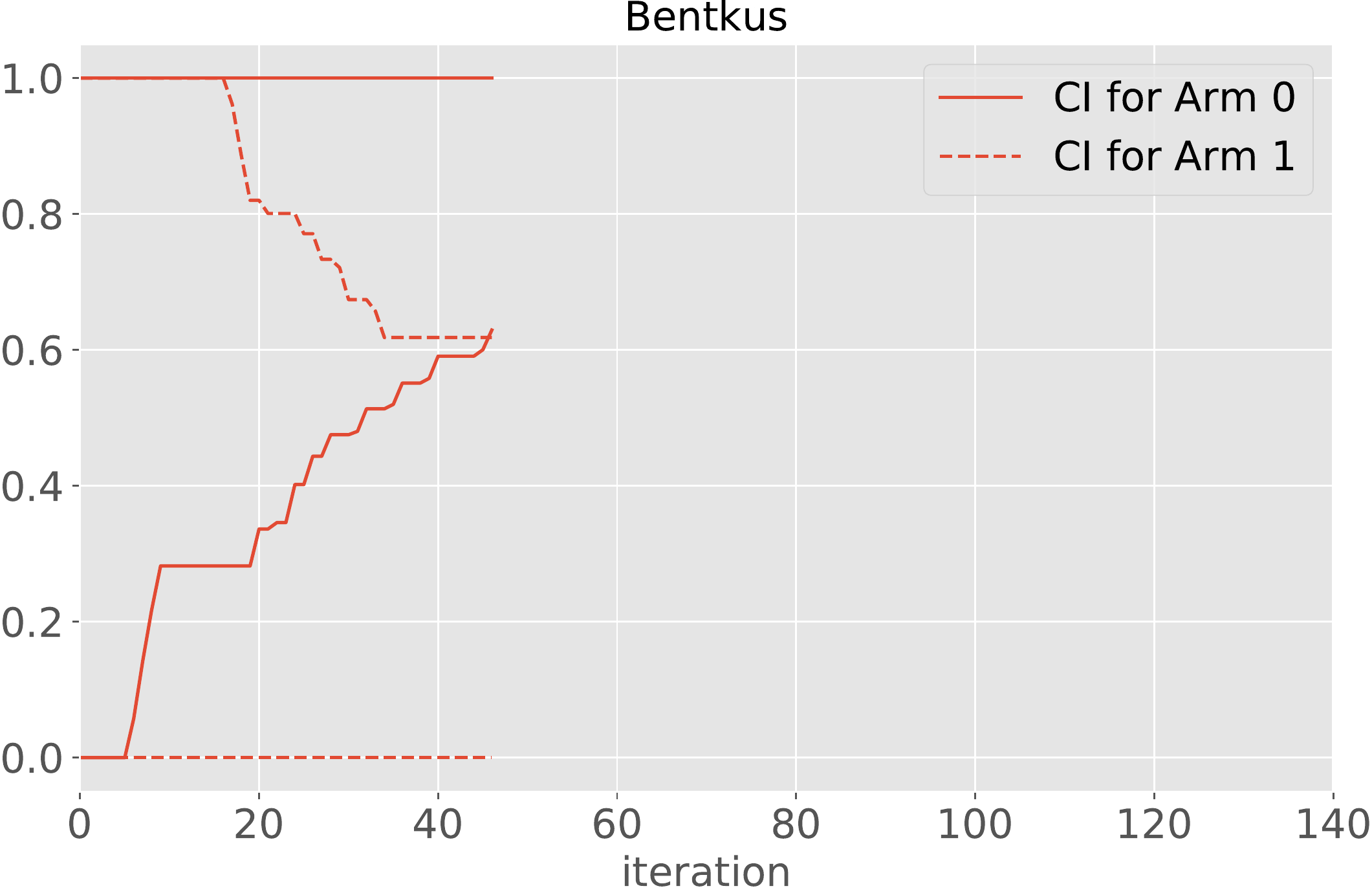}
     \caption{Confidence intervals of \bk (variant) for two arms.}
     \label{fig:best_arm_confseq_bk}
 \end{subfigure}
 \begin{subfigure}{.49\textwidth}
     \centering
     \includegraphics[width=0.95\textwidth]{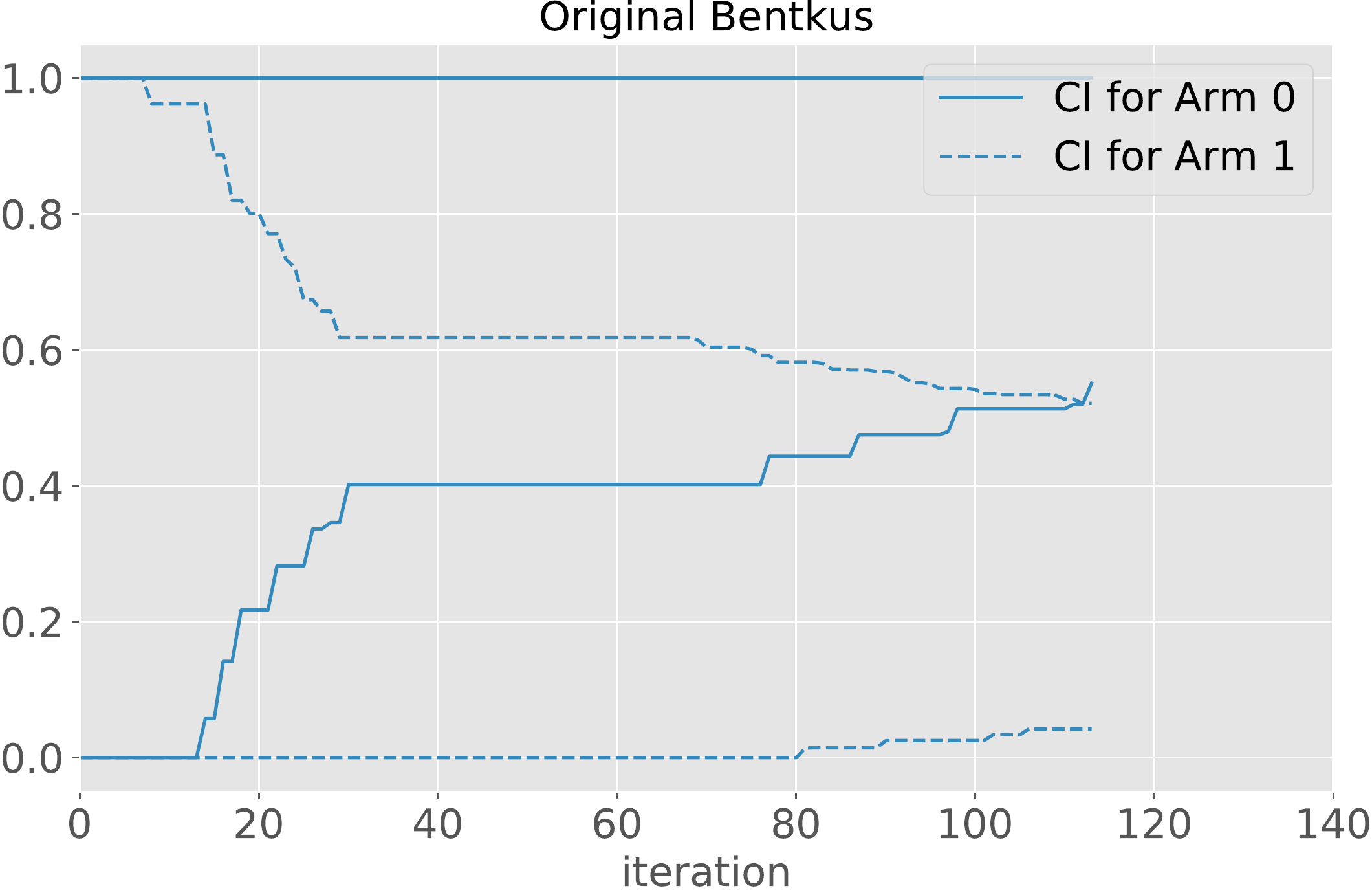}
     \caption{Confidence intervals of \bk (original) for two arms.}
     \label{fig:best_arm_best_arm_confseq_bk_truncated}
 \end{subfigure}
 
 \begin{subfigure}{.49\textwidth}
 \centering
     \includegraphics[width=0.95\textwidth]{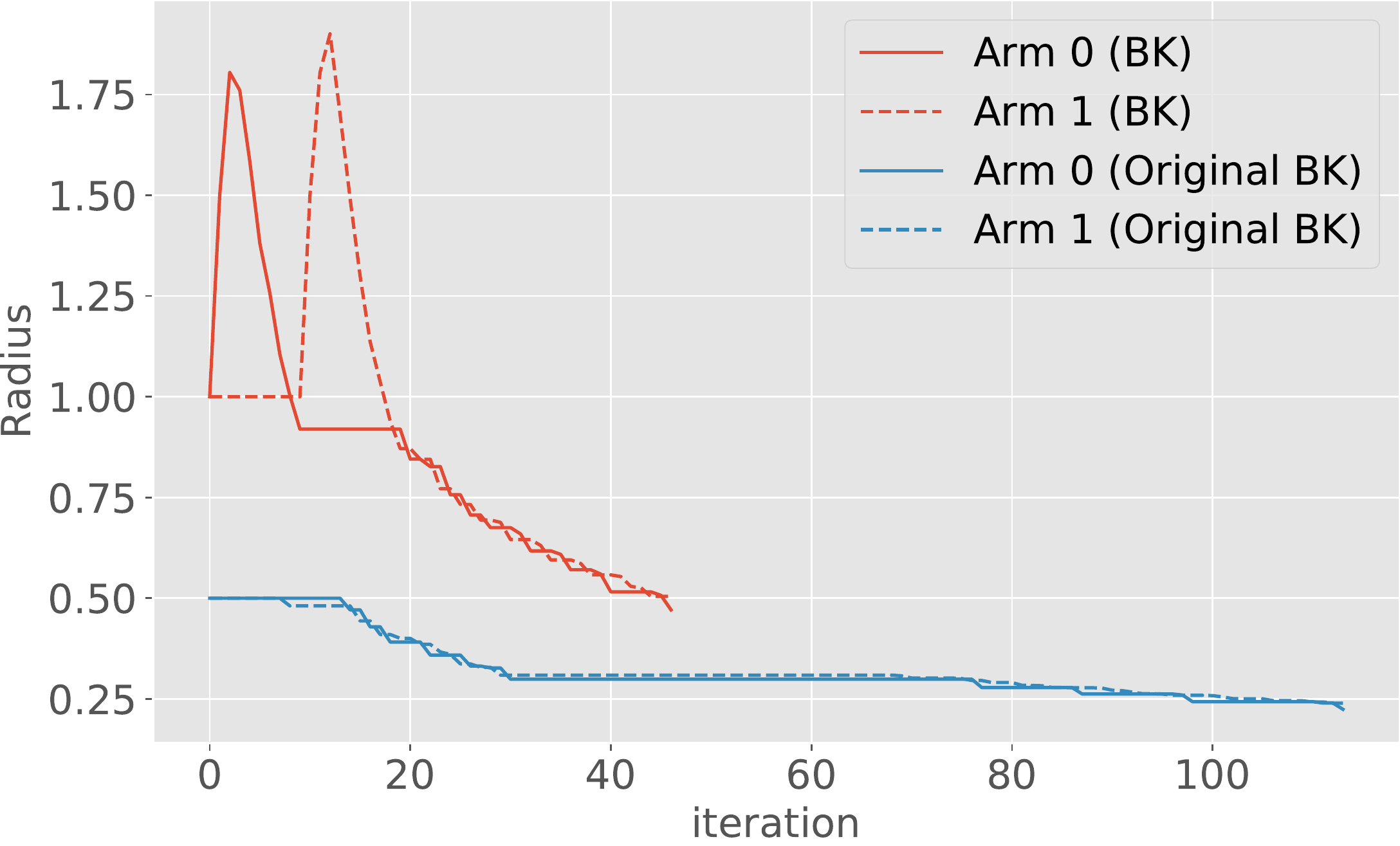}
     \caption{Radius of \bk (original and variant) for two arms.}
     \label{fig:best_arm_raidus_bk}
 \end{subfigure}
 \begin{subfigure}{.49\textwidth}
     \centering
     \includegraphics[width=0.95\textwidth]{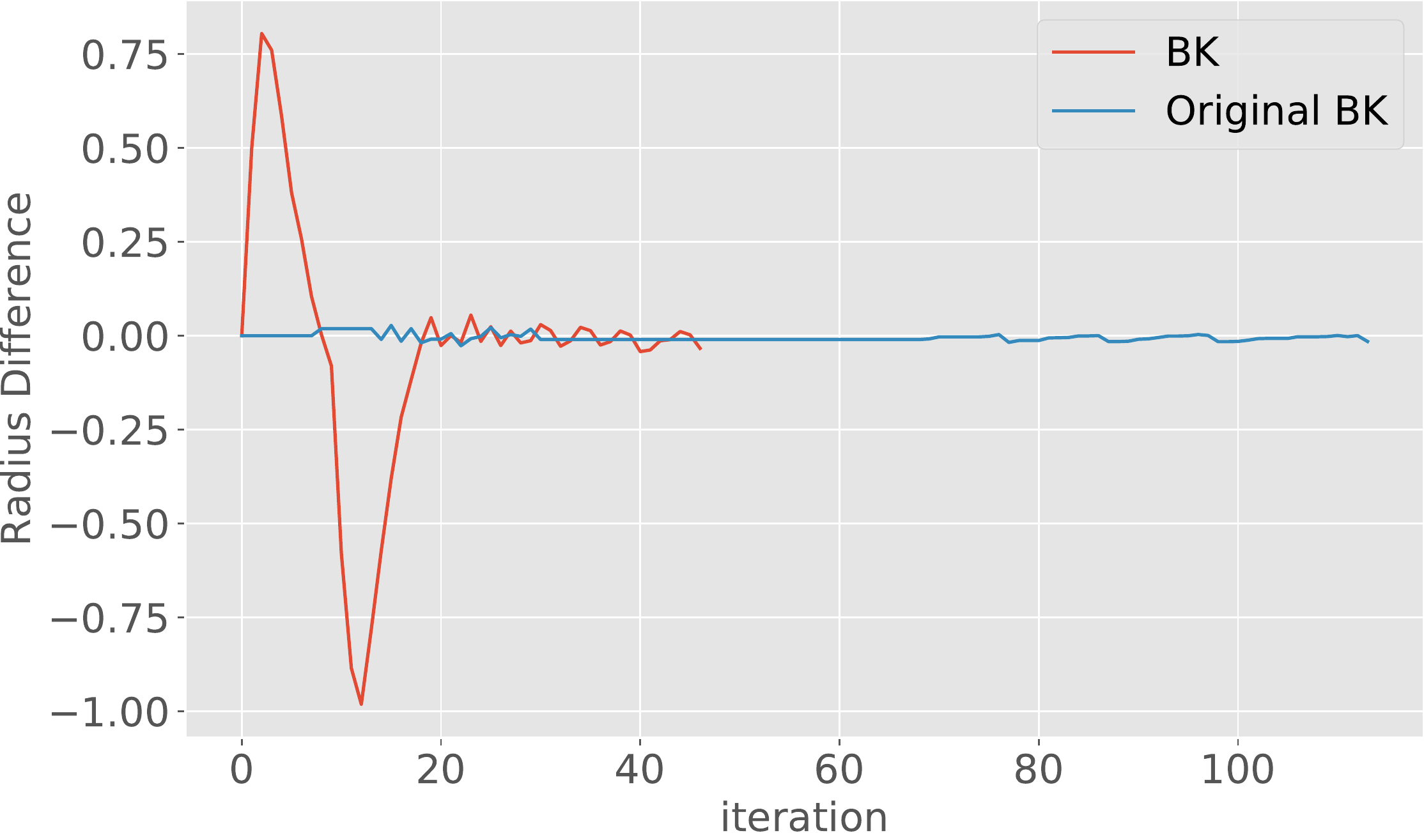}
     \caption{The difference of the radius for two arms: $R_0 - R_1$. For positive difference value, Arm 0 will be pulled. For negative difference
    value, Arm 1 will be pulled.}
     \label{fig:best_arm_raidus_diff_bk}
 \end{subfigure}
 
 \begin{subfigure}{.49\textwidth}
     \centering
     \includegraphics[width=0.95\textwidth]{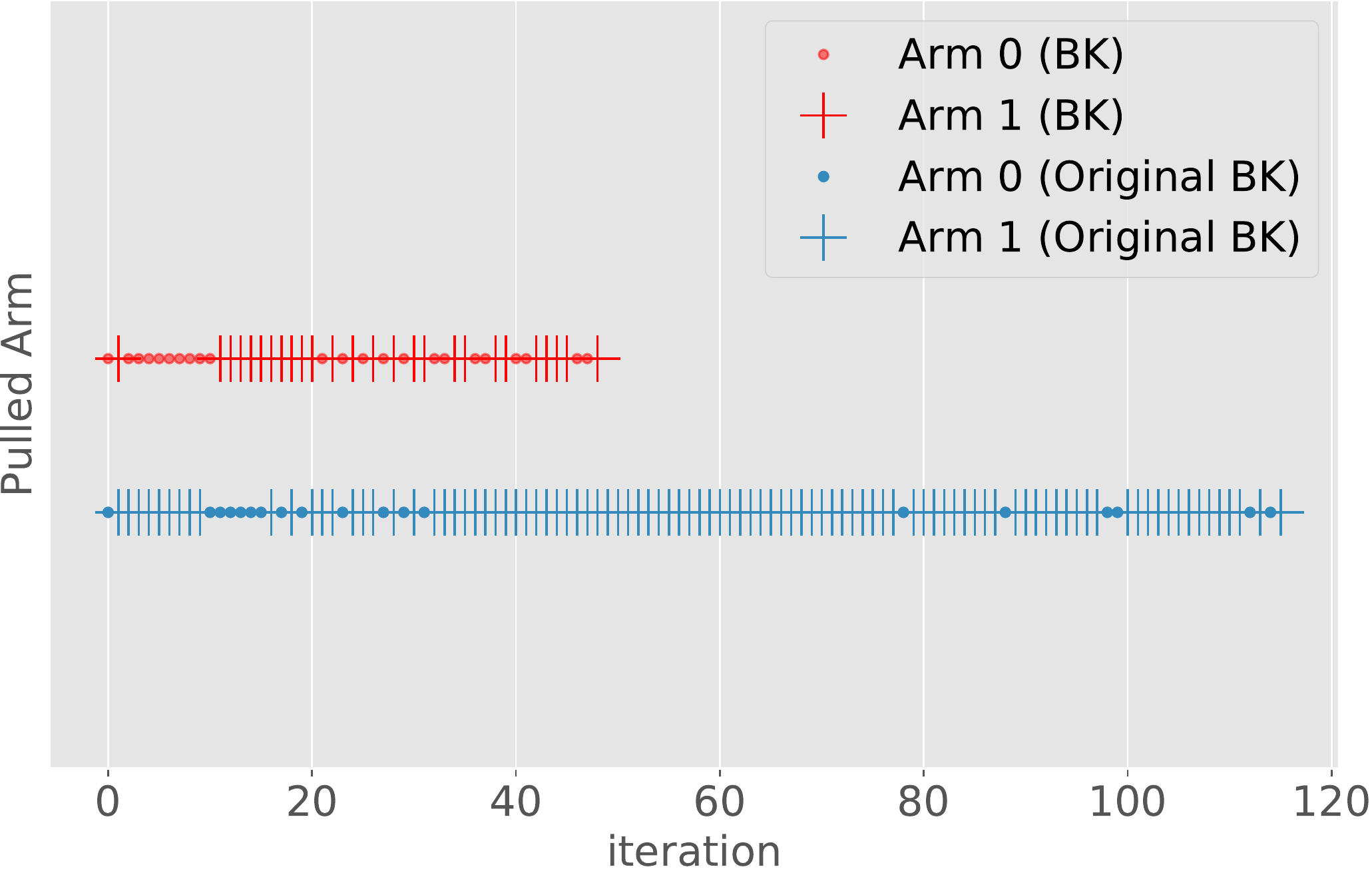}
     \caption{The arms pulled at each iteration. `$\cdot$': Arm 0 is pulled. `+': Arm 1 is pulled. \bk and original \bk are marked in 
   red and blue, respectively. }
     \label{fig:best_arm_pulled_arm_bk}
 \end{subfigure}
 \caption{Identify the best arm out of two using original \bk and the variant introduced in Section~\ref{sec:bandit}.}
 \label{fig:best_arm_bk}
\end{figure}
%


Patterns similar to the \ah and its truncated version happen here too. Although \bk keeps sampling the same arm in the beginning phase, it alternates the samples in the later stage. Comparing Figures~\ref{fig:best_arm_pulled_arm_ah} (\ah) and~\ref{fig:best_arm_pulled_arm_bk} (\bk), the sampling pattern of \ah is more uniform, however, \bk still outperforms \ah due to its fast convergence.

\section{Computation of $q(\delta; n, \mathcal{A}, B)$}\label{appsec:computation-of-q-function}
In this section we provide some details on the computation of $q(\delta; n, \mathcal{A}, B)$ based on~\citet{bentkus2004hoeffding} and~\citet{pinelis2009bennett}. We will restrict to the case where $A_1 = A_2 = \cdots = A_n = \cdots = A$.

For any random variable $\eta$, define
\[
P_2(u; \eta) ~:=~ \inf_{x \le u}\frac{\mathbb{E}[(\eta - x)_+^2]}{(u - x)_+^2}.
\]
For any $A, B$, set $\pab = {A^2}/{(A^2+B^2)}$.
Define Bernoulli random variables $R_1, R_2, \ldots, R_n$ as
\[
    \mathbb{P}(R_i = 1) ~=~ \pab ~=~ 1 - \mathbb{P}(R_i = 0).
\]
Set $Z_n = \sum_{i=1}^n R_i$. $Z_n$ is a binomial random variables
with $n$ trials and success probability $\pab$: $Z_n \sim \mbox{Bi}(n, \pab)$.
For $0 \le k\le n$, define
\begin{align*}
    p_{k} ~&:=~ \mathbb{P}\left(Z_n \ge k\right),\quad e_{k} ~:=~ \mathbb{E}\left[Z_n\mathds{1}\left\{Z_n \ge k\right\}\right],\quad v_{k} ~:=~ \mathbb{E}\left[Z_n^2\mathds{1}\left\{Z_n \ge k\right\}\right].
\end{align*}
\begin{prop}
For all $u\in\mathbb{R}$,
\[
    P_2\left(u; \sum_{i=1}^n G_i\right) ~=~ P_2\left(\frac{Bu + nA^2}{A^2 + B^2};\, Z_n \right) ~=~ P_2\left(\frac{Bu + nA^2}{A^2 + B^2}; Z_n\right).
\]
Furthermore, for any $x\ge 0$ and $1\le k \le n-1$,
\[
P_2\left(x; Z_n\right) ~=~
\begin{cases}
    1,  &   \mbox{if }x ~\le~ n \pab,\\
    \frac{n\pab(1 - \pab)}{(x - n\pab)^2 + n\pab(1 - \pab)}, &\mbox{if }n \pab < x \le \frac{v_0}{e_0},\\
    \frac{v_{k}p_{k} - e_{k}^2}{x^2p_{k} - 2xe_{k} + v_{k}}, &\mbox{if }\frac{v_{k-1} - (k-1)e_{k-1}}{e_{k-1} - (k-1)p_{k-1}} < x \le \frac{v_k - ke_k}{e_{k} - kp_k},\\
    \mathbb{P}\left(Z_n = n\right) = \pab^n, &\mbox{if }x~\ge~ \frac{v_{n-1} - (n-1)e_{n-1}}{e_{n-1} - (n-1)p_{n-1}} = n.
\end{cases}
\]
Formally, we can set $P_2(x; Z_n) = 0$ for all $x > n$ because $\mathbb{P}(Z_n > n) = 0$.
\end{prop}

\begin{proof}
The result is mostly an implication of Proposition 3.2 of~\citet{pinelis2009bennett}.
It is clear that
\[
M_n := \sum_{i=1}^n G_i ~\overset{d}{=}~ \frac{A^2 + B^2}{B}\left(\sum_{i=1}^n R_i - \frac{nA^2}{A^2 + B^2}\right),
\]
where $R_i\sim\mathrm{Bernoulli}(A^2/(A^2 + B^2))$, that is,
\[
    \mathbb{P}\left(R_i = 1\right) ~=~ \pab ~=~ 1 - \mathbb{P}(R_i = 0).
\]
Proposition 3.2(vi) of~\citet{pinelis2009bennett} implies that
\[
    P_2(u;\, M_n) ~:=~ P_2\left(\frac{Bu + nA^2}{A^2 + B^2};\, Z_n \right).
\]
Hence it suffices to find $P_2(x;\,Z_n)$ for all $x\in\mathbb{R}$. The support of $Z_n$ is given by
\[
\mathrm{supp}(Z_n) = \{0, 1, 2, \ldots, n\}.
\]
Proposition 3.2(iv) of~\citet{pinelis2009bennett} (with $\alpha = 2$) implies that
\[
    P_2\left(x; Z_n \right) ~=~
\begin{cases}
    1,&\mbox{if }x \le n \pab,\\
    \mathbb{P}\left(Z_n = n\right), &\mbox{if }x\ge n.
\end{cases}
\]
Furthermore, $x\mapsto P_2(x; \sum_{i=1}^n R_i)$ is strictly decreasing on $(n \pab,
n)$. Define function $F(h): \R \rightarrow \R$ such that
\begin{equation}\label{eq:definition-F_h}
F(h) ~:=~ \frac{\mathbb{E}[Z_n\left(Z_n - h\right)_+]}{\mathbb{E}\left(Z_n -
        h\right)_+}.
\end{equation}
For any $n \pab < x < n$, let $h_x$ be the unique solution of
\begin{equation}\label{eq:root-solving}
    F(h)~=~ x
\end{equation}
(Uniqueness here is established by Proposition 3.2(ii) of~\citet{pinelis2009bennett}.) Then by Proposition 3.2(iii) of~\citet{pinelis2009bennett},
\begin{equation}\label{eq:P_2-based-on-h-x}
\begin{split}
      P_2\left(x; Z_n\right) ~&=~ \frac{\mathbb{E}[(Z_n - h_x)_+^2]}{(x - h_x)_+^2}\\
~&=~ \frac{\mathbb{E}[Z_n\left(Z_n - h_x\right)_+] - h_x\mathbb{E}[\left(Z_n - h_x\right)_+]}{(x - h_x)_+^2}\\
~&=~ \frac{(x - h_x)\mathbb{E}[\left(Z_n - h_x\right)_+]}{(x - h_x)_+^2}\\
~&=~ \frac{\mathbb{E}[\left(Z_n - h_x\right)_+]}{(x - h_x)_+}.
\end{split}
\end{equation}
This holds for all $nA^2/(A^2 + B^2) < x < n$. We will now discuss solving~\eqref{eq:root-solving}.

Proposition 3.2(i) of~\citet{pinelis2009bennett} implies that $h\mapsto F(h)$ is
continuous and increasing. 

If $h\le0$,
\[
F(h)
= \frac{\mathbb{E}[Z_n(Z_n - h)]}{\mathbb{E}[Z_n - h]}
= \frac{n\pab(1 - \pab) + n^2\pab^2 - hn\pab}{n\pab - h} = n\pab + \frac{np(1-\pab)}{np-h}.
\]
This is strictly increasing on $(-\infty, 0]$, and $F(0) = n\pab + (1 - \pab)$.
We get that for any $n\pab < x \leq n\pab + (1 - \pab)$,
\[
F(h) = x \quad\Leftrightarrow\quad h_x = n\pab - \frac{n\pab(1 - \pab)}{x - n\pab}.
\]
This further implies (from~\eqref{eq:P_2-based-on-h-x}) that
\begin{align*}
P_2(x; Z_n) ~&=~ \frac{\mathbb{E}[Z_n - h_x]}{x - h_x}\\
~&=~ \frac{n\pab(1 - \pab)}{(x - n\pab)^2 + n\pab(1 - \pab)},\quad\mbox{for}\quad n\pab \le x \le n\pab + (1 - \pab).
\end{align*}

If $0 < h < n-1$, set $k = \ceil{h}$, in other words, $k -1 < h \leq k$.
Since $\{Z_n \ge h\} \Leftrightarrow \{Z_n \ge k\}$, hence
\begin{align*}
   \mathbb{E}[Z_n\left(Z_n - h\right)_+]
&= \mathbb{E}[Z_n^2\mathds{1}\{Z_n \ge h\}] - h\mathbb{E}[Z_n\mathds{1}\{Z_n \ge h\}]\\
&= \mathbb{E}[Z_n^2\mathds{1}\{Z_n \ge k\}] - h\mathbb{E}[Z_n\mathds{1}\{Z_n \ge k\}],\\
    \mathbb{E}[(Z_n - h)_+]
&= \mathbb{E}[Z_n\mathds{1}\{Z_n \ge k\}] - h\mathbb{P}(Z_n \ge k).
\end{align*}
Therefore,
\begin{align*}
F(h) ~&=~ \frac{\mathbb{E}[Z_n^2\mathds{1}\{Z_n \ge k\}] - h\mathbb{E}[Z_n\mathds{1}\{Z_n \ge k\}]}{\mathbb{E}[Z_n\mathds{1}\{Z_n \ge k\}] - h\mathbb{P}(Z_n \ge k)}\\
~&=~ \frac{v_k - he_k}{e_k - hp_k}.
\end{align*}

It is not difficult to verify that $F(\cdot)$ is strictly increasing in $(k-1, k]$ and hence
\[
    h_x ~=~ \frac{v_k - xe_k}{e_k - xp_k}, \quad\mbox{if}\quad F(k-1) < x \le F(k).
\]
Substituting this $h_x$ in~\eqref{eq:P_2-based-on-h-x} yields the value of $P_2(x;\,Z_n)$, that is,
\begin{align*}
P_2(x;\,Z_n) ~&=~ \left(x - \frac{v_{k} - xe_{k}}{e_{k} - xp_{k}}\right)^{-1}\left(e_{k} - \frac{v_{k} - xe_{k}}{e_{k} - xp_{k}}p_{k}\right)\\
~&=~ \left(\frac{e_{k} - xp_{k}}{2xe_{k} - x^2p_{k} - v_{k}}\right)\left(\frac{e_{k}^2 - v_{k}p_{k}}{e_{k} - xp_{k}}\right)\\
~&=~ \frac{e_{k}^2 - v_{k}p_{k}}{2xe_{k} - x^2p_{k} - v_{k}},\quad\mbox{whenever}\quad
F(k-1) < x \le F(k),
\end{align*}
where $F(k) = \frac{v_k - k e_k}{e_k - k p_k}$, $1 \le k \le n-1$.
Hence for $1\le k \le n-1$,
\[
P_2(x; Z_n) ~=~ \frac{v_{k}p_{k} - e_{k}^2}{x^2p_{k} - 2xe_{k} + v_{k}},\quad\mbox{whenever}\quad \frac{v_{k-1} - (k-1)e_{k-1}}{e_{k-1} - (k-1)p_{k-1}} < x \le \frac{v_k - ke_{k}}{e_k - kp_k}.
\]

Finally, we prove that $F(\cdot)$ is a constant on $[n-1, n]$. It is clear that
\begin{align*}
F(n-1) &=  \frac{v_{n-1} - (n-1)e_{n-1}}{e_{n-1} - (n-1)p_{n-1}}\\
&= \frac{\mathbb{E}[Z_n^2\mathds{1}\{Z_n \ge n-1\}] - (n-1)\mathbb{E}[Z_n\mathds{1}\{Z_n \ge n-1\}]}{\mathbb{E}[Z_n\mathds{1}\{Z_n \ge n-1\}] - (n-1)\mathbb{P}(Z_n \ge n-1)}\\
&= \frac{(n^2 - n(n-1))\mathbb{P}(Z_n = n)}{(n - (n-1))\mathbb{P}(Z_n = n)} = n.
\end{align*}
Further if $h > n-1$, then $(Z_n - h)_+ > 0$ if and only if $Z_n = h$ and hence from~\eqref{eq:definition-F_h}
\[
F(h) = \frac{\mathbb{E}[Z_n(Z_n - h)_+]}{\mathbb{E}[(Z_n - h)_+]} = \frac{n(n - h)\mathbb{P}(Z_n = n)}{(n-h)\mathbb{P}(Z_n = n)} = n.
\]
Therefore, the function $F(h)$ is constant on $[n-1, n]$.

For $h > n$, we set $F(h) = n$ since $\mathbb{P}(Z_n > h) = 0$. To put all the pieces together, we have
\begin{equation}
    \label{eq:Fh_final}
    F(h) =
    \begin{cases}
        n\pab + \frac{np(1-\pab)}{np-h} & \quad\mbox{if} \quad h <= 0, \\
        \dfrac{v_\ceil{h} - h e_\ceil{h}}{e_\ceil{h} - hp_\ceil{h}} & \quad\mbox{if} \quad 0 < h \le n-1,\\
        n & \quad\mbox{if} \quad  h > n-1.
    \end{cases}
\end{equation}
Consequently, for $n\pab < x < n$,
\begin{equation}
    \label{eq:Fh_inv_final}
    h_x = F^{-1}(x) =
    \begin{cases}
        n\pab - \frac{n\pab(1 - \pab)}{x - n\pab}, & \quad\mbox{if} \quad n\pab < x \leq n\pab + (1 - \pab), \\
        \frac{v_k - xe_k}{e_k - xp_k}, & \quad\mbox{if}\quad F(k-1) < x \le F(k), \,1 \leq k \leq n-1. \\
    \end{cases}
\end{equation}
\end{proof}
As a graphical example, Figure~\ref{fig:function_F_and_P2} plots $F(h)$ and $P_2(x; Z_n)$ when $n=3$, $A=0.1$ and
        $B=1.0$.
        
\begin{figure}[t]
    \centering
    \includegraphics[width=0.32\textwidth]{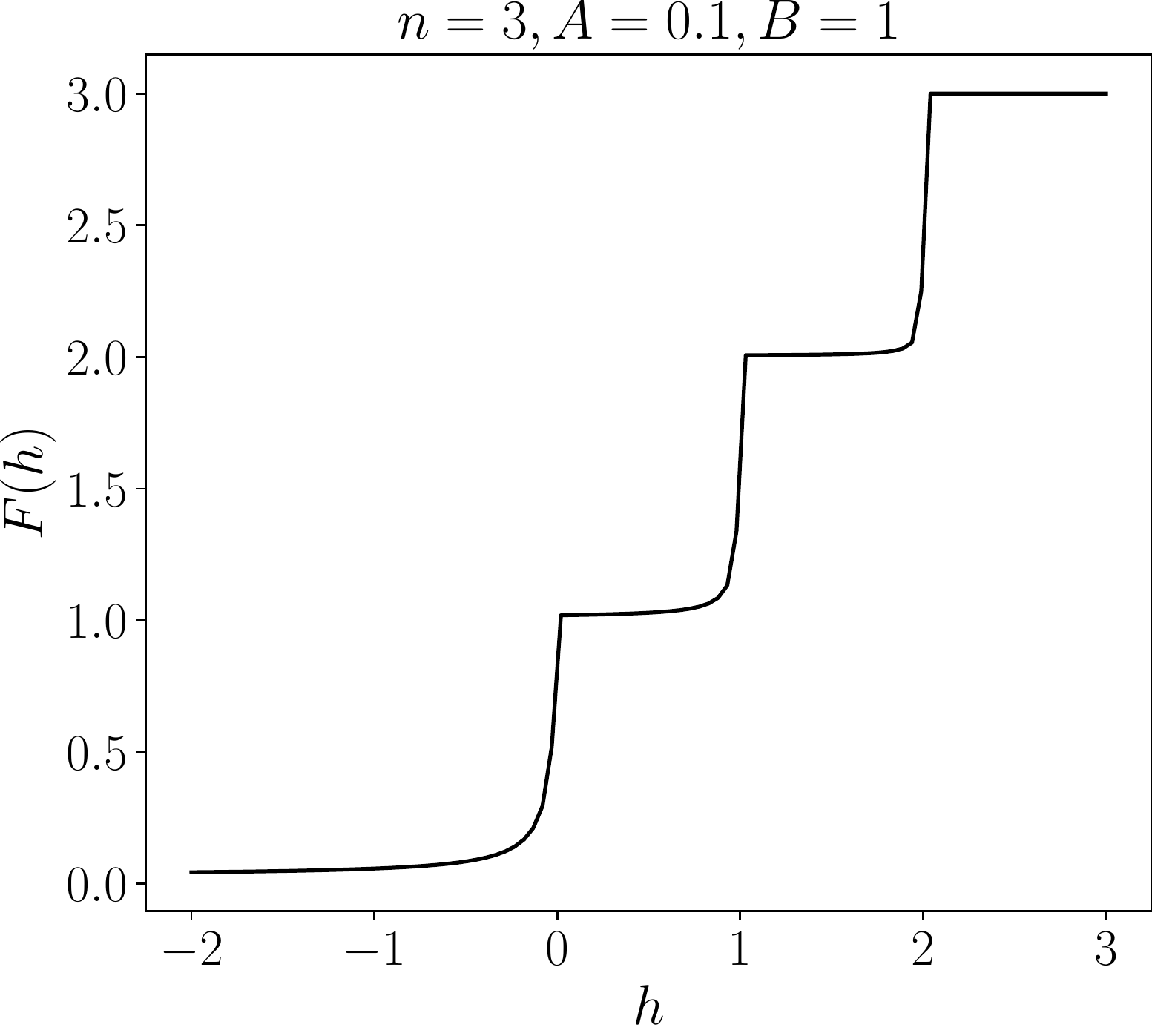}
    \hspace{0.01\textwidth}
    \includegraphics[width=0.64\textwidth]{function_P2_n3}
    \caption{Examples functions $F(h)$ and $P_2(x; Z_n)$ when $n=3$, $A=0.1$ and
        $B=1.0$. We plot $P_2(x;Z_n)$ in both linear (second plot) and log (third plot) scales on the y-axis.}
    \label{fig:function_F_and_P2}
\end{figure}

\subsection{Computation of the Quantile}\label{subsec:computation-of-quantile}
Recall $\pab = A^2/(A^2 + B^2), Z_n = \sum_{i=1}^n R_i,$ and $\sum_{i=1}^nG_i$ is identically distributed as $B^{-1}(A^2 + B^2)(Z_n - n\pab)$. We will compute $x_{\delta}$ such that
\begin{equation}\label{eq:quantile-of-Zn}
P_2(x_{\delta}; Z_n) = \delta.
\end{equation}
This implies that
\[
P_2\left(\frac{(A^2 + B^2)x_{\delta} - nA^2}{B};\,\sum_{i=1}^n G_i\right) = \delta,\quad\mbox{or equivalently,}\quad q(\delta; n, A, B) = \frac{(A^2 + B^2)x_{\delta} - nA^2}{B}.
\]
Hence we concentrate on solving~\eqref{eq:quantile-of-Zn}. Recall that for any $x\ge 0$ and $1\le k \le n-1$,
\begin{equation}\label{eq:recall-p2}
P_2\left(x; Z_n\right) ~=~
\begin{cases}
    1,  &   \mbox{if }x ~\le~ n \pab,\\
    \frac{n\pab(1 - \pab)}{(x - n\pab)^2 + n\pab(1 - \pab)}, &\mbox{if }n \pab < x \le \frac{v_0}{e_0} = n\pab + (1 - \pab),\\
    \frac{v_{k}p_{k} - e_{k}^2}{x^2p_{k} - 2xe_{k} + v_{k}}, &\mbox{if }\frac{v_{k-1} - (k-1)e_{k-1}}{e_{k-1} - (k-1)p_{k-1}} < x \le \frac{v_k - ke_k}{e_{k} - kp_k},\\
    \mathbb{P}\left(Z_n = n\right) = \pab^n, &\mbox{if }x~\ge~ \frac{v_{n-1} - (n-1)e_{n-1}}{e_{n-1} - (n-1)p_{n-1}} = n.
\end{cases}
\end{equation}
The function $P_2(\cdot; Z_n)$ is a non-increasing function and hence if $\delta \le \pab^n$, then we get $x_{\delta} = n + 10^{-8}$; this corresponds to the last case in~\eqref{eq:recall-p2}. If $P_2(v_0/e_0; Z_n) \le \delta \le 1$, then
\[
x_{\delta} = n\pab + \sqrt{\frac{(1 - \delta)n\pab(1-\pab)}{\delta}};
\]
this corresponds to the first and second case in~\eqref{eq:recall-p2}; note that $P_2(v_0/e_0; Z_n) = n\pab(1-\pab)/[(1-\pab)^2 + n\pab(1-\pab)]$.
For the remaining cases, note that if
there exists a $1 \le k\le n-1$ such that
\[
P_2\left(\frac{v_k - ke_k}{e_{k} - kp_k}; Z_n\right) ~\le~ \delta ~\le~ P_2\left(\frac{v_{k-1} - (k-1)e_{k-1}}{e_{k-1} - (k-1)p_{k-1}}; Z_n\right),
\]
then
\begin{equation}\label{eq:finding-k}
\frac{v_{k-1} - (k-1)e_{k-1}}{e_{k-1} - (k-1)p_{k-1}} ~\le~ x_{\delta} ~\le~ \frac{v_k - ke_k}{e_{k} - kp_k},
\end{equation}
and using the closed form expression of $P_2(\cdot; Z_n)$ on this interval, we get
\begin{equation}\label{eq:final-formula-x-delta}
x_{\delta} = \frac{e_k + \sqrt{e_k^2 - p_k(v_k - (v_kp_k - e_k^2)/\delta)}}{p_k}.
\end{equation}
Using these calculations, one can find $k$ looping over $1\le k\le n-1$ such that~\eqref{eq:finding-k} holds. This approach has a complexity of $O(n)$, assuming the availability of $p_k, e_k,$ and $v_k$.

We now describe an approach that reduces the complexity by finding quick-to-compute upper and lower bounds on $x_{\delta}$. Lemmas 1.1 and 3.1 of~\citet{bentkus2006domination} show that
\begin{equation}\label{eq:two-bounds-P2}
\mathbb{P}(Z_n \ge x) \le P_2(x; Z_n) \le \frac{e^2}{2}\mathbb{P}^{\circ}(Z_n \ge x),
\end{equation}
where $\mathbb{P}^{\circ}(Z_n \ge x)$ represents the log-linear interpolation of $P(Z_n \ge x)$, that is, for $x \in \{0, 1, \ldots, n\}$
\begin{equation}\label{eq:at-lattice-points}
\mathbb{P}^{\circ}(Z_n \ge x) ~=~ \mathbb{P}(Z_n \ge x),
\end{equation}
and for $x\in(k-1, k)$ such that $x = (1 - \lambda)(k-1) + \lambda k$,
\[
\mathbb{P}^{\circ}(Z_n \ge x) = (\mathbb{P}(Z_n \ge k-1))^{1-\lambda}(\mathbb{P}(Z_n \ge k))^{\lambda}.
\]
Equation (2) of~\citet{bentkus2002remark} further shows that
\begin{equation}\label{eq:linear-interpolation}
\mathbb{P}^{\circ}(Z_n \ge x) \le (1 - \lambda)\mathbb{P}(Z_n \ge k-1) + \lambda\mathbb{P}(Z_n \ge k).
\end{equation}
Hence, to find $x = x_{\delta}$ satisfying $P_2(x; Z_n) = \delta$, find $k_1\in\{0, 1,\ldots, n\}$ such that
\[
\mathbb{P}(Z_n \ge k_1) \ge \delta.
\]
This implies (from~\eqref{eq:two-bounds-P2}) that $P_2(k_1; Z_n) \ge \delta$ and because $x\mapsto P_2(x; Z_n)$ is decreasing, $x_{\delta} \ge k_1$. Further, find $k_2\in\{0, 1, \ldots, n\}$ such that
\[
\mathbb{P}(Z_n \ge k_2) \le {2\delta}/{e^2}.
\]
This implies (from~\eqref{eq:linear-interpolation}) that $\mathbb{P}^o(Z_n \ge k_2) = \mathbb{P}(Z_n \ge k_2) \le 2\delta/e^2$. Hence using~\eqref{eq:at-lattice-points}, we get $P_2(k_2; Z_n) \le \delta$ which implies that $x_{\delta} \le k_2$. Summarizing this discussion, we get that $x_{\delta}$ satisfying $P_2(x_\delta; Z_n) = \delta$ also satisfies
\begin{equation}\label{eq:initial-bounds}
k_1 \le x_{\delta} \le k_2,
\end{equation}
where
\[
\mathbb{P}(Z_n \ge k_1) \ge \delta\quad\mbox{and}\quad \mathbb{P}(Z_n \ge k_2) \le {2\delta}/{e^2}.
\]
The bounds in~\eqref{eq:initial-bounds} are not very useful because the closed form experssion~\eqref{eq:final-formula-x-delta} of $x_{\delta}$ requires finding upper and lower bounds for $x_{\delta}$ in terms of $(v_k - ke_k)/(e_k - kp_k)$'s.

Now we note that
\[
v_{k} \ge ke_{k} \ge k^2p_{k}\quad\Rightarrow\quad \frac{v_{k_2} - k_2e_{k_2}}{e_{k_2} - k_2p_{k_2}} \ge k_2.
\]
This combined with~\eqref{eq:initial-bounds} proves that
\[
k_1 \le x_{\delta} \le k_2 \le \frac{v_{k_2} - k_2e_{k_2}}{e_{k_2} - k_2p_{k_2}}.
\]
The lower bound here is still not in terms of the ratios $(v_k - ke_k)/(e_k - kp_k)$. But given the upper bound, we can search for $k \le k_2$ (by running a loop from $k_2$ to 0) such that
\begin{equation}\label{eq:finding-k-second}
\frac{v_{k-1} - (k-1)e_{k-1}}{e_{k-1} - (k-1)p_{k-1}} ~\le~ x_{\delta} ~\le~ \frac{v_k - ke_k}{e_{k} - kp_k}.
\end{equation}
Another approach is to make use of the lower bound in~\eqref{eq:initial-bounds}. Because $k_1 \le (v_{k_1} - k_1e_{k_1})/(e_{k_1} - k_1p_{k_1})$, there are two possibilities:
\begin{enumerate}
    \item $k_1 \le x_{\delta} \le (v_{k_1} - k_1e_{k_1})/(e_{k_1} - k_1p_{k_1})$;
    \item $k_1 \le (v_{k_1} - k_1e_{k_1})/(e_{k_1} - k_1p_{k_1}) < x_{\delta}$.
\end{enumerate}
In the first case, it suffices to search for $k \le k_1$ such
that~\eqref{eq:finding-k-second}. In the second case, we can search over $k_1 + 1 \leq k \le k_2$ as before.

\section{Proof of Theorem~\ref{thm:initial-maximal-ineq}}\label{appsec:proof-of-theorem-initial-maximal-ineq}
It is clear that $(S_t, \mathcal{F}_t)_{t = 1}^{n}$ with $\mathcal{F}_t = \sigma\{X_1, \ldots, X_t\}$ is a martingale because
\[
\mathbb{E}\left[S_t|\mathcal{F}_{t-1}\right] = S_{t-1} + \mathbb{E}[X_t] = S_{t-1}.
\]
Consider now the process
\[
D_t := \left(S_t - x\right)_+^2\quad\mbox{for a fixed}\quad x > 0.
\]
The function $f:y\mapsto (y - x)^2_+$ is continuous and satisfies
\[
f'(y) = \begin{cases}0,&\mbox{if }y \le x,\\2(y - x), &\mbox{if }y > x,\end{cases}\quad\mbox{and}\quad f''(y) = \begin{cases}0,&\mbox{if }y \le x,\\
2, &\mbox{if }y > x.\end{cases}
\]
Therefore, $f(\cdot)$ is a convex function. This implies by Jensen's inequality that
\[
\mathbb{E}[D_t|\mathcal{F}_{t-1}] = \mathbb{E}[f(S_t)|\mathcal{F}_{t-1}] \ge f(S_{t-1}).
\]
Hence $(D_t, \mathcal{F})_{t=1}^{n}$ is a submartingale. Doob's inequality now implies that
\begin{align*}
\mathbb{P}\left(\max_{1 \le t \le n} S_t \ge u\right) ~&\overset{(a)}{=}~ \mathbb{P}\left(\max_{1 \le t \le n}(S_t - x)^2_+ \ge (u - x)_+^2\right)\\
~&=~ \mathbb{P}\left(\max_{1 \le t\le n}D_t \ge (u - x)_+^2\right)\\
~&\overset{(b)}{\le}~ \frac{\mathbb{E}[D_{n}]}{(u - x)_+^2} ~\le~ \frac{\mathbb{E}[(S_{n} - x)_+^2]}{(u - x)_+^2}.
\end{align*}
Here equality (a) holds for every $x \le u$ and inequality (b) holds because of Doob's inequality. Because $x \le u$ is arbitrary, we get
\[
\mathbb{P}\left(\max_{1 \le t \le n} S_t \ge u\right) \le \inf_{x \le u}\frac{\mathbb{E}[(S_{n} - x)_+^2]}{(u - x)_+^2},
\]
and condition~\eqref{eq:boundedness-conditions} along with Theorem 2.1 of~\citet{bentkus2006domination} (or~\citet{pinelis2006binomial}) imply that
\[
\mathbb{P}\left(\max_{1 \le t\le n}S_t \ge u\right) ~\le~ \inf_{x\le u}\,\frac{\mathbb{E}[(\sum_{i=1}^n G_i - x)_+^2]}{(u - x)_+^2}.
\]
The definition~\eqref{eq:quantile-A-B-2nd-version} of $q(\delta; n, \mathcal{A}, B)$ readily implies
\[
\mathbb{P}\left(\max_{1 \le t\le n} S_t ~\ge~ q(\delta; n, \mathcal{A}, B)\right) \le \delta.
\]
This completes the proof of~\eqref{eq:maximal-inequality}. We now prove the sharpness. Note that the condition
\[
\mathbb{P}\left(\max_{1\le t\le n}\,S_t \ge n\tilde{q}(\delta^{1/n};A, B)\right) \le \delta\quad\mbox{for all}\quad \delta\in[0, 1],
\]
is equivalent to the existence of a function $x\mapsto H(x; A, B)$ such that
\[
\mathbb{P}\left(\max_{1\le t\le n}\,S_t \ge nu\right) \le H^n(u; A, B),\quad\mbox{for all}\quad u.
\]
(The function $\delta\mapsto \tilde{q}(\delta^{1/n}; A, B)$ is the inverse of $u\mapsto H^n(u; A, B)$.) In particular, this implies that
\[
\mathbb{P}\left(S_n \ge nu\right) \le H^n(u; A, B)\quad\mbox{for all}\quad u.
\]
Now, Lemma 4.7 of~\citet{bentkus2004hoeffding} (also see Eq. (2.8) of~\citet{hoeffding1963probability}) implies that
\begin{align*}
H^n(u; A, B) ~&\ge~ \left\{\left(1 + \frac{Bu}{A^2}\right)^{-(A^2 + Bu)/(A^2 + B^2)}\left(1 - \frac{u}{B}\right)^{-(B^2 - Bu)/(B^2 + A^2)}\right\}^n\\
~&=~ \inf_{h\ge0}\,e^{-nhu}\mathbb{E}\left[e^{h\sum_{i=1}^n G_i}\right],
\end{align*}
where $G_1, \ldots, G_n$ are independent random variables constructed through~\eqref{eq:G_distribution}. Proposition 3.5 of~\citet{pinelis2009bennett} implies that
\[
\inf_{h\ge0}\,e^{-nhu}\mathbb{E}\left[e^{h\sum_{i=1}^n G_i}\right] \ge \inf_{x \le nu}\,\frac{\mathbb{E}[(\sum_{i=1}^n G_i - x)_+^2]}{(nu - x)_+^2}.
\]
Summarizing the inequalities, we conclude
\[
\mathbb{P}\left(S_n \ge nu\right) \le \inf_{x \le nu}\,\frac{\mathbb{E}[(\sum_{i=1}^n G_i - x)_+^2]}{(nu- x)_+^2} \le \inf_{h\ge0}\,\mathbb{E}\left[e^{h\sum_{i=1}^n G_i - h(nu)}\right] \le H^n(u; A, B)\quad\forall\; u.
\]
This proves that $q(\delta; n, A, B) \le n\tilde{q}(\delta^{1/n}; A, B)$ for any valid $\tilde{q}(\cdot; A, B)$.
\section{Proof of Theorem~\ref{thm:uniform-in-n}}\label{appsec:proof-of-theorem-uniform-in-n}
The proof is based on~\eqref{eq:maximal-inequality} and a union bound. It is clear that
\begin{align*}
&\mathbb{P}\left(\exists\, t\ge 1:\,\sum_{i=1}^t X_i \ge q(\delta/h(k_t); c_t, \mathcal{A}, B)\right)\\
&\qquad=~ \mathbb{P}\left(\bigcup_{k=0}^{\infty}\left\{\exists\,\lceil\eta^k\rceil\le t\le \lfloor\eta^{k+1}\rfloor:\,\sum_{i=1}^t X_i \ge q(\delta/h(k_t); c_t, \mathcal{A}, B)\right\}\right)\\
&\qquad=~ \mathbb{P}\left(\bigcup_{k=0}^{\infty}\left\{\exists\,\lceil\eta^k\rceil \le t \le \lfloor\eta^{k+1}\rfloor:\,\sum_{i=1}^t X_i \ge q(\delta/h(k); \lfloor\eta^{k+1}\rfloor, \mathcal{A}, B)\right\}\right)\\
&\qquad\le~ \sum_{k=0}^{\infty} \mathbb{P}\left(\max_{\lceil\eta^k\rceil \le t \le \lfloor\eta^{k+1}\rfloor}\sum_{i=1}^t X_i \ge q(\delta/h(k); \lfloor\eta^{k+1}\rfloor, \mathcal{A}, B)\right)\\
&\qquad\le~ \sum_{k=0}^{\infty} \frac{\delta}{h(k)} \le \delta.
\end{align*}
\section{Proof of Theorem~\ref{thm:uniform-in-n-empirical-Bentkus}}\label{appsec:proof-of-lemma-variance-estimation}
Theorem~\ref{thm:uniform-in-n} implies that
\[
\mathbb{P}\left(\exists n\ge 1:\, S_n\ge q\left(\frac{\delta_1}{h(k_n)}; c_n, A, B\right)\right) \le \delta_1.
\]
Lemma~\ref{lem:uniform-in-n-variance} (below) proves
\[
\mathbb{P}\left(\exists n\ge1:\,A \ge \widebar{A}_n(\delta_2)\right) \le \delta_2.
\]
In particular this implies that
\[
\mathbb{P}\left(\exists n\ge1:\,A \ge \min_{1\le s\le n}\widebar{A}_s(\delta_2)\right) \le \delta_2.
\]
Combining the inequalities above with a union bound (and Lemma~\ref{lem:non-decreasing-quantile}) proves the result.
\begin{lem}\label{lem:uniform-in-n-variance}
Under the assumptions of Theorem~\ref{thm:uniform-in-n-empirical-Bentkus}, we have for any $\delta\in[0, 1],$
\begin{equation}\label{eq:uniform-in-n-variance}
\mathbb{P}\left(\exists t\ge 1:\, V_{2\lfloor t/2\rfloor} - \lfloor t/2\rfloor A^2 \le -\frac{\sqrt{\lfloor c_t/2\rfloor}(B - \underline{B})A}{\sqrt{2}}\Phi^{-1}\left(1 - \frac{2\delta}{e^2h(k_t)}\right)\right) \le \delta,
\end{equation}
where $W_i = (X_{2i} - X_{2i-1})^2/2$ and $V_t := \sum_{i=1}^{\lfloor t/2\rfloor} W_i.$
\end{lem}
\begin{proof}
Fix $x \ge 0$. Note that for any $u\ge -x$,
\begin{align*}
\mathbb{P}\left(\max_{1\le t\le n}\,\{V_{2t} - tA^2\} \le -x\right)
&= \mathbb{P}\left(\max_{1\le t\le n}\,(u - \{V_{2t} - tA^2\})_+ \ge (u + x)_+\right),\\
&{\le} \frac{\mathbb{E}[(u - \{V_{2n} - 2nA^2\})_+^2]}{(u + x)_+^2}.
\end{align*}
where the last inequality follows from the fact that $\{(u - \{V_{2t} - tA^2\}\}_{t\ge1}$ is a submartingale. Therefore,
\begin{align*}
\mathbb{P}\left(\max_{1\le t\le n}\,\{V_{2t} - tA^2\} \le -x\right) &\le \inf_{u \ge -x}\,\frac{\mathbb{E}[(u - \{V_{2n} - nA^2\})_+^2]}{(u + x)_+^2}\\
&= \inf_{u\ge -x}\,\frac{\mathbb{E}[(u + nA^2 - V_{2n})_+^2]}{(u + x)_+^2}\\
&= \inf_{u\ge nA^2 - x}\,\frac{\mathbb{E}[(u - V_{2n})^2_+]}{(u - nA^2 + x)^2}.
\end{align*}
Corollary 2.7 (Eq. (2.24)) of~\citet{pinelis2016optimal} implies that
\begin{equation}\label{eq:second-moment-best}
\inf_{u\ge nA^2 - x}\,\frac{\mathbb{E}[(u - V_{2n})^2_+]}{(u - nA^2 + x)^2} \le P_2(E_{1,n} + Z\sqrt{E_{2,n}}; nA^2 - x) = P_2(E_{1,n} + Z\sqrt{E_{2,n}}; E_{1,n} - x),
\end{equation}
where $E_{j,t} = \sum_{i=1}^{\lfloor t/2\rfloor}\mathbb{E}[W_i^j]$ for $j = 1,2$ and $Z$ stands for a standard normal distribution. Inequality~\eqref{eq:second-moment-best} is \emph{not} the best inequality to use and there is a more precise version; see Theorem 2.4(I) and Corollary 2.7 of~\citet{pinelis2016optimal}. With the more precise version, the following steps will lead to a refined upper bound on $A$; we will not pursue this direction here.

It now follows from~\citet{bentkus2008extension} that
\[
P_2(E_{1,n} + Z\sqrt{E}_{2,n}; E_{1,n} - x) \le \frac{e^2}{2}\mathbb{P}\left(Z \le -\frac{x}{\sqrt{E_{2,n}}}\right).
\]
Because $X_i\in[\underline{B}, B]$ with probability 1, $W_i \le (B - \underline{B})^2/2$ and hence
\[
E_{2,n} = \mathbb{E}\sum_{i=1}^{n} \mathbb{E}[W_i^2] \le \frac{(B - \underline{B})^2}{2}\sum_{i=1}^{n} \mathbb{E}[W_i] = (B - \underline{B})^2E_{1,n}/2 = n(B - \underline{B})^2A^2/2.
\]
This implies that
\[
\mathbb{P}\left(\max_{1\le t\le n}\{V_{2t} - tA^2\} \le -x\right) \le \frac{e^2}{2}\mathbb{P}\left(Z \le -\frac{\sqrt{2}x}{\sqrt{n}(B - \underline{B})A}\right).
\]
Equating the right hand side to $\delta$ yields
\begin{equation}\label{eq:maximal-variance-case}
\mathbb{P}\left(\max_{1\le t\le n}\{V_{2t} - tA^2\} \le - \frac{\sqrt{n}(B - \underline{B})A}{\sqrt{2}}\Phi^{-1}\left(1 - \frac{2\delta}{e^2}\right)\right) \le \delta.
\end{equation}
Because of this maximal inequality, we can apply stitching and get~\eqref{eq:uniform-in-n-variance}.
Note that
\begin{align*}
&\mathbb{P}\left(\exists t\ge 1:\,V_{2\lfloor t/2\rfloor} - \lfloor t/2\rfloor A^2 \le -\frac{\sqrt{\lfloor c_t/2\rfloor}(B - \underline{B})A}{\sqrt{2}}\Phi^{-1}\left(1 - \frac{2\delta}{e^2h(k_t)}\right)\right)\\
&= \mathbb{P}\left(\exists t \ge 2:\,V_{2\lfloor t/2\rfloor} - \lfloor t/2\rfloor A^2 \le -\frac{\sqrt{\lfloor c_t/2\rfloor}(B - \underline{B})A}{\sqrt{2}}\Phi^{-1}\left(1 - \frac{2\delta}{e^2h(k_t)}\right)\right)\\
& = \mathbb{P}\left(\bigcup_{k=0}^{\infty}\left\{\exists \lceil\eta^k\rceil \le t\le \lfloor\eta^{k+1}\rfloor:\, V_{2\lfloor t/2\rfloor} - \lfloor t/2\rfloor A^2 \le -\frac{\sqrt{\lfloor c_t/2\rfloor}(B - \underline{B})A}{\sqrt{2}}\Phi^{-1}\left(1 - \frac{2\delta}{e^2h(k_t)}\right)\right\}\right)\\
& \le \sum_{k=0}^{\infty} \mathbb{P}\left(\exists\lceil\eta^k\rceil \le t\le \lfloor\eta^{k+1}\rfloor:\, V_{2\lfloor t/2\rfloor} - \lfloor t/2\rfloor A^2 \le -\frac{\sqrt{\lfloor c_t/2\rfloor}(B - \underline{B})A}{\sqrt{2}}\Phi^{-1}\left(1 - \frac{2\delta}{e^2h(k_t)}\right)\right)\\
&\le \sum_{k=0}^{\infty} \frac{\delta}{h(k)} \le \delta,
\end{align*}
where the last inequality follows from~\eqref{eq:maximal-variance-case} applied to $\{1\le t\le \lfloor c_t/2\rfloor\}$.

Inequality~\eqref{eq:uniform-in-n-variance} yields
\[
\mathbb{P}\left(tA^2 - \frac{\sqrt{\lfloor c_t/2\rfloor}(B - \underline{B})A}{\sqrt{2}}\Phi^{-1}\left(1 - \frac{2\delta}{e^2h(k_t)}\right) - V_{2t} \le 0\quad\forall\, t\ge 1 \right) \ge 1 - \delta.
\]
Inequality
\[
tA^2 - \frac{\sqrt{\lfloor c_t/2\rfloor}(B - \underline{B})A}{\sqrt{2}}\Phi^{-1}\left(1 - \frac{2\delta}{e^2h(k_t)}\right) - V_{2t} \le 0
\]
holds for $A > 0$ if and only if
\[
A \le g_{2,t} + \sqrt{g_{2,t}^2 + g_{3,t}},
\]
where
\[
g_{2,t} = \frac{\sqrt{\lfloor c_t/2\rfloor}(B - \underline{B})A}{2\sqrt{2}t}\Phi^{-1}\left(1 - \frac{2\delta}{e^2h(k_t)}\right)\quad\mbox{and}\quad g_{3,t} = \frac{V_{2\lfloor t/2\rfloor}}{\lfloor t/2\rfloor}.
\]
Hence a rewriting of~\eqref{eq:uniform-in-n-variance} is
\[
\mathbb{P}\left(A \ge g_{2,t} + \sqrt{g_{2,t}^2 + g_{3,t}}\quad\forall\; t\ge 1\right) \ge 1 - \delta.
\]
It is clear that $g_{2,t} = O(1/\sqrt{t})$ and $\mathbb{E}[V_{2\lfloor t/2\rfloor}/\lfloor t/2\rfloor] = A^2$ and hence the upper bounds above grows like $A + O(\sqrt{\log(h(k_t))/t})$.
\end{proof}
\section{Proof of Theorem~\ref{thm:empirical-Bentkus-arbitrary-mean}}\label{appsec:proof-empirical-bentkus-arbitrary-mean}
The assumption $\mathbb{P}(L \le X_i \le U) = 1$ implies that $\mathbb{P}(L - \mu \le X_i - \mu \le U - \mu) = 1$ and hence applying Theorem~\ref{thm:uniform-in-n} with $X_i - \mu$ and its upper bound $U - \mu$ yields
\begin{equation}\label{eq:upper-tail}
\mathbb{P}\left(\exists n\ge1:\, \sum_{i=1}^n (X_i - \mu) \ge q\left(\frac{\delta_1/2}{h(k_n)};\,c_n,\,A,\,U - \mu\right)\right) \le \frac{\delta_1}{2}.
\end{equation}
Similarly applying Theorem~\ref{thm:uniform-in-n} with $\mu - X_i$ and its upper bound $\mu - L$ yields
\begin{equation}\label{eq:lower-tail}
\mathbb{P}\left(\exists n\ge1:\, \sum_{i=1}^n (\mu - X_i) \ge q\left(\frac{\delta_1/2}{h(k_n)};\,c_n,\,A,\,\mu - L\right)\right) \le \frac{\delta_1}{2}.
\end{equation}
Finally Lemma~\ref{lem:uniform-in-n-variance} implies that
\begin{equation}\label{eq:variance-bound-arbitrary-mean}
\mathbb{P}\left(\exists n\ge1:\, A \ge \widebar{A}_n^*(\delta_2; U, L)\right) \le \delta_2.
\end{equation}
Now combining inequalities~\eqref{eq:upper-tail},~\eqref{eq:lower-tail}, and~\eqref{eq:variance-bound-arbitrary-mean} yields with probability $\ge 1 - \delta_1 - \delta_2$, for all $n\ge1$
\begin{align*}
-\frac{1}{n}q\left(\frac{\delta_1/2}{h(k_n)};\,c_n,\,A,\,\mu - L\right) \le \frac{S_n}{n} - \mu \le \frac{1}{n}q\left(\frac{\delta_1/2}{h(k_n)};\,c_n,\,A,\,U - \mu\right),\;\mbox{and}\; A \le \widebar{A}_n^*(\delta_2).
\end{align*}
On this event, we get by using $U - \mu \le U - L$ and $\mu - L \le U - L$,
\[
\mu^{\low}_0 \le \mu \le \mu^{\up}_0,
\]
and then recursively using $\mu_{n-1}^{\low} \le \mu \le \mu_{n-1}^{\up}$,
\[
-\frac{1}{n}q\left(\frac{\delta_1/2}{h(k_n)};\,c_n,\,\widebar{A}_n^*(\delta_2),\,\mu_{n-1}^{\up} - L\right) ~\le~ \frac{S_n}{n} - \mu ~\le~ \frac{1}{n}q\left(\frac{\delta_1/2}{h(k_n)};\,c_n,\,\widebar{A}_n^*(\delta_2),\,U - \mu_{n-1}^{\low}\right).
\]
This proves the result.
\section{Auxiliary Results}\label{appsec:auxiliary-results}
Define $M_t, t \ge 1$ as $M_t := \sum_{i=1}^t G_i,$
with
\[
\mathbb{P}\left(G_i = -A_i^2/B\right) = \frac{B^2}{A^2_i + B^2}\quad\mbox{and}\quad\mathbb{P}\left(G_i = B\right) = \frac{A^2_i}{A^2_i + B^2}.
\]
\begin{lem}\label{lem:empirical-variance-first-step}
For any $t\ge1$ and $x\in\mathbb{R}$, the map $(A_1, \ldots, A_t)\mapsto \mathbb{E}[(M_t - x)_+^2]$ is non-decreasing.
\end{lem}
\begin{proof}
Suppose we prove that for every $y\in\mathbb{R}$,
\begin{equation}\label{eq:non-increasing-bernoulli}
A_1 \mapsto \mathbb{E}[(G_1 - y)_+^2]\mbox{ is non-decreasing},
\end{equation}
then by conditioning on $G_2, \ldots, G_t$ and taking $y = x + G_2 + \cdots + G_t$, we get for $A_1 \le A_1'$
\[
\mathbb{E}[(G_1(A_1) - y)_+^2] \le \mathbb{E}[(G_1(A_1') - y)_+^2].
\]
Now taking expectations on both sides with respect to $G_2, \ldots, G_t$ implies non-decreasingness of $A_1 \mapsto \mathbb{E}[(M_t - x)_+^2]$. This implies the result.

To prove~\eqref{eq:non-increasing-bernoulli},
\[
\mathbb{E}[(G_1 - y)_+^2] = \frac{B^2}{A^2_1 + B^2}\left(-\frac{A_1^2}{B} - y\right)_+^2 + \frac{A_1^2}{A_1^2 + B^2}(B - y)_+^2.
\]
Because $A_1 \to A_1^2/B^2$ is increasing, it suffices to show $A_1^2/B^2 \mapsto \mathbb{E}[(G_1 - y)_+^2]$ is non-decreasing with respect to $A_1^2/B^2$. Set $p = A_1^2/B^2$ and define
\[
g(p) = \frac{1}{1 + p}\left(-Bp - y\right)_+^2 + \frac{p}{1 + p}(B - y)_+^2.
\]
Differentiating with respect to $p$ yields
\begin{align*}
\frac{\partial g(p)}{\partial p} &= -\frac{(-Bp - y)_+^2}{(1 + p)^2} -\frac{2B(-Bp-y)_+}{1 + p} + \frac{(B - y)_+^2}{(1 + p)^2}\\
&= \frac{- (-Bp - y)_+^2 - 2B(1 + p)(-Bp - y)_+ + (B - y)_+^2}{(1 + p)^2}.
\end{align*}
If $y \le -Bp$ then $y + Bp < 0$ and $B - y > B(1 + p) > 0$ and hence
\[
\frac{\partial g(p)}{\partial p} = \frac{-(Bp + y)^2 + 2B(1 + p)(Bp + y) + (B - y)^2}{(1 + p)^2} = \frac{B^2 + B^2p^2 + 2B^2p}{(1 + p)^2} > 0.
\]
If $-Bp < y < B$ then $y + Bp > 0$ and $B - y > 0$ and hence
\[
\frac{\partial g(p)}{\partial p} = \frac{(B - y)^2}{(1 + p)^2} > 0.
\]
If $y > B$, then $\partial g(p)/\partial p = 0$. Hence $\partial g(p)/\partial p \ge 0$ for all $p$.
This proves~\eqref{eq:non-increasing-bernoulli}.
\end{proof}
Recall the definition of $q(\delta; t, \mathcal{A}, B)$ from~\eqref{eq:quantile-A-B-2nd-version}. In the case of equal variances, that is, $A_1 = A_2 = \ldots = A$, we write $A, q(\delta; t, A, B)$ for $\mathcal{A}, q(\delta; t, \mathcal{A}, B),$ respectively. We now prove that $A\mapsto q(\delta; t_2, A, B)$ is an non-decreasing function.
\begin{lem}\label{lem:non-decreasing-quantile}
For any $t\ge1$, the function $A\mapsto q(\delta; t, A, B)$ is an non-decreasing function.
\end{lem}
\begin{proof}
Lemma~\ref{lem:empirical-variance-first-step} proves that $A\mapsto \mathbb{E}[(M_{t} - x)_+^2]$ is non-decreasing. This implies that $I(A; u)$ is also non-decreasing in $A$, where
\[
I(A; u) ~:=~ \inf_{x\le u}\,\frac{\mathbb{E}[(M_{t} - x)_+^2]}{(u - x)_+^2}.
\]
Lemma 3.1 of~\citet{bentkus2006domination} proves that $I(A; u)$ is also non-increasing in $u$.
Fix $A_1 \le A_2$. From the definition of $\delta$,
\[
I(A_1, q(\delta; t, A_1, B)) = \delta\quad\mbox{and}\quad I(A_2, q(\delta; t, A_2, B)) = \delta.
\]
Because $I(A; u)$ is non-decreasing in $A$,
\[
I(A_2; q(\delta; t, A_2, B)) = \delta = I(A_1; q(\delta; t, A_1, B)) \le I(A_2; q(\delta; t, A_1, B))
\]
Hence $I(A_2; q(\delta; t, A_2, B)) \le I(A_2; q(\delta; t, A_1, B))$ and because $I(A; u)$ is non-increasing in $u$, we conclude that $q(\delta; t, A_1, B) \le q(\delta; t, A_2, B)$. This proves the result modulo the condition $A\mapsto \mathbb{E}[(M_{t} - x)_+^2]$ is non-decreasing.
\end{proof}
\begin{lem}\label{lem:homogenity-delta-function}
For any $\delta\in[0, 1]$, $q(\delta; t, AB, B^2) = Bq(\delta; t, A, B)$.
\end{lem}
\begin{proof}
Recall that $q(\delta; t, AB, B^2)$ is defined as the solution of
\[
\inf_{x\le u}\,\frac{\mathbb{E}[(M_t' - x)^2_+]}{(u - x)_+^2} = \delta,
\]
where $M_t'$ is defined as $M_t' = \sum_{i=1}^t G_i'$ with
\begin{align*}
\mathbb{P}\left(G_i' = -(A^2B^2)/B^2\right) &= \frac{B^4}{A^2B^2 + B^4} = \frac{B^2}{A^2 + B^2}\quad\mbox{and},\\
\mathbb{P}\left(G_i' = B^2\right) &= \frac{A^2B^2}{A^2B^2 + B^4} = \frac{A^2}{A^2 + B^2}.
\end{align*}
This implies that $G_i'\overset{d}{=}BG_i$ and hence $M_t'\overset{d}{=}BM_t$. Therefore,
\[
\mathbb{E}[(M_t' - x)_+^2] = \mathbb{E}[(BM_t - x)_+^2] = B^2\mathbb{E}[(M_t - x/B)_+^2],
\]
and
\[
\inf_{x\le u}\,\frac{\mathbb{E}[(M_t' - x)_+^2]}{(u - x)_+^2} = B^2\inf_{x\le u}\,\frac{\mathbb{E}[(M_t - x/B)_+^2]}{B^2(u/B - x/B)_+^2} = \inf_{x \le u/B}\,\frac{\mathbb{E}[(M_t - x)_+^2]}{(u/B - x)_+^2}.
\]
The right hand side above equals $\delta$, when $u = Bq(\delta; t, A, B)$ because the definition of $q(\delta; t, A, B)$ implies that
\[
\inf_{x \le q(\delta; t, A, B)}\,\frac{\mathbb{E}[(M_t - x)_+^2]}{(q(\delta; t, A, B) - x)_+^2} = \delta.
\]
This completes the proof.
\end{proof}

\section{Alternative Empirical Bentkus Confidence Sequences with Estimated Variance} \label{appsec:alternative_variance_estimation}
In Section~\ref{sec:empirical-Bentkus}, we presented one actionable version of Theorem~\ref{thm:uniform-in-n}, where we used an analytical upper bound on the variance $A^2$. In this section, we present an alternative empirical Bentkus confidence sequence that requires numerical computation. In our initial experiments, we found solving for the upper bound of $A$ in this way to be unstable. Because the proof technique here is very analogues to that of the empirical Bernstein bound in~\citet[Eq. (48)-(50)]{audibert2009exploration}, we present the alternative bound below.

Define the empirical variance as
\begin{equation}\label{eq:variance-def}\textstyle
\widehat{A}_n^2 ~:=~ n^{-1}\sum_{i=1}^n (X_i - \widebar{X}_n)^2,\quad\mbox{where}\quad \widebar{X}_n = n^{-1}\sum_{i=1}^n X_i.
\end{equation}
For any $\delta_1, \delta_2\in[0,1]$, define
\begin{equation}\label{eq:variance-upper-bound}\textstyle
\widebar{A}_n := \sup\left\{a\ge0:\,\widehat{A}_n^2 ~\ge~ a^2 - \frac{B}{n}q\left(\frac{\delta_1}{h(k_n)}; c_n, a, B\right) - \frac{1}{n^2}q^2\left(\frac{\delta_2}{2h(k_n)}; c_n, a, B\right)\right\}.
\end{equation}
Lemma~\ref{lem:variance-estimation} shows that $\widebar{A}_n$ is an over-estimate of $A$ uniformly over $n$ and yields the following actionable bound.
Recall that $S_n = \sum_{i=1}^n X_i = n\widebar{X}_n$.
\begin{thm}\label{thm:uniform-in-n-empirical-variance}
If $X_1, X_2, \ldots$ are mean-zero independent random variables satisfying
$\mathrm{Var}(X_i) = A^2$ and $\mathbb{P}(|X_i| > B) = 0$ for all $i\ge1$,
then for any $\delta_1, \delta_2 \in [0, 1]$,
\[
\mathbb{P}\left(\exists n\ge1:\,\abs{S_n} \ge q\left(\frac{\delta_2}{2h(k_n)}; c_n, \widebar{A}_n^*, B\right)\quad\mbox{or}\quad A \ge \widebar{A}^*_n(\delta_1) \right) \le \delta_1 + \delta_2,
\]
where $\widebar{A}^*_n := \min_{1\le s\le n}\widebar{A}_s$. Here $k_n$ and $c_n$ are same as those defined in Theorem~\ref{thm:uniform-in-n}.
\end{thm}
This theorem is an analogue of the empirical Bernstein inequality \citet[Eq. (5)]{mnih2008empirical}. Furthermore, the upper bound $\widebar{A}_n$ on $A$ is better than that in the Bernstein version \citet[Eq. (49)-(50)]{audibert2009exploration}; see Lemma~\ref{lem:better-than-Bernstein-variance}.  
\subsection{Proof of Theorem~\ref{thm:uniform-in-n-empirical-variance} and Comparison of Standard Deviation Estimation from Other Inequalities}
\begin{lem}\label{lem:variance-estimation}
If $X_1, X_2, \ldots$ are mean-zero independent random variables satisfying
\[
\mathrm{Var}(X_i) = A^2\quad\mbox{and}\quad \mathbb{P}(|X_i| > B) = 0,\quad\mbox{for all}\quad i\ge1,
\]
then for any $\delta\in[0,1]$
\[
\mathbb{P}\left(\exists\,t\ge1:\,\widehat{A}_t^2 ~\le~ A^2 - \frac{B}{t}q\left(\frac{\delta}{h(k_t)}; c_t, A, B\right) - \frac{1}{t^2}\left|\sum_{i=1}^t X_i\right|^2\right) \le \delta.
\]
\end{lem}
\begin{proof}
Consider the random variable $X_i^2 - \mathbb{E}[X_i^2]$. These are mean zero and are bounded in absolute value by $B^2$. Further the variance can be bounded as
\[
\mathrm{Var}(X_i^2 - \mathbb{E}[X_i^2]) = \mathbb{E}[(X_i^2 - \mathbb{E}[X_i^2])^2] \le B^2\mathbb{E}[|X_i|^2] = B^2A^2.
\]
Applying Theorem~\ref{thm:uniform-in-n} with variables $X_i^2 - \mathbb{E}[X_i^2]$ implies
\[
\mathbb{P}\left(\exists t\ge1:\,\sum_{i=1}^t -(X_i^2 - \mathbb{E}[X_i^2]) \ge q\left(\frac{\delta}{h(k_t)};c_t, AB, B^2\right)\right) \le \delta.
\]
Lemma~\ref{lem:homogenity-delta-function} proves that
\[
q\left(\frac{\delta}{h(k_t)};c_t, AB, B^2\right) = Bq\left(\frac{\delta}{h(k_t)};c_t, A, B\right).
\]
Hence we get with probability at least $1 - \delta$, simultaneously for all $t\ge1$
\begin{align*}
\sum_{i=1}^t (X_i - \widebar{X}_t)^2 ~&=~ \sum_{i=1}^t X_i^2 - \frac{1}{t}\left(\sum_{i=1}^t X_i\right)^2\\
~&\ge~ \sum_{i=1}^t \mathbb{E}[X_i^2] - Bq\left(\frac{\delta}{h(k_t)}; c_t, A, B\right) - \frac{1}{t}\left|\sum_{i=1}^t X_i\right|^2.
\end{align*}
Hence for any $\delta\in[0,1]$,
\[
\mathbb{P}\left(\exists\,t\ge1:\,t\widehat{A}_t^2 \le tA^2 - Bq\left(\frac{\delta}{h(k_t)}; c_t, A, B\right) - \frac{1}{t}\left|\sum_{i=1}^n X_i\right|^2\right) \le \delta.
\]
This completes the proof.
\end{proof}

We will now prove Theorem~\ref{thm:uniform-in-n-empirical-variance}. Theorem~\ref{thm:uniform-in-n} implies that
\begin{equation}\label{eq:implication-thm-unknown-A}
\mathbb{P}\left(\exists\, t\ge 1:\,\left|\sum_{i=1}^t X_i\right| \ge q\left(\frac{\delta_2}{2h(k_t)}; c_t, \mathcal{A}, B\right)\right) \le \delta_2,
\end{equation}
Lemma~\ref{lem:variance-estimation} implies that
\[
\mathbb{P}\left(\exists\,t\ge1:\,\widehat{A}_t^2 ~\le~ \frac{t}{t-1}A^2 - \frac{B}{t-1}q\left(\frac{\delta_1}{h(k_t)}; c_t, A, B\right) - \frac{1}{t(t-1)}\left|\sum_{i=1}^t X_i\right|^2\right) \le \delta_1.
\]
Hence with probability at least $1 - \delta_1 - \delta_2$, simultaneously for all $t\ge1,$
\begin{align*}
\left|\sum_{i=1}^t X_i\right| ~&\le~ q\left(\frac{\delta_2}{2h(k_t)}; c_t, A, B\right),\\
\widehat{A}_t^2 ~&\le~ \frac{t}{t-1}A^2 - \frac{B}{t-1}q\left(\frac{\delta_1}{h(k_t)}; c_t, A, B\right) - \frac{1}{t(t-1)}\left|\sum_{i=1}^t X_i\right|^2
\end{align*}
On this event, $A \le \widebar{A}_t$ simultaneously for all $t\ge1$ which in turn implies that $A\le \min_{1\le s\le t}\widebar{A}_s$ also holds simultaneously for all $t\ge1$. Substituting this in~\eqref{eq:implication-thm-unknown-A} (along with Lemma~\ref{lem:non-decreasing-quantile}) implies the result.
\begin{lem}\label{lem:better-than-Bernstein-variance}
Suppose $\delta\mapsto\tilde{q}(\delta^{1/n}; A, B)$ is a function such that
\begin{equation}\label{eq:power-quantile-bound}
\mathbb{P}\left(\max_{1\le t\le n}\,S_t \ge n\tilde{q}(\delta^{1/n}; A, B)\right) \le \delta,
\end{equation}
for all $\delta\in[0, 1]$ and independent random variables $X_1, \ldots, X_n$ satisfying~\eqref{eq:boundedness-conditions}. Define the (over)-estimator of $A$ as
\[
\tilde{A}_t ~:=~ \sup\left\{a\ge0:\,\widehat{A}_t^2 \ge a^2 - \frac{Bc_t}{t}\tilde{q}\left((\delta/(3h(k_t)))^{1/c_t}; a, B\right) - \frac{c_t^2}{t^2}\tilde{q}^2\left((\delta/(3h(k_t)))^{1/c_t}; a, B\right)\right\}.
\]
Then $\widebar{A}_n \le \tilde{A}_n.$
\end{lem}
\begin{proof}
We have proved in Appendix~\ref{appsec:proof-of-theorem-initial-maximal-ineq} that~\eqref{eq:power-quantile-bound} implies
\[
q\left(\delta; n, a, B\right) \le n\tilde{q}\left(\delta^{1/n}; a, B\right),
\]
for all $n, a,$ and $B$. Hence if $a$ satisfies
\[
\widehat{A}_t^2 \ge a^2 - \frac{B}{t}q\left(\frac{\delta}{3h(k_t)}; c_t, a, B\right) - \frac{1}{t^2}q^2\left(\frac{\delta}{3h(k_t)}; c_t, a, B\right),
\]
then
\[
\widehat{A}_n^2 \ge a^2 - \frac{Bc_t}{t}\tilde{q}\left((\delta/(3h(k_t)))^{1/c_t}; a, B\right) - \frac{c_t^2}{t^2}\tilde{q}^2\left((\delta/(3h(k_t)))^{1/c_t}; a, B\right),
\]
which implies the result.
\end{proof}

\end{document}